\documentclass[reqno, english]{amsart}
\usepackage{amsfonts}
\usepackage{amsthm}
\usepackage{amsmath}
\usepackage{amscd}
\usepackage[latin2]{inputenc}
\usepackage{t1enc}
\usepackage[mathscr]{eucal}
\usepackage{indentfirst}
\usepackage{graphicx}
\usepackage{graphics}
\usepackage{pict2e}
\usepackage{epic}
\usepackage{bbm}
\usepackage{bm}
\usepackage[left=1in,right=1in,top=1in,bottom=1in]{geometry}
\usepackage{epstopdf} 
\usepackage{enumitem}

\usepackage{amssymb}

\usepackage{hyperref}
\hypersetup{
	colorlinks = true,
    linkcolor = blue,
    citecolor = blue,
    urlcolor = black,
}

\theoremstyle{plain}
\newtheorem{thm}{Theorem}[section]
\newtheorem{lem}[thm]{Lemma}
\newtheorem{lemma}[thm]{Lemma}
\newtheorem{cor}[thm]{Corollary}
\newtheorem{prop}[thm]{Proposition}
\newtheorem{claim}[thm]{Claim}
\newtheorem{question}[thm]{Question}

\newtheorem{obs}[thm]{Observation}
\newtheorem{conj}[thm]{Conjecture}

\newtheorem*{claim*}{Claim}

\newcommand{\K}{\mathcal K}

\newcommand{\Q}{\mathcal Q}

\newcommand{\cB}{\mathcal{B} }
\newcommand{\cC}{\mathcal{C} }
\newcommand{\cD}{\mathcal{D} }
\newcommand{\cE}{\mathcal{E} }

\newcommand{\cK}{\mathcal{K} }

\newcommand{\cN}{\mathcal{N} }
\newcommand{\cP}{\mathcal{P} }
\newcommand{\cQ}{\mathcal{Q} }
\newcommand{\cR}{\mathcal{R} }
\newcommand{\cS}{\mathcal{S} }

\newcommand{\cW}{\mathcal{W} }

\newcommand{\cZ}{\mathcal{Z} }

\newcommand{\pr}{\mathbb{P}}

\newcommand{\E}[0]{\mathbb{E}}
\newcommand{\mbone}[0]{\mathbbm{1}}

\newcommand{\beq}[1]{\begin{equation}\label{#1}}
\newcommand{\enq}[0]{\end{equation}}

\newcommand{\mn}[0]{\medskip\noindent}
\newcommand{\nin}[0]{\noindent}

\newcommand{\sub}[0]{\subseteq}
\newcommand{\sm}[0]{\setminus}
\renewcommand{\iff}[0]{\Longleftrightarrow}
\newcommand{\ov}[0]{\overline}

\renewcommand{\dots}[0]{,\ldots,}

\newcommand{\more}[0]{~~\mbox{\raisebox{-.9ex}{$\stackrel{\textstyle{>}}{\sim}$}}~~}

\newcommand{\A}[0]{{\mathcal A}}
\newcommand{\B}[0]{{\mathcal B}}
\newcommand{\cee}[0]{{\mathcal C}}
\newcommand{\D}[0]{{\mathcal D}}
\newcommand{\eee}[0]{{\mathcal E}}

\newcommand{\m}[0]{{\mathcal M}}

\newcommand{\pee}[0]{{\mathcal P}}

\newcommand{\sss}[0]{{\mathcal S}}
\newcommand{\T}[0]{{\mathcal T}}
\newcommand{\U}[0]{{\mathcal U}}

\newcommand{\W}[0]{{\mathcal W}}

\newcommand{\ra}[0]{\rightarrow}
\newcommand{\Ra}[0]{\Rightarrow}

\newcommand{\Lra}[0]{\Leftrightarrow}

\newcommand{\0}[0]{\emptyset}

\newcommand{\C}[2]{\binom{{#1}}{{#2}}}
\newcommand{\Cc}[0]{\tbinom}
\newcommand{\ga}[0]{\alpha }
\newcommand{\gb}[0]{\beta }
\newcommand{\gc}[0]{\gamma }
\newcommand{\gd}[0]{\delta }
\newcommand{\gD}[0]{\Delta }
\newcommand{\gG}[0]{\Gamma }

\newcommand{\gl}[0]{\lambda }
\newcommand{\gL}[0]{\Lambda}
\newcommand{\go}[0]{\omega}
\newcommand{\gO}[0]{\Omega}

\newcommand{\gs}[0]{\sigma}
\newcommand{\gS}[0]{\Sigma}

\newcommand{\gz}[0]{\zeta}
\newcommand{\eps}[0]{\varepsilon }
\newcommand{\vt}[0]{\vartheta}
\newcommand{\vs}[0]{\varsigma}
\newcommand{\vr}[0]{\varrho}
\newcommand{\vp}[0]{\varphi}

\newcommand{\cond}{\mid}

\newcommand{\tdt}[0]{\tilde{\tau}}
\newcommand{\tds}[0]{\tilde{\gs}}

\newcommand{\tdm}[0]{\tilde{\mu}}

\newcommand{\gLp}[0]{\gL^p}

\newcommand{\prh}[1][]{\pr_h}

\usepackage{xcolor}

\newcommand{\sg}{\gd_1}

\newcommand{\mmm}{a}  
\newcommand{\MMM}{b}  
\newcommand{\CCC}{C}   

\newcommand{\HH}[0]{{H^*}}

\newcommand{\LLL}{L}
\newcommand{\Hs}{{H^*}}
\newcommand{\cHs}{{\cop{H}^*}}
\newcommand{\cKs}{{\cop{K}^*}}
\newcommand{\cg}{{\cop{g}}}

\newcommand{\Ks}{{K^*}}

\newcommand{\td}{\tilde}
\newcommand{\muu}{\mu_{_U}}

\newcommand{\ogd}{\ov{\gD}}

\newcommand{\lp}[0]{\pmb{(}}
\newcommand{\rp}[0]{\pmb{)}}

\newcommand{\rb}[0]{\pmb{]}}
\newcommand{\bft}[0]{\pmb{\tau}}
\newcommand{\bfs}[0]{\pmb{\sigma}}

\newcommand{\cop}[1]{\bar{#1}}
\newcommand{\inte}[1]{{\kern0pt#1}^{\mathrm{o}}}

\newcommand{\eH}{e_H}
\newcommand{\vH}{v_H}
\newcommand{\dHx}{d}
\newcommand{\ogD}{\ov{\gD}}

\newcommand{\ellll}{\ell}
\newcommand{\elll}{l}

\newcommand{\ccc}{\mathfrak{c}}

\newcommand{\aut}{{\rm aut}}
\newcommand{\Aut}{{\rm Aut}}

\title{On the $H$-space of a random graph}

\author{Quentin Dubroff}
\thanks{Department of Mathematics, Carnegie Mellon University}
\author{Jeff Kahn}
\thanks{Department of Mathematics, Rutgers University}
\email{qdubroff@andrew.cmu.edu,jkahn@math.rutgers.edu}
\address{Department of Mathematics, Carnegie Mellon University \\
Wean Hall 6113\\
Pittsburgh, PA 15213, USA}
\address{Department of Mathematics, Rutgers University \\
Hill Center for the Mathematical Sciences \\
110 Frelinghuysen Rd.\\
Piscataway, NJ 08854, USA}

\begin{document}

\begin{abstract}
The edge space $\mathcal{E}(G)$ of a graph $G$ is the vector space $\mathbb{F}_2^{E(G)}$ with members naturally identified with subgraphs of $G$, and the $H$-space is the subspace 
$\mathcal{C}_H(G) $ of $ \mathcal{E}(G)$ spanned by copies of the graph $H$. 
We are interested in when the random graph $G = G_{n,p}$ is likely to satisfy
\[\cC_H(G) = \cW_H(G),\]
where $\cW_H(G)$ takes one of four natural values, depending on the value of $\cC_H(K_n)$.
We show that for strictly $2$-balanced $H$, w.h.p.\ the above equality holds whenever every 
edge of $G$ is in a copy of $H$.
\end{abstract}

\maketitle

\section{Introduction}\label{Intro}

Our main result is Theorem~\ref{thm:Hprecise}, for which 
we first need to fill in some background.

Recall that the \emph{edge space} of a graph $G$,
denoted $\mathcal{E}(G)$, is the vector space $\mathbb{F}_2^{E(G)}$
(elements of which are regarded in the natural way as spanning subgraphs of $G$). 
The \emph{cycle space}, 
$\mathcal{C}(G)$, is the subspace generated by the (indicators of) cycles of $G$;
equivalently, the subspace of even-degree subgraphs of $G$, which we here abusively call
\emph{Eulerian} without requiring that they be connected.
(See \cite[Sec.\ 1.9]{Diestel} for an exposition.)

Following \cite{BK} we define,
the $H$-{\em space} of $G$, denoted $\cee_H(G)$, 
to be the subspace of
$\mathcal{E}(G)$ generated by the copies of the graph $H$ in $G$.
In what follows $H$ will be fixed and $G$ will almost always be large.

When $G=K_n$, its $H$-space has a natural form for all but a few
values of $n$;
the following observation, for which we set
$\mathcal{D}(G) = \{D \in \mathcal{E}(G) : |D| \equiv 0\pmod{2}\}$ 
and $|E(H)|=|H| $,
is \cite[Theorem 1.5]{BK}
(or, really, the more precise version proved in \cite[Sec.\ 8]{BK}).

\begin{thm}
\label{thm:Hspace}
For any graph $H$ with at least one edge and $n\geq |V(H)|+2$,
\begin{align}
\label{eq:Hspace}
\cee_H(K_n) =
\begin{cases}
\mathcal{C}(K_n) &\text{ if } H \text{ is Eulerian and } |H| \text{ is odd,} \\
\mathcal{C}(K_n) \cap \mathcal{D}(K_n) &\text{ if } H \text{ is Eulerian and } |H| \text{ is even,} \\
\mathcal{E}(K_n) &\text{ if } H \text{ is not Eulerian and } |H| \text{ is odd,} \\
\mathcal{D}(K_n) &\text{ if } H \text{ is not Eulerian and } |H| \text{ is even.}
\end{cases}
\end{align}
\end{thm}
\nin
Note the left-to-right containments ($\cee_H(K_n)\sub \cee(K_n)$ and so on) are obvious.

The question that will concern us here is (roughly), 
\begin{quote} 
\emph{when (i.e.\ for what $p=p(n)$) is 
the analogue of Theorem~\ref{thm:Hspace}
likely to hold in $G_{n,p}$?}
\end{quote}

\nin
The natural value of $\cee_H(G)$, here denoted $\W_H(G)$,
is what one gets by replacing
$K_n$ by $G$ in the appropriate expression on the r.h.s.\ of \eqref{eq:Hspace}
(e.g.\ if $H $ is a $\kappa$-cycle, then
$   
\mathcal{W}_H(G)$ is 
$\cee(G)$ if $\kappa$ is odd and
$\cee(G)\cap \D(G)$ if $\kappa$ is even);
so we are really asking,
\beq{mQ}
\mbox{\emph{when is $G_{n,p}$ likely to lie in
$\mathcal{T}_H := \{G:\cee_{H}(G) = \mathcal{W}_H(G)\}$?}}
\enq
(Again, $\cee_H(G) \subseteq \mathcal{W}_H(G)$ is trivial.  
Past work on subgraphs generating $\cee(G)$ has focused on cycles; see 
\cite{CNP} and its references.)
There are two easy necessary conditions:

\begin{itemize}

\item[(i)]
each edge of $G$ belonging to some member of $\W_H(G)$
lies in a copy of $H$ in $G$;

\item[(ii)]
each vertex with odd degree in some member of $\W_H(G)$
has odd degree in some copy of $H$ in $G$.
\end{itemize}

For values of $p$ for which the above questions are interesting,
it will be true that
\emph{in $G=G_{n,p}$,
w.h.p.\footnote{\emph{With high probability,} meaning with 
probability tending to 1.
Throughout this paper limits are taken as $n\ra\infty$.}
each edge lies in some member of $\W_H(G)$ and, if $H$ is non-Eulerian, 
each vertex has odd degree in
some member of $\W_H(G)$.}
We accordingly consider slightly simpler conditions, and define
\[
\Q_H= \{G:\mbox{each edge of $G$ lies in a copy of $H$}\},
\]
and
\[
\cR_H=\{G:\mbox{if $H$ is non-Eulerian, then each vertex of $G$ has odd degree in some copy of $H$}\}.
\]
(Note:  in \cite{BK} $\Q_H$ was used for what's now $\Q_H\wedge \cR_H$.)

The following statement, which was proved for triangles in \cite{DHK} and for
all other odd cycles in \cite{BK}, says that, at least when $H$ is an odd cycle (in which case (ii)
does not apply and $\W_H=\cee$), failure of $\Q_H$ is, in a precise sense,
the main obstruction to membership of $G_{n,p}$ in $\T_H$.
This is basically a ``hitting time'' statement adapted to the present situation,
in which the target property $\Q_H$ is not increasing (that is, is not preserved by 
addition of edges).  See also Conjecture~\ref{THT} for a version that becomes a hitting
time statement when we do aim for something increasing.

\begin{thm}
\label{TBK}
For any odd cycle H,
$ \max_p \, \pr(G_{n,p} \in \cQ_H \setminus \T_H) \ra 0$ 
(as $n\ra\infty$).
\end{thm}
\nin
For perspective we note that the triangle result of \cite{DHK} had been of particular 
interest, being the first open case of a
conjecture of M.\ Kahle \cite{Kahle,KahleArxiv} on the homology of the clique complex of
$G_{n,p}$, and that the argument of \cite{BK} doesn't even extend to
even cycles (where $\cee$ becomes $\cee\cap \D$).

The next item is \cite[Question 1.5]{BK}.
\begin{question}\label{Gen'l?}
Could it be that for each (fixed) H,
$\max_p \, \pr(G_{n,p} \in (\Q_H\wedge \cR_H) \setminus \T_H) \ra 0$
(as $n\ra\infty$)?
\end{question}
\nin
As observed below, this completely general suggestion turns out to be too optimistic.  
On the other hand our main result, to which we now turn,
provides a positive answer general enough to cover many (or most?) $H$'s of common interest.

We first recall a central notion.
The \emph{$2$-density} of a graph $H$ is 
\[
d_2(H)=\left\{\begin{array}{ll}
(e_H-1)/(v_H-2)&\mbox{if $v_H \geq 3$,}\\
0&\mbox{otherwise}
\end{array}\right.
\]
(where $(v_H,e_H) =(|V(H)|,|E(H)|)$); 
the \emph{maximum $2$-density} of $H$ is
\[
m_2(H)  = \max\{d_2(K): K\sub H\};
\]
and $H$ is \emph{$2$-balanced} 
if $m_2(H) = d_2(H)$, and \emph{strictly $2$-balanced} 
if
$d_2(K)<d_2(H)$ for all $K\subsetneq H$.
Throughout this paper, ``$H$ 2-balanced''
is taken to include $v_H\geq 3$.
See \cite{JLR} (from which our usage differs slightly) for a thorough account of these and related notions.

For present purposes $\cR_H$ turns out to be irrelevant for strictly 2-balanced $H$, since for such $H$
\[
\max_p \, \pr(G_{n,p} \in \cQ_H \sm \cR_H) \ra 0.
\]
We will not need this and leave the easy proof as an exercise, but mention it to explain
the disappearance of $\cR_H$ from almost all of what follows (all but
the discussion around Question~\ref{Gen'l?'}).

Here we extend Theorem~\ref{TBK} to all strictly 2-balanced graphs:

\begin{thm}
\label{thm:Hprecise}
For any strictly 2-balanced (fixed) $H$,
\beq{orig}
 \max_p \, \pr(G_{n,p} \in \cQ_H \setminus \T_H) \ra 0;
 \enq
equivalently,
\beq{equiv}
\forall \, p=p(n), ~~~\pr(G_{n,p} \in \mathcal{Q}_H\setminus \T_H) \ra 0.
\enq
\end{thm}
\nin
(The (trivial) equivalence is given by the observation that
\eqref{equiv} holds iff it holds when, for each $n$, $p=p(n)$ is a value
achieving the maximum in \eqref{orig} (and in this case the two statements
are the same).)

As mentioned earlier, the following strengthening of Theorem~\ref{thm:Hprecise} 
would extend
``hitting time'' to the present situation, in which the properties of interest
are not increasing.  (We expect the proof of Theorem~\ref{thm:Hprecise} 
can 
be strengthened to give this, but we have not done so and for now state it as 
a conjecture.)
\begin{conj}\label{THT}  
If $e_1\dots e_{\C{n}{2}}$ is a uniform permutation of $E(K_n)$ and 
$G_i $ is the graph on $V=[n]$ with edge set $\{e_1\dots e_i\}$, then w.h.p.
\[
\{i:G_i\in \Q_H\}\sub \{i:G_i\in \T_H\}.
\]
\end{conj}
\nin
(With $\T_H$ and $\Q_H$ replaced by \emph{increasing} properties, say $\T$ and $\Q$,
this becomes a \emph{hitting time} statement:  if $T=\min\{i: G_i\in \Q\}$, then w.h.p.\
$G_T\in \T$.)

\mn
\emph{Thresholds.}
Though, as for hitting times, ``thresholds'' don't quite make sense for properties that are not
increasing, it's easy to get from Theorem~\ref{thm:Hprecise} to 
a \emph{threshold-like} statement, as follows.

Given $H$, let 
\beq{gaH}
\vr =\vr (H) = \sum_{(a,b)}1/\aut(H;a,b),
\enq
where $(a,b)$ ranges over isomorphism types of \emph{ordered} edges of $H$
(i.e.\ over a set of representatives for the isomorphism classes), and $\aut(\cdot)=|\Aut(\cdot)|$,
with $\Aut(H;a,b)=\{\gc\in\Aut(H): \gc(a) = a, ~\gc(b) = b\}$.

The expected number of copies of $H$ in $G=G_{n,p}$ 
containing a given $xy\in G$ is 
\beq{gLexact}
\vr\cdot(n-2)_{\vH-2}~p^{\eH-1} = (1-O(1/n))\vr n^{\vH-2}p^{\eH-1}
\enq
\nin

\nin
(where $(a)_t=a(a-1)\cdots (a-t+1)$), and 
\begin{align}
\label{eq:p^*intro}
p^*_H = p^*_H(n) := n^{-1/m_2(H)}[(2 - 1/m_2(H))\vr^{-1}\ln n]^{1/(\eH-1)}
\end{align}
behaves like a 
\emph{sharp threshold} for $\mathcal{Q}_H$
in the sense that:

\begin{lemma}
\label{lem:Qsharpthresh}
For any fixed $H$ and $\eps>0$,
\begin{align}
\label{eq:lemQsharpthresh}
\pr(G_{n,p} \in \mathcal{Q}_{H}) \rightarrow
\begin{cases}
0 &\text{ if } n^{-2}\ll~ p < (1-\eps)p^*_H, \\
1 &\text{ if } p > (1+\eps)p^*_H.
\end{cases}
\end{align}
\end{lemma}
\noindent
We prove this routine observation in Section \ref{SEC:LemmaThresh}.
The cases in \eqref{eq:lemQsharpthresh} are called the \emph{0-statement}
and the \emph{1-statement} respectively.
With Theorem~\ref{thm:Hprecise}, Lemma~\ref{lem:Qsharpthresh}
easily implies the promised ``threshold-like'' statement:

\begin{thm}
\label{lem:Qsharpthresh'}
For any fixed, stricly 2-balanced $H$ and $\eps>0$,
\begin{align}
\label{eq:lemQsharpthresh2}
\pr(G_{n,p} \in \T_{H}) \rightarrow
\begin{cases}
1 &\mbox{if $p> (1+\eps)p^*_H(n)$,}\\
0 &\mbox{if
$\W_H\in \{\cee,\cee\cap \D\}$ and $(1-o(1))/n< p < (1-\eps)p^*_H(n)$}\\
&\mbox{or $\W_H\in \{\eee,\D\}$ and $n^{-2}\ll p< (1-\eps)p^*_H(n)$.}
\end{cases}
\end{align}
\end{thm}
\nin
Here and in \eqref{eq:lemQsharpthresh} the conditions involving $n^{-2}$ 
avoid (w.h.p.) $E(G_{n,p})=\0$ (where $\Q_H$ and $\T_H$ do hold) and
$|E(G_{n,p})|=1$ (where
$\T_H$ holds if $\W_H=\D$).

\mn

As mentioned earlier, the answer to Question~\ref{Gen'l?}
turns out to be negative, for a fairly silly 
reason.  
Suppose for example that $H$ consists of a copy of $K_5$
plus a 7-cycle attached to that copy at a vertex, and let $p =n^{-1/2}$. 
Then with probability $\gO(1)$, 
$G=G_{n,p}$ satisfies $\Q_H$ (again, there is no $\cR_H$ here) and contains exactly one copy, 
say $L$, of $K_5$, as well as a cycle containing exactly one edge of $L$, which cycle is then in 
$\cee(G)\sm \cee_H(G)$
($=\W_H(G)\sm\cee_H(G)$).

Persisting, we add a small additional condition to $\cQ_H$ and $\cR_H$,
namely
\[
\cS_H =\{G:\mbox{no two edges of $G$ are in precisely the same copies of $H$}\},
\]
and try again:
\begin{question}\label{Gen'l?'}
Could it be that for every H,
$\max_p \, \pr(G_{n,p} \in (\Q_H\wedge \cR_H\wedge \sss_H) \setminus \T_H) \ra 0$
(as $n\ra\infty$)?
\end{question}
\nin
This still seems optimistic, but it would be nice to see a counterexample.

\mn 
\textbf{\emph{Outline.}} The rest of the paper is organized as follows.
Usage notes conclude the introduction. Section~\ref{SEC2} recalls edge space
basics
and outlines the main points, Lemmas~\ref{lem:p=O(p)'}-\ref{LemCG},
for the proof of Theorem~\ref{thm:Hprecise}.  The next two sections
contain preliminaries:  
Section~\ref{SEC:Tools} reviews preexisting tools, and the longer 
Section~\ref{CandE} develops information on the frequencies of subgraphs
and related structures in $G_{n,p}$ that will be a basic resource for what comes later.
Section~\ref{SEC:LemmaThresh} then proves Lemma~\ref{lem:Qsharpthresh}
and gives the easy derivation of
Theorem~\ref{lem:Qsharpthresh'} from Theorem~\ref{thm:Hprecise}.

Of the ``main points,'' Lemma~\ref{LemCG} needs no argument here, being
immediate from \cite[Theorem~5.1]{BK} (a consequence of \cite{BMS,ST}, here recalled as
Theorem~\ref{TCGA}), and the easy proof of Lemma~\ref{lem:couplingdown} is given in 
Section~\ref{PLCD}.  The real main point, Lemma~\ref{lem:p=O(p)'},
is proved in Sections~\ref{SEC:LemmaP=O(p*)}-\ref{PLJ2}, beginning in
Section~\ref{SEC:LemmaP=O(p*)} with statements of the main ingredients 
(Lemmas~\ref{lem:Claim2} and \ref{Rlemma}) and derivation of Lemma~\ref{LemCG}
from these, and continuing with the proofs of Lemmas~\ref{lem:Claim2} and \ref{Rlemma}
themselves in Sections~\ref{PC2}
and \ref{PL7.2}-\ref{PLJ2} respectively.

Much of this structure---and even some of the prose---is borrowed from \cite{BK}.
The most important differences are in the proof of Lemma~\ref{lem:Claim2} and,
especially, Lemma~\ref{Rlemma} and its proof, which are unrelated to anything in \cite{BK}.
We will say more about these connections as they arise.

\mn \textbf{\emph{Usage.}}  
We will usually use $xy$ rather than $\{x,y\}$ for an edge (and usually write
``$xy\in G$'' for ``$xy\in E(G)$'').  As above, we use $v_G$ and $e_G$ for the numbers of vertices 
and edges of $G$.

If we are in some host graph $G$ on $V$, then
for $v\in V$ and $F \subseteq G$ we use $N_F(v) = \{x : vx \in F\}$ and $d_F(v) = |N_F(v)|$.
For disjoint $A, B \subseteq V$, $\nabla_F(A,B)$ is the set of $F$-edges joining $A$ and $B$,
$\nabla_F(A)=\nabla_F(A, V\setminus A)$ (these are the {\em cuts} of $F$) and
$\nabla_F(v)=\nabla_F(\{v\})$.  In all cases we drop the subscripts when $F=G$.
As usual $\gD_G$ (or $\gD(G)$)
is the maximum degree of $G$ and $V(F)$ is the set of vertices incident to the edge of $F$.

We use $[n]$ for $\{1\dots n\}$ (for positive integer $n$),
$\log$ for $\ln$ and
$a = (1 \pm b)c$ for $(1-b)c \leq a \leq (1+b)c$.
Asymptotic notation ($\sim$, $O(\cdot)$, $\gO(\cdot)$, $\tilde{\Theta}(\cdot)$ and so on), is standard,
with $a\ll b$ and $a\asymp b$ replacing $a=o(b)$ and $a=\Theta(b)$ when convenient.
Throughout the paper we assume $n$ is large enough to support our various assertions,
and usually pretend large numbers are integers.

Less conventionally, we often use
$\ccc$ to denote an unspecified positive constant, with different $\ccc$'s unrelated except where otherwise stated; 
so $\ccc$ is nearly the same as $\gO(1)$, but has
the advantage that it allows reference to particular $\ccc$'s.

\section{Main points for the Proof of Theorem~\ref{thm:Hprecise}}
\label{SEC2}

Before outlining the proof of Theorem~\ref{thm:Hprecise},
we recall just a little more background.

\subsection{Edge space basics}
\label{subsec:EdgeSpace}
The edge space $\mathcal{E}(G)$ of a graph $G$ (defined in this paper's second paragraph),
being an $\mathbb{F}_2$-vector space, comes equipped with a standard inner product:
$\langle J,K\rangle = \sum_{e \in E(G)} J_eK_e = |J \cap K|$, where sum and cardinality
are interpreted mod 2. (The first expression thinks of $J$ and $K$ as vectors, the second as
subgraphs of $G$.) With this, the orthogonal complement, $\mathcal{S}^\perp$, of a subspace
$\mathcal{S}$ of $\mathcal{E}(G)$ is defined as usual. Then $\mathcal{C}^\perp(G)$, 
the \emph{cut space} of $G$, consists of the (indicators of) cuts of $G$ (one of which is $\0$);
$(\cee(G)\cap \D(G))^\perp$ consists of cuts and their complements; and
$\cee_H^\perp(G)$ is the set of
subgraphs of $G$ having even intersection with every copy of $H$ in $G$.
(And of course $\D^\perp(G)=\{\0,G\}$ and $\eee^\perp(G) =\{\0\}$.)

For the proof of Theorem~\ref{thm:Hprecise} we will work with
the complementary spaces, noting to begin that the
already mentioned $\cee_H(G) \subseteq \mathcal{W}_H(G)$ is equivalent to
\[   
\mathcal{W}_H^\perp(G) \subseteq \cee^\perp_H(G), 
\]   
and equality here is the same as $G \in \T_H$.

For the last little bit of this subsection we assume $H$ is Eulerian;
so $\W_H(G)\in \{\cee(G) ,\cee\cap \D(G)\}$.
The following trivial observation will be useful at a few points.

\begin{prop}\label{prop:addcut}
Let $G$ be a graph and $L\sub G$, and suppose
$L',L''$ are (respectively) smallest and largest members of
the coset $L + \W_H^\perp(G)$.  Then
\[
\forall \, v\in V
~~~ d_{L'}(v)\leq d_{G}(v)/2 \leq
d_{L''}(v).
\]
\end{prop}
\nin
(For example if $d_{L'}(v)>d_{G}(v)/2$, then $L' + \nabla(v)$
($\in L + \W_H^\perp(G)$)
is smaller than $L'$.)

In particular, if $G \notin \T_H$, then since
$\cee_H^\perp(G) \setminus \W_H^\perp(G) \supseteq L+\W_H^\perp(G)$
for any $L \in \cee_H^\perp(G) \setminus \W_H^\perp(G)$, a smallest
$F\in\cee_H^\perp(G) \setminus \W_H^\perp(G)$ satisfies
\begin{align}
\label{dFdG}
d_F(v)\leq d_{G}(v)/2 \;\;\; \forall \, v \in V.
\end{align}

\subsection{Structure of the proof}
\label{SEC:Overview}

The case $H=K_3$ of Theorem~\ref{thm:Hprecise} was, as already mentioned, shown
in \cite{DHK}, and (as was also true in \cite{BK})
we exclude it in what follows, since it would require special treatment,
adding to the length of this already too long paper.  
We will also find it convenient to exclude $H=P_2$ (a 2-edge path);
the proof in this case is a routine random graph exercise 
and is left to the reader.

Thus for the rest of the paper we fix a strictly 2-balanced $H\neq P_2,K_3$,
and set
$p^*=p^*_H$ (see \eqref{eq:p^*intro}), $\cQ=\cQ_H$ and $\T=\T_H$; in particular our objective,
\eqref{orig}, becomes
\beq{orig'}
 \max_p \, \pr(G_{n,p} \in \cQ\sm \T) \ra 0.
 \enq
(The interest here is really in $p$ at least about $p^*$, smaller
values being handled by Lemma~\ref{lem:Qsharpthresh}; see
\eqref{prange}.)

As sometimes happens, though \eqref{orig'} should become ``more true'' as $p>p^*$ grows,
some points in the proof
run into difficulties for larger $p$, and it seems easiest to first deal with smaller $p$
and then derive the full statement from this restricted version. The next two lemmas,
the first of
which is our main point, implement this plan.

\begin{lemma}
\label{lem:p=O(p)'}
For any fixed $K$ and $p\leq  Kp^*$,
\beq{MP}
\pr(G_{n,p} \in \cQ\setminus \T)\ra 0.
\enq
\end{lemma}
\nin

\begin{lemma}
\label{lem:couplingdown}
There exists $K$ such that if $p>q:=Kp^*$, then
\[
\pr(G_{n,p}\notin \T) < \pr(G_{n,q}\notin \T) + o(1).
\]
\end{lemma}

\nin
Applying Lemmas~\ref{lem:couplingdown} and \ref{lem:p=O(p)'}, together with (the
1-statement of) Lemma~\ref{lem:Qsharpthresh}, to $p'(n):=\min\{p(n),Kp^*(n)\}$
then easily gives
Theorem~\ref{thm:Hprecise}.
(For $n$'s with $p(n)>Kp^*$, we have,
using Lemma~\ref{lem:couplingdown} for the first inequality and
Lemmas~\ref{lem:p=O(p)'} and
\ref{lem:Qsharpthresh} for the final $o(1)$,
\begin{align*}
\pr(G_{n,p}\in \cQ\sm \T)&< \pr(G_{n,p'}\not\in \T) +o(1)\\
&< \pr(G_{n,p'}\in \cQ\sm \T) +\pr(G_{n,p'}\not\in \cQ)+o(1) =o(1);
\end{align*}
and for the remaining $n$'s we have $p=p'$ and
Lemma~\ref{lem:p=O(p)'} applies directly.)

\medskip
The following device will play a central role in the proofs of both of these lemmas
(so in most of the paper).
For the remainder of our discussion we fix some rule that associates with each finite
graph $G$ a subgraph $F(G)$ satisfying
\begin{align}
\label{eq:f}
F(G) =
\begin{cases}
\0 &\text{if } G \in \T, \\
\text{some smallest element of } \cee_{H}^\perp(G)
\setminus \W^\perp_H(G) &\text{if } G \notin \T.
\end{cases}
\end{align}

\nin
We will use this only with $G=G_{n,p}$, so set $F(G_{n,p})=F$ throughout.
A crucial point is that $G$ determines $F$ (for ``crucial'' see the paragraph preceding
Proposition~\ref{prop:Frightsize}).
That $F$ is a
minimizer will be used only to say that it is small and, for Eulerian $H$, has small
degrees, as promised by \eqref{dFdG}.
With $F$ thus defined we may replace the event $\{G_{n,p}\notin \T\} $
by the more convenient $\{F \neq \0 \}$, which in particular allows us
to tailor our treatment to the size of a hypothetical $F$.

This brings us to the crucial starting point for the proof of Theorem~\ref{thm:Hprecise}, 
the fact that we really only need to worry about quite small $F$'s:
\begin{lemma}
\label{LemCG}
For fixed $c>0$ and $p\gg n^{-1/m_2(H)}$,
$    
\,\,\,\pr(|F|>c n^2p)\ra 0.
$    
\end{lemma}
\nin
This is immediate from the next result, which is Theorem~5.1 of \cite{BK} 
(with proof given in \cite[Sec.\ 4.8]{Baron}).
\begin{thm}\label{TCGA}
For any fixed $H$ and $G=G_{n,p}$ the following is true.
For any $\eps>0$ there is a C such that if $p>C n^{-1/m_2(H)}$ then w.h.p.:
for each $X\in \cee_H^\perp(G)$ there is a $Y\in \W_H^\perp(G)$ with
$|X\Delta Y|< \eps n^2p$;
in particular, if $\cee_H(G)\neq \W_H(G)$, then
\[
\min\{|X|:X\in \cee^\perp_H(G)\sm \W^\perp_H(G)\}< \eps n^2p.
\]
\end{thm}
\nin
While still requiring some effort, the proof of Theorem~\ref{TCGA}
is mainly based on the ``container'' machinery of
\cite{BMS,ST}; so this first step is itself something substantial.  (But it is not a surprise;
e.g.\ \cite{DHK}, which predates \cite{BMS,ST}, begins with something like 
Theorem~\ref{TCGA}---Theorem~8.34 of \cite{JLR}---which even before \cite{DHK} seems
to have been recognized as a natural starting point.)

Thus the real problem in proving Lemma~\ref{lem:p=O(p)'},
and the most interesting part of the whole business,
is dealing with (nonempty) $F$'s that are small relative to
$G$.

\section{Tools}
\label{SEC:Tools}

\subsection{Deviation and correlation}
\label{Deviation and correlation}
Set
\beq{eq:varphidef}
\varphi(x) = (1+x)\log(1+x)-x
\enq
for $ x > -1$ and (for continuity) $ \varphi(-1)=1$.
We use ``Chernoff's Inequality'' in the following form; see for example
\cite[Thm. 2.1]{JLR}.
\begin{thm}
\label{thm:Chernoff}
If $X \sim \mathrm{Bin}(n,p)$ and $\mu = \mathbb{E}[X] = np$, then for $t \geq 0$,
\begin{align}
\label{eq:ChernoffUpper}
\pr(X \geq \mu + t) &\leq \exp\left[-\mu\varphi(t/\mu)\right]
\leq \exp\left[-t^2/(2(\mu+t/3))\right], \\
\label{eq:ChernoffLower}
\pr(X \leq \mu - t) &\leq \exp[-\mu\varphi(-t/\mu)]
\leq \exp[-t^2/(2\mu)].
\end{align}
\end{thm}

\nin
For larger deviations the following consequence of the finer bound in \eqref{eq:ChernoffUpper}
will be convenient.
\begin{thm}
\label{thm:Chernoff'}
For $X\sim B(n,p)$ and any $K$, with $\mu=\mathbb{E}[X]=np$,
\begin{eqnarray*}
\pr(X > K\mu) < \exp[-K\mu \log (K/e)].
\end{eqnarray*}
\end{thm}
\nin (Of course this is only helpful if $K>e$.)

\medskip
We will make substantial use of
the following fundamental lower tail bound of
Svante Janson (\cite{Janson} or
\cite[Theorem 2.14]{JLR}), for which
we need a little notation.
Suppose $A_1\dots A_m$ are subsets of the
finite set $\gG$. Let $\Gamma_p$ be the random subset of $\Gamma$
gotten by including each $x$ ($ \in \Gamma$) with probability $p$,
these choices made independently.
For $j\in [m]$, let $I_j$ be the indicator of the event
$\{\gG_p\supseteq A_j\}$, and set $X=\sum I_j$,
$\mu = \mathbb{E} X =\sum_j\mathbb{E} I_j$
and
\beq{Delta}
\mbox{$\ov{\gD} = \sum\sum\{\mathbb{E} I_iI_j: A_i\cap A_j\neq\0  \}.$}
\enq
(Note this includes the diagonal terms.)

\begin{thm}\label{TJanson}
With notation as above, for any $t\in [0,\mu]$,
\[\pr(X\leq \mu -t) \leq \exp[-\varphi(-t/\mu)\mu^2/\ov{\gD}] \leq \exp[-t^2/(2\ov{\gD})].\]
\end{thm}

\nin

The following generalization of Theorem~\ref{TJanson}, from \cite{DK},
is the basis for the crucial Lemma~\ref{J1} in Section~\ref{PLJ2}.
(The generalization owed some inspiration to the observation of Riordan and 
Warnke \cite{RW} that 
Theorem~\ref{TJanson} is valid for general increasing events.)

Consider some product probability measure on $2^\gG$, and
suppose $B_{ij} \sub 2^\gG$ are increasing and
$B_i=\cup_jB_{ij}$.
Write $(i,j)\sim (k,l)$ if $B_{ij}$ and $B_{kl}$
are dependent.
(Note that, unlike \cite{RW},
we take $(i,j)\sim (i,j)$.)
Let $I_{ij}$ and $I_i$ be
the indicators of $B_{ij}$ and $B_i$ and
set $X=\sum I_i$,
\[\mbox{$\mu = \sum_{i,j}\E I_{ij}$,}\]
\[\mbox{$\ov{\gD} = \sum\sum\{\E I_{ij}I_{kl}:
(i,j)\sim (k,l)\}~~~$} \]
and
\beq{gamma'}
\mbox{$\gc =\sum_i\sum_{\{j,k\}}\E I_{ij}I_{ik}$,}
\enq
with the inner sum over (unordered) pairs with $j\neq k$.

Specializing the next statement to
the case when there is just one $j$ for each $i$
yields the above-mentioned Riordan-Warnke extension of Janson 
to increasing events.

\begin{thm}\label{JDK}
With notation as above, for any $t\in [\gc,\mu]$,
\begin{eqnarray}\label{JbdA}
\pr(X\leq \mu -t)
&\leq &\exp[-(t-\gc)^2/(2\ov{\gD})].
\end{eqnarray}
\end{thm}

\nin
\emph{Remark.}  As observed in \cite{DK}, Theorem~\ref{JDK} is
useful when 
$\pr(B_i)\approx \sum_j\pr(B_{ij})$; that is,
when the probability of seeing at least two $B_{ij}$'s
for a given $i$ is small relative to the probability
of seeing just one.

The next result is \cite[Lemma 2.46]{JLR} (originally \cite[Lemma 2]{Janson}).
\begin{lemma}\label{JUB}
For events $A_1\dots A_n$
in a probability space, and $\mu=\sum\pr(A_i)$,
\begin{align*}
\pr(\mbox{some $\mu+t$
independent $A_i$'s occur})
&\leq
\exp\left[-\mu\varphi(t/\mu)\right]\\
&\leq \exp\left[-t^2/(2(\mu+t/3))\right].
\end{align*}
\end{lemma}
\nin
Note the bound here is the same as the one in \eqref{eq:ChernoffUpper}, which is
thus contained in Lemma~\ref{JUB}.
(Strictly speaking, \cite{Janson} and \cite{JLR} state Lemma~\ref{JUB} only in setting
of Theorem~\ref{TJanson}, but their proofs are valid for the version here.)

\mn

We will also find some use for Harris' Inequality (essentially \cite{Harris}),
which we state just for product measures on Boolean algebras.
(Recall that $\A \subseteq \Omega=2^{[n]}$ 
is \emph{increasing} if $B\subseteq A\in \A \Ra B\in \A$ and \emph{decreasing} if
its complement is increasing.)
\begin{thm}
\label{thm:Harris}
If $A,B\sub \gO$ are either both increasing or both decreasing, then
\[
\pr(A \cap B) \geq \pr(A)\pr(B).
\]
If one is increasing and the other decreasing then the inequality is reversed.
\end{thm}

\subsection{Density generics}

{\em From now on we use
$G$ for $G_{n,p}$ and $V$ for $[n]=V(G)$.}
%
%
Theorems~\ref{thm:Chernoff} and \ref{thm:Chernoff'}
easily imply the next two standardish propositions, whose proofs we omit.

\begin{prop}
\label{prop:routine}
For $p \gg n^{-1}\log n$, w.h.p.
\beq{generic1}
\mbox{$|G| \sim n^2p/2~~$ and $~~d(v) \sim np \;\, \forall \ v \in V$.}
\enq
\end{prop}
\nin
(Of course the second conclusion implies the first, which just needs $p\gg n^{-2}$.)

\begin{prop}\label{density}
For any $\eps>0$ there is a $K$
such that w.h.p.\ for all disjoint $S,T\sub V$ with
$|S|,|T|>Kp^{-1}\log n$
\[
\mbox{$|\nabla_G(S,T)| =(1\pm \eps) |S||T|p$}
\]
and
\[
\mbox{$|G[S]| =(1\pm \eps) \Cc{|S|}{2}p$.}
\]
\end{prop}
\nin

\begin{prop}
\label{cpts}
For fixed $\eps >0$ and $p \gg 1/n$, w.h.p.:
if $H \sub G$ satisfies
\beq{deltaF}
d_H(v) > (1-\eps)np/2 ~~~\forall \; v \in V,
\enq
then no component of $H$ has size less than $(1-2\eps)n/2$.
\end{prop}

\begin{proof}
For a given $W\sub V$ of size $w< (1-2\eps)n/2$, let $\chi =|G[W]|$.
Then $\mu:=\mathbb{E}\chi = \binom{w}{2}p< w^2p/2$, while if $W$ is a component of an $H$
satisfying \eqref{deltaF} then
\[
\chi \geq |H[W]| > w(1-\eps)np/4 > \tfrac{(1-\eps)n}{2w}\mu =:K\mu.
\]
But (since $K>(1-\eps)/(1-2\eps)=1+\gO(1)$)
Theorems~\ref{thm:Chernoff} and \ref{thm:Chernoff'} give
\[
\gc_w:=\pr(\chi >K\mu) <\left\{\begin{array}{ll}
\exp[-\gO(\mu)]&\mbox{if $K<e^2$ (say),}\\
\exp[-K\mu\log (K/e)]&\mbox{otherwise.}
\end{array}\right.
\]
Thus, with sums over $w\in (0,(1-2\eps)n/2)$,
the probability that some $H$ as in the lemma admits a component of size less
than $(1-2\eps)n/2$ is less than
\[
\sum\Cc{n}{w} \gc_w < \sum\exp[w\log(en/w)]\gc_w,
\]
which for $p\gg 1/n$ is easily seen to be $o(1)$.
\end{proof}

Finally, the following observation will be helpful in the second part of 
Section~\ref{SEC:LemmaThresh}. 
\begin{claim}\label{evencyc}
For any $\eps > 0$,

\nin
(a) if $p> (1+\eps)\log n/n$, then w.h.p.\ every edge of $G$ lies in an even cycle;
    
\nin
(b) if $p > (1 + \eps)/n$, then w.h.p.\ $G$ contains an even cycle of length $\gO(n)$;
    
\nin    
(c) if $p > (1 - o(1))/n$ then w.h.p.\ $G$ contains an even cycle of length $\go(1)$.

\end{claim}
\begin{proof}
(a) This follows from \cite[Thm.~1.5]{BFF}, which says that, for $p$ as in (a) 
(or a little smaller),
$G_{n,p}$ is w.h.p.\ \emph{Hamilton connected} 
(i.e.\ contains a Hamilton path between any two vertices), applied to $G$ itself if $n$ is even,
and otherwise to any $G-v$.

\nin
(b) 
This is a minor consequence of \cite[Remark 1.2]{AKL}.

\nin
(c)  This follows from Poisson approximation for subgraph counts, as in 
\cite[Theorem 3.19 and Remark 3.20]{JLR}, which say that for any fixed $\ell$, 
the probability that $G$ contains no even cycle of length between (say) $\ell$ and $\ell^3$
is $O(\ell^{-1/2})$.
\end{proof}

\subsection{Coupling}
\label{subsec:Coupling}

A central role in the proofs of
Lemmas~\ref{lem:p=O(p)'} and \ref{lem:couplingdown}
is played by the usual coupling of $G:=G_{n,p}$ and $G_0:=G_{n,q}$,
where $p$ will always be the value we're really interested in and $q<p$ will depend
on what we're trying to do.
A standard description:

Let $\gl_e$, $e\in E(K_n)$, be chosen uniformly and independently from $[0,1]$ and set
\[
G = \{e:\gl_e<p\}, ~~~ G_0 = \{e:\gl_e<q\}.
\]
In particular $G_0\sub G$.
Probabilities in the proofs of Lemmas~\ref{lem:p=O(p)'} and \ref{lem:couplingdown} refer
to the joint distribution of $G$ and $G_0$.

We will get most of our leverage from two alternate ways of viewing the choice of
the pair $(G,G_0)$:

\begin{itemize}
\item[(A)]  Choose $G$ first; thus we choose $G$ ($=G_{n,p}$) in the usual way and let $G_0$
be the (``($q/p$)-random") subset of $G$ gotten by retaining edges of $G$ with probability
$q/p$, these choices made independently (a.k.a.\ {\em percolation on} $G$).

\item[(B)] Choose $G_0$ first; that is, we choose $G_0$ ($=G_{n,q}$) in the usual way,
define $p'$ by $(1-q)(1-p')=1-p$, and let $G$ be the random superset of $G_0$ gotten by
adding each edge of $\overline{G}_0$ to $G_0$ with probability $p'$, these choices again
made independently.
\end{itemize}

\nin
As in \cite{BK} we refer to these as ``coupling down'' and ``coupling up'' (respectively).

\medskip
The proof of
Lemma~\ref{lem:couplingdown} is based naturally (or inevitably) on the viewpoint in (A);
namely, we show that (with $p,q$ as in the lemma) if $G=G_{n,p}$ is ``bad''
(meaning $G\not\in \T$) then
the coupled $G_0=G_{n,q}$ is likely to be bad as well.
For the proof of
Lemma~\ref{lem:p=O(p)'}, viewpoint (B) is the primary mover, though the role of
(A) is also crucial.

With reference to the setup introduced at \eqref{eq:f},
when working with $G=G_{n,p}$ and $G_0=G_{n,q}$ as above, we set
$F_0=G_0 \cap F$
(a $(q/p)$-random subset of $F$.  (Note this has nothing to do with $F(G_0)$,
which will play no role here.)
Then automatically
\beq{F0eperp}
F_0 \in \cee^\perp_{H}(G_0),
\enq
since $F_0\cap K=F\cap K$ for any copy $K$ of $H$ in $G_0$.

We will want to say that certain features of $(G,F)$ are reflected in
$(G_0,F_0)$.
A simple but crucial point here is that there is no summing (of probabilities) over possible
$F$'s, since there is just one $F$ for each $G$.
The following proposition will be sufficient for our purposes.

\begin{prop}
\label{prop:Frightsize}
With the above setup,
for any $p$, $q$ and $g=g(n)=\go(1)$, w.h.p.
\[
|F_0|\sim |F|q/p ~~\mbox{if} ~~ |F| > gp/q
\]
and
\[
d_{F_0}(v) \left\{\begin{array}{ll}
\sim d_F(v)q/p&\mbox{$\forall v$ with $d_F(v) > (g \log n )p/q$,}\\
<3g\log n &\mbox{$\forall v$ with $d_F(v) \leq (g \log n )p/q$.}
\end{array}\right.
\]
\end{prop}

\nin
As mentioned following \cite[Prop.~3.15]{BK}, the proof of that result
gives Proposition~\ref{prop:Frightsize}
for any rule that specifies a particular subgraph (in place of $F$)
for each graph; so we will not repeat the proof here.

\section{Copies and extensions}\label{CandE}

As usual (and as above) a \emph{copy} of a graph $K$ in a graph $G$,
here often denoted $\cop{K}$, 
is a subgraph of $G$ isomorphic to $K$.   
Alternatively it is an equivalence class of isomorphisms from $K$ into $G$, 
where isomorphisms $\psi$ and $\vp$ are equivalent if $\vp =\psi\circ \gc$ 
for some $\gc\in \Aut(K)$.

We extend this to allow restrictions on $\vp$; e.g.\
if $s\in V(K)$ and $g\in E(K)$, a \emph{copy}, $(\cop{K};\cop{s}, \cop{g})$, 
of $(K;s,g)$ in $G$ is
an equivalence class of isomorphisms $\vp$ from $K$ into $V$ 
with $\vp(s) = \cop{s}$ and $\vp(g) = \cop{g}$;
equivalence of $\psi$ and $\vp$ now meaning $\vp =\psi\circ \gc$ 
for some $\gc\in \Aut(K;s,g) := \{\phi\in \Aut(K): \phi(s)=s,~\phi(g)=g\}$
(where $\phi(g)=g$ means $\phi$ fixes $g$ as a set).

For sets $W\sub Z$, we take 
$\lp W,Z\rp=\{Y:W\subsetneq Y\subsetneq Z\}$ 
and define $\lp W,Z\rb$ and so on similarly. 
(We use bold to distinguish the brackets here from those
in the next paragraph, though which is meant should always be clear from context.)

For a graph $F$ and $W\subsetneq Z \subseteq V(F)$,
we use $[W,Z]$ (with $F$ understood and suppressed)
for the graph on $Z$ with edge set $E(F[Z])\sm E(F[W])$
(where $F[X]$ is the induced subgraph on $X$).
We also use $[W,F]$ for $[W,V(F)]$.
The \emph{internal} vertices of $[W,Z]$ are those in $Z\sm W$.  
We use $e(W,Z)$ and $v(W,Z) $  for the numbers of edges and internal vertices of $[W,Z]$.

The \emph{density} of $[W,Z]$ is 
\[
d(W,Z) =d_F(W,Z) =e(W,Z)/v(W,Z);
\] 
for example,
for $T\in\lp W,Z\rp$,
\beq{dSTSU}
d(W,T)\geq d(W,Z) \Lra d(T,Z)\leq d(W,Z) 
\enq
(and similarly with the inequalities reversed).
As usual, $[W,Z]$ is \emph{balanced} 
if 
\[
d(W,Y)\leq d(W,Z) \,\,\,\,\forall Y\in\lp W,Z \rp,
\]
and \emph{strictly balanced} if these inequalities are strict. 
Notice that, by \eqref{dSTSU},
\beq{dYZWZ}
d(Y,Z)\ge d(W,Z)  \,\,\forall Y\in \lp W,Z\rp
\enq
if $[W,Z]$ is balanced, and these inequalities are strict
if the balance is.  
For example (and for later use), strict 2-balance of a graph $K$ is equivalent to 
\beq{S2Bequiv}
\mbox{$[xy,K]$ is strictly balanced for all $xy\in K$.}
\enq
When $|Z\sm W|\geq 2$ we use 
\beq{ab}
\mbox{$\mmm(W,Z)= \max\{d(W,Y):Y\in \lp W,Z\rp \} \,\,\,$ and $\,\,\,
\MMM(W,Z)= \min\{d(Y,Z):Y\in \lp W,Z\rp\}.$}
\enq

For $F,W,Z$ as above with $(v_1\dots v_s)$ an ordering of $W$;
$G$ a graph; and $X=(x_1\dots x_s)$ a sequence of distinct vertices of $G$,
a $[W,Z]$-\emph{extension} (with $F$ again understood) \emph{on} $X$
(\emph{in} $G$, but $G$ will usually be understood)
is an isomorphism $\psi$ of $[W,Z]$ into $G$ extending the map 
$v_i\mapsto x_i$, $i\in [s]$.
Alternately, with $Z\sm W=\{v_{s+1}\dots v_{s+t}\}$, one may think of an extension as
a sequence $(x_{s+1}\dots x_{s+t})$ of vertices of $G$, distinct from each other and 
from $x_1\dots x_s$, such that  $v_iv_j\in [W,Z] \Ra x_ix_j\in G$.
The ordering of $W$ will be relevant just once (see \eqref{deltabd}) and
elsewhere will be left implicit (and, of course, fixed within any particular discussion).

Extensions $\psi$ and $\vp$ of $X$ are \emph{equivalent} 
if $\vp =\psi\circ \gc$ for some $\gc\in \Aut[W,Z]$,
the set of automorphisms of $[W,Z]$ that fix $W$ pointwise.
A \emph{copy of $[W,Z]$ on $X$} is an equivalence class of extensions
(so an ``unlabeled'' extension, but note we continue to label $X$).

We use 
$\tau_{[W,Z]}(G, X)$ for the number of 
copies of $[W,Z]$ on $X$ in $G$, 
and $\gs_{[W,Z]}(G,X)$ for the maximum size of a family of 
internally disjoint copies.
We will usually use $\tau$ for $\tau_{[W,Z]}$ if $W$ and $Z$ are clear and 
$\tau^p(X)$ for $\tau(G_{n,p},X)$ (there will never be any doubt about $n$), and similarly for $\gs$.
As earlier, we also use  $\tau_{[W,F]}$ if $Z=V(F)$.
Use of $X$ in the argument of any $\tau$ or $\gs$ promises that $|X|=|W|$.

We  use $\tdt$ for the labeled counterpart of $\tau$, which 
counts \emph{extensions} rather than \emph{copies}.  
(We won't need $\tds$.)
There is, of course, an easy correspondence, namely
\[
\tau_{[W,Z]} =  \tdt_{[W,Z]}/\aut[W,Z]
\]
(where, again, $\aut(\cdot)=|\Aut(\cdot)|$).
We will often be interested in
\[
\tdm_p(W,Z) := n^{v(W,Z)}p^{e(W,Z)}
=(1+O(1/n))\E[\tdt^p_{[W,Z]}(X)]
\]
and its unlabeled counterpart
\[
\mu_p(W,Z) := n^{v(W,Z)}p^{e(W,Z)}/\aut[W,Z]
=(1+O(1/n))\E[\tau^p_{[W,Z]}(X)],
\]
and note for future use that 
\beq{mutilde}
\tdm_p(W,Z)\geq \E[\tdt^p_{[W,Z]}(X)]
\enq
(and similarly for $\mu_p$).
These clean versions of $\E$
will be convenient and the difference will never matter.

\mn

Recall (see Section~\ref{SEC:Overview}) that
we have fixed the strictly 2-balanced $H\neq P_2,K_3$,
and $p^*$ is as in \eqref{eq:p^*intro}.
A first easy observation (routine verification omitted) is:
\beq{Otriv}
\mbox{$H$ is connected with no vertices of degree 1, and $d_2(H)>1$.}
\enq
Since $p^* =\td{\Theta}( n^{-1/d_2(H)})$
it follows that
\beq{q}
n^{-1+\ccc}<p^*< n^{-\ccc}.
\enq
\nin
(Recall we use $\ccc$ for an unspecified positive constant.)

\begin{prop}\label{LS2Bd}
For any $x\in V(H)$, $[x, H]$ is strictly balanced with 
$d(x, H)< d_2(H)$.
\end{prop}
\begin{proof}
Notice that $d_2(H)>1$ (see \eqref{Otriv}) may be rewritten
\beq{ev}
e_H - v_H + 1 > 0.
\enq
This implies the proposition's second assertion (that is, $d(x,H) < d_2(H)$), since it gives
\[
\frac{e_H-1}{v_H-2} - \frac{e_H}{v_H - 1} = \frac{e_H - v_H + 1}{(v_H-1)(v_H - 2)} >0.
\]

We turn to the strict balance of $[x,V(H)]$, for which we prove the equivalent assertion that
\beq{strict}
\frac{e_H}{v_H - 1} - \frac{e_F}{v_F - 1} = \frac{e_H v_F - e_F v_H - e_H + e_F}{(v_H - 1)(v_F-1)} > 0
\enq
for each $F \subsetneq H$ with $\{x\}\subsetneq V(F)$.
If $v_F = 2$ (so $e_F \leq 1$), then \eqref{strict} follows from \eqref{ev}, since
\[
e_H v_F - e_F v_H - e_H + e_F \geq e_H - v_H + 1.
\]
For $v_F > 2$, we show
\beq{side}
e_H - v_H - e_F +v_F \geq 0,
\enq
which implies \eqref{strict} via
\[
0 < \frac{e_H -1}{v_H - 2} - \frac{e_F -1}{v_F - 2} 
= \frac{(e_H v_F - e_F v_H - e_H + e_F) - (e_H - v_H - e_F +v_F)}{(v_H - 2)(v_F-2)}.
\]
To see that \eqref{side} holds, we just observe that its failure implies the contradiction
\[
\frac{e_H -1}{v_H - 2} = \frac{e_H -v_H + 1}{v_H - 2} + 1
< \frac{e_F -v_F + 1}{v_F - 2} + 1
 = \frac{e_F -1}{v_F - 2}.
 \]
    
\end{proof}

For any graph $K$ we use
\[
\mbox{$\gL_K^p = n^{v_K -2}p^{e_K-1}~$ and $~ \Psi_K^p=n^2p\gL_K = n^{v_K}p^{e_K}$.}
\]
Thus 
$\gL_K^p$ is on the order of the expected number of 
copies of $K$ on a given edge of $G_{n,p}$, and similarly 
for $\Psi_K^p$ relative to the total number of copies.

We will use $\gL^p$ (unsubscripted)
for
the expected number of copies of $H$ on an edge of $G_{n,p}$; so
as observed in \eqref{gLexact} (with $\vr=\vr(H)$ as in \eqref{gaH}),
\[
\gL^p=\vr\cdot(n-2)_{\vH-2}~p^{\eH-1} = (1-O(1/n))\vr \gL_H^p.
\]
In particular we abbreviate 
\beq{gLup*}
\gL^{p^*} =\gL^* ~ \sim ~ (2-1/m_2(H))\log n ~\sim ~\log[\Cc{n}{2}p^*]
\enq
(see the definition of $p^*$ in \eqref{eq:p^*intro}).

If $p \asymp p^*$, then 
\beq{gLH}
\gL_H^p\asymp  \gL^p\asymp \log n,
\enq
and strict $2$-balance of $H$ implies that for any $K\subsetneq H$ with $e_K\geq 2$,
\beq{gLKbig}
\gL_K^p > n^{\ccc}
\enq
(since
$
\gL_K^p \asymp
n^{\left[\frac{v_K-2}{e_K-1} - \frac{v_H-2}{e_H-1} \right] (e_K-1)}$).
The following related observation will also be needed in a few places.

\begin{prop}\label{balcons}
There are $\vt, \eta >0$ such that for any $K \subsetneq H$ with $v_K \geq 3$, 
\beq{nzeta}
\gL_H^p/\gL^p_K  
\,\,\,(=n^{v_H - v_K}p^{e_H - e_K} ) \,\,< 
\begin{cases}
n^{-\eta} & \text{if } p < n^\vt p^*,\\
\gL_H^p n^{-\eta} p^*/p & \mbox{for any $p\geq p^*$.}
\end{cases}
\enq
\end{prop}

\begin{proof}
We may assume $p \geq p^*$ since the l.h.s.\ of \eqref{nzeta} is increasing
in $p$. Then for the first bound, using \eqref{gLH} and taking $\ccc$ as in \eqref{gLKbig}, 
we find that $\gL_H^p/\gL^p_K$ is at most
\[
\td{O}((p/p^*)^{e_H-1}) n^{-\ccc}=\td{O} (n^{\vt(e_H-1)-\ccc}),
\]
so the bound holds if we choose $\vt,\eta$ with $\eta < \ccc - \vt(e_H-1)$.

For the second bound in \eqref{nzeta}
we take $\eta$ less than the minimum of the $\ccc$'s in \eqref{q} and \eqref{gLKbig}. If
$e_K \geq 2$ then 
\[
\gL_H^p/\gL^p_K
= (\gL_H^p/\gL^*_K ) \cdot (p^*/p)^{e_K - 1} \leq ( \gL_H^p/\gL^*_K )\cdot  p^*/p
\leq \gL_H^p n^{-\eta}p^*/p
\]
(where the last inequality uses \eqref{gLKbig}); while if $e_K \leq 1$ then, recalling that $v_K \geq 3$ 
and using \eqref{q}, we have
$\gL_K^p \geq n^{3-2}p^{1-1}=n > n^{\eta}p/p^*$, which is (equivalent to) the desired bound.
\end{proof}

We will be particularly interested in \emph{bridges}.
An $xy$-\emph{bridge}---or a \emph{bridge on $xy$}---in a graph $G$ is 
$K-xy$, with $K$ a copy of $H$ in $G\cup\{xy\}$ containing $xy$ 
(note the order of $x$ and $y$ doesn't matter),
and a \emph{bridge} is an $xy$-bridge for some $x,y$.

For an $xy$-bridge $K$, we use $\inte{V}(K)$
for the set of {\em internal} vertices of $K$, meaning, as before, those not equal to $x,y$.
We use 
$\bft_G(xy)$ for the number of $xy$-bridges in $G$ and
$\bfs_G(xy)$ for
the maximum size of a set of 
internally disjoint $xy$-bridges.  
The notation (bold rather than new letters) is chosen to gesture to the relation
to $\tau$ and $\gs$; e.g.
$
\bft_G(xy) =
\sum_{(a,b)}\tau_{[(a,b), F]}(xy),
$
where $(a,b)$ ranges over isomorphism types of ordered edges of $H$
(cf.\ $\vr$ in \eqref{gaH}).
We again use $\bft^p$ and $\bfs^p$ in place of $\bft_G$ and $\bfs_G$ when
$G=G_{n,p}$,
and observe that 
\beq{Etau}
\E \bft^p(xy) =\gL^p.
\enq

For $F=F(G)$ as in \eqref{eq:f},
the following simple observation will be useful in a couple places.
\beq{xyF}
xy\in F ~~ \Longrightarrow ~~ |F|\geq \bfs_G(xy)+1.
\enq
\begin{proof}
Since $F $ lies in $\cee^\perp_H(G)$, it must contain a second edge of each copy of $H$ on $xy$, 
and there is a set of $\bfs(xy)$ such copies that
share no edges other than $xy$.
\end{proof}

Most results in the rest of this section are in the same ballpark as those of Spencer~\cite{Spencer},
and the structure of their proofs is similar to his, though the details are more like
those in \cite{BK}. Though these arguments take some time, they are relatively
routine, and a reader might want to take a quick look at statements and 
move on to the more interesting material that follows.

For the rest of this discussion, where not otherwise specified, 
$F,W,Z$ are general, and we use $\tau$ and $\gs$ for $\tau_{[W,Z]}$ and $\gs_{[W,Z]}$.
(See Lemma~\ref{disext} for a first instance.)
For the rest of the paper, $G$ means $G_{n,p}$, and from now until the end of 
Section~\ref{SEC:LemmaThresh}
we drop 
the superscript $p$ from
$\tau, \gs, \bft, \bfs$ (so e.g.\ $\tau =\tau^p =\tau_G$),
\emph{except} in the proof of Lemma~\ref{PropSB},
which involves the interplay of $\tau^p$ and $\tau^q$.

The next two lemmas, echoing a key observation 
from \cite{Spencer}, will often allow us to use $\gs$ as a surrogate for $\tau$.
For use in the proofs of Lemma~\ref{311} and Corollary~\ref{cor:pip}, we define
$\gG_{xy}$ to be the graph with vertices the $xy$-bridges of $G$ and 
two bridges adjacent if they are not internally disjoint.

\begin{lemma}\label{disext}
If $[W,Z]$ is strictly balanced 
and $p< n^{-1/d(W,Z) +\vs}$ for a small enough fixed $\vs>0$ (depending on $[W,Z]$), then 
\beq{PtauG}
\pr(\tau(X) - \gs(X) > C) < n^{-f(C)},
\enq
where $f(C)\ra \infty$ as $C\ra \infty$. 
\end{lemma}
\nin
(In \eqref{PtauG}---and occasionally below---we skip ``for all $X$,'' 
since the assertion doesn't depend on $X$.)

Note that Lemma~\ref{disext} is trivial if $\lp W,Z\rp=\0$ 
since in this case $\tau$ and $\gs$ coincide.

\begin{lemma}\label{311}
There are constants $c,C>0$ (depending on $H$) such that if $p < n^c p^*$, then w.h.p.
\beq{eq:tausig}
\bft(xy) - \bfs(xy) < C \quad \forall \{x,y\} \in \Cc{V}{2}.
\enq
\end{lemma}

The proofs of Lemmas~\ref{disext} and \ref{311} are similar 
(and similar to the proof of \cite[Prop.\ 3.11]{BK}),
and we give only the latter. 

\begin{proof}[Proof of Lemma~\ref{311}]

Choose $c < \min\{\vt,\eta/(2\eH)\}$, where $\vt,\eta$ are as in Proposition~\ref{balcons}. 
In view of \eqref{gLH},
this choice guarantees that, for $p$ as in the lemma,
\beq{cchoice}
\gL_H^p n^{-\eta} = (p/p^*)^{(\eH-1)}\gL_H^* n^{-\eta} 
= O(n^{c(\eH-1)}\log n )n^{-\eta} < n^{-\gd}
\enq 
for some (fixed) $\gd > 0$.

Noting that $\bft(x,y)-\bfs(x,y)\leq |E(\gG_{xy})|$
(recall $\gG_{xy}$ was defined before Lemma~\ref{disext}),
we find that \eqref{eq:tausig} (with an appropriate $C$) holds for $x,y$ provided
\begin{itemize}
\item[(i)] the maximum number of vertices in a component of $\gG_{xy}$ is $O(1)$ and
\item[(ii)] the maximum size of an induced matching in $\gG_{xy}$ is $O(1)$;
\end{itemize}
so we want to say that w.h.p.\ these conditions hold for all $x,y$.
(Of course replacing (i) by an $O(1)$ bound on degrees would also suffice.)

\medskip
For (i) we show that, for some fixed $M$, w.h.p.\ there do not exist
$x,y$ and bridges $K_1\dots K_M$ on $xy$ such that for 
$i\geq 2$, $\inte{V}(K_i)$ meets, but is not contained in,
$\cup_{j<i}inte{V}(K_j)$.
This bounds (by $(\vH-3)M+1$) the number of internal vertices in the bridges belonging
to a given component of $\gG_{xy}$, so gives (i).

For the preceding ``w.h.p.''\ statement, notice that if $K_1\dots K_M$ are such bridges, say 
with $R_i=\cup_{j\leq i}K_j$ and, for $i\geq 2$,
$|E(K_i)\sm E(R_{i-1})|=b_i$ and
$|V(K_i)\sm V(R_{i-1})|=a_i$ ($\leq \vH - 3$), then \eqref{nzeta} (with our choice of $c$) 
gives $n^{a_i}p^{b_i}< n^{-\eta}$
for $i\geq 2$,
implying
\beq{nVRM}
n^{|V(R_M) \sm \{x,y\}|}p^{|E(R_M)|} = n^{\vH-2} p^{\eH-1}\prod_{i=2}^M n^{a_i}p^{b_i} 
\leq n^{\vH-2} p^{\eH-1}n^{-\eta(M-1)}.
\enq
The r.h.s.\ is thus an upper bound 
on the probability that, for given $x,y$, there are 
$K_1\dots K_M$ as above with $(K_1\dots K_M$) of a given isomorphism type 
(defined in the obvious way).
So the probability that there are $x,y$ admitting such $K_i$'s 
(with automorphism types unspecified) is
$O(n^{\vH} p^{\eH-1}n^{-\eta(M-1)})$
(where the implied constant depends on $M$),
which is $o(1)$ for large enough $M$.

\medskip
The argument for (ii) is similar.
Here we want to rule out (w.h.p.), for some fixed $M$, existence of bridges
$K_1,L_1\dots K_M,L_M$ on some $xy$, with $\inte{V}(K_i)\cap \inte{V}(L_j)\neq\0$ 
and the $\inte{V}(K_i)$'s and $\inte{V}(L_i)$'s otherwise disjoint.
A discussion like the one above shows that for any such sequence, 
with $R = \cup (K_i \cup L_i)$, we have 
\beq{napb}
n^{|V(R) \sm \{x,y\}|}p^{|E(R)|} < (n^{\vH - 2}p^{\eH-1} n^{-\eta})^M < n^{-M\gd},
\enq
where the last inequality uses \eqref{cchoice}. 
As before, the r.h.s.\ of \eqref{napb} bounds the probability that, 
for given $x,y$, there is $(K_1,L_1\dots K_M,L_M)$ of a given isomorphism type.
This bounds the probability that there 
are $x,y$ and bridges as above by $O(n^{2-M\gd})$, which is $o(1)$ for large enough $M$.
\end{proof}

\begin{lemma}\label{PropSB}
If $[W,Z]$ is strictly balanced then, with $\mu =\mu_p(W,Z)$:

\nin
{\rm (a)}  for any $\gb < 1/d(W,Z)$, there is an $\eps>0$ 
such that if
$p> n^{-\gb}$, then 
\beq{prtauST}
\pr(\tau^p(X) > (1+n^{-\eps}) \mu )< n^{-\go(1)};
\enq

\nin
{\rm (b)}
for any $p$ and $\CCC >e^2$ with $\CCC \mu \gg 1$, 
\beq{prtauST'}
\pr(\tau^p(X) > \CCC \mu)< n^{-\go(1)} +
e^{-(1-o(1))\CCC \mu }.
\enq
\end{lemma}
\nin

\nin
Notice that \eqref{prtauST} and \eqref{prtauST'} don't change if we replace
$\tau$ and $\mu$ on their left hand sides by $\tdt$ and $\tdm$.

\nin

\nin
\emph{Proof.}
We first observe that if $\lp W,Z\rp=\0$ 
then $\tau^p(X)$ is binomial with mean at most $\mu$ and the lemma follows easily from 
Theorems~\ref{thm:Chernoff} and \ref{thm:Chernoff'}; precisely: (b) is immediate from 
Theorem~\ref{thm:Chernoff'}, and (a) is a weak consequence of
Theorem~\ref{thm:Chernoff}, once we observe that $p> n^{-\gb}$ implies
$\mu> n^\gd$ for a fixed $\gd>0$, and take (say) $\eps =\gd/3$.
We may thus assume $\lp W,Z\rp\neq \0$, so $b(W,Z)$ (see \eqref{ab}) is defined, and 
strict balance implies $b(W,Z)> d(W,Z)$ (see \eqref{dYZWZ}).

We first fix $\ccc < \vs$ with $\vs$ as in Lemma~\ref{disext} 
and assume $p=q:= n^{-\gc}$, where $\gc$ is fixed with
\[    
\mbox{$\gc > 1/d(W,Z) - \ccc \,\,$  and, if we are in (a),
$\,\, \gc < 1/d(W,Z)$.}
\]   

\nin
Our choices of $p$ in (a) and (b) satisfy the assumption in Lemma~\ref{disext},
and if we are in (a) there is then a (fixed) $\gd>0$ such that
\beq{mungd}
\mu> n^\gd.
\enq

Let $\vt = \vt(n)$ satisfy
\beq{1vt}
1\ll \vt \ll \left\{\begin{array}{ll}
\mu^{2/3}&\mbox{if we are in (a),}\\
\CCC \mu &\mbox{if we are in (b).}
\end{array}\right.
\enq
By Lemma~\ref{disext}, 
\[
\pr(\tau^p(X) - \gs^p(X) > \vt) =n^{-\go(1)},
\]
which we use in combination with Lemma~\ref{JUB}:

For (a) we set $\eps =\gd/3$ ($\gd$ as in \eqref{mungd}), and have
(recalling from \eqref{mutilde} that $\E \tau^p(X) \leq \mu$)
\begin{eqnarray*}
\pr(\tau^p(X) >  (1+n^{-\eps})\mu)
&<&\pr(\tau^p(X) >  \mu + \mu^{2/3} ) \\
&\leq&
 \pr(\tau^p(X) - \gs^p(X) > \vt) +
\pr(\gs^p(X) > \mu + \mu^{2/3}-\vt) \\
&< & n^{-\go(1)} +\exp[-\gO(\mu^{1/3})] \, = \, n^{-\go(1)}
\end{eqnarray*}
(using \eqref{1vt} and,
with room, $\mu^{1/3}\gg \log n$).

Similarly, for (b),
\beq{smallqb}
\pr(\tau^p(X) >  \CCC \mu ) 
\,\le\,
n^{-\go(1)} +\pr(\gs^p(X) > \CCC \mu -\vt) 
\,<\,n^{-\go(1)} +
e^{-(1-o(1))\CCC \mu },
\enq
where we sacrifice a factor like $\log(\CCC /e)$ in the last exponent
(and are using $\CCC >e^2$).

\mn

For larger $p$---say $p> q= n^{-\gc}$ 
with $\gc\in (1/d(W,Z) - \ccc,1/d(W,Z))$ fixed
(so we are aiming for \eqref{prtauST})---we 
consider the usual coupling of $G_p$ and $G_q$ 
and define events
\[
\A =\{\tau^q( X) > (1+n^{-\eps})\mu_q(W,Z)\},
\]
\[
\B =\{\tau^{p}(X) > (1+2n^{-\eps})\mu_p(W,Z)\},
\]
with the (positive) constant $\eps$ chosen so $\pr(\A) =n^{-\go(1)}$
(i.e.\ \eqref{prtauST} holds with $q$ in place of $p$).  
Then
\[
(1+2n^{-\eps})\mu_q(W,Z) < \E[\tau^q( X)|\B] < n^{v(W,Z)} \pr(\A|\B) + (1+n^{-\eps})\mu_q(W,Z);
\]
so $\pr(\A|\B) > n^{-v(W,Z)}n^{-\eps}\mu_q(W,Z) = n^{-O(1)}$ and
\[
n^{-\go(1)} =\pr(\A) > \pr(\B)\pr(\A|\B) =\pr(\B)n^{-O(1)},
\]

\nin
yielding \eqref{prtauST}  (with,
pedantically speaking, $2n^{-\eps}$ in place of $n^{-\eps}$).

\qed

When $[W,Z]$ is not strictly balanced we form, again as in \cite{Spencer},
a (not necessarily unique) sequence 
\beq{WZchain}
W = S_0 \subset S_1 \subset \cdots \subset S_k= Z
\enq
by choosing, for $i=1,\ldots~$ until $S_i= Z$, 
$S_i\supset S_{i-1}$ with $d(S_{i-1}, S_i)$ maximum and $|S_i|$ minimum 
subject to this.  Then
\beq{L8a}
\mbox{each $[S_{i-1},S_i ]$ 
is strictly balanced} 
\enq
and
\beq{L8b}
\mmm(W,Z)=d(S_0,S_1) \geq d(S_1,S_2) \geq \cdots \geq d(S_{k-1},S_k)
\enq
(since $d(S_{i-1},S_i)\geq d(S_{i-1},S_{i+1})$ implies $d(S_{i-1},S_{i+1})\geq d(S_i,S_{i+1})$;
see \eqref{dSTSU}).

\nin

\begin{lemma}\label{Erho}
With the above setup:

\nin
{\rm (a)}  For each $\gb<1/\mmm(W,Z)$, there is $\eta > 0$ such that 
if $p> n^{-\gb}$, then with probability $1-n^{-\go(1)}$
\beq{prtauST''}
\tau(X) < (1+n^{-\eta}) \mu_p(W,Z) \,\,\,\forall X.
\enq
{\rm (b)}
For any $p$, if
$\CCC_i >e^2$ and
$\CCC_i\mu_p(S_{i-1},S_i) \gg \log n$ for each $i\in [k]$, then with probability $1-n^{-\go(1)}$
\[
\tau(X) < \mu_p(W,Z)\prod_i \CCC_i \,\,\,\forall X.
\]
Equivalently, these assertions hold with $\tdt$ and $\tdm$ in place of $\tau$ and $\mu$.
\end{lemma}

\nin
(Of course here it's enough to bound failure probabilities by $n^{-\go(1)}$ for a single $X$.)

\begin{proof}  
Here, as mentioned earlier, it's better to work with extensions
(think about what happens to \eqref{mutau} if we work with copies);
so we prove the formulation in the lemma's final sentence. As in Lemma~\ref{PropSB}, the proofs of (a) and (b) are similar, and we combine them as much as possible.

Set $\tdm=\tdm_p(W,Z)$, $\tdm_i=\tdm_p(S_{i-1},S_i)$, $\tdt= \tdt_{[W,Z]}$,
and $\tdt_i=\tdt_{[S_{i-1},S_i]}$.  
With products over $i\in [k]$, we have
\beq{mutau}
\tdm =\prod \tdm_i\,\,\,\,
\text{and}
\,\,\,\,\max_X \tdt(X) \leq \prod \max_Y \tdt_i(Y)
\enq
(where, for each $i$, $Y$ runs over sequences of $|S_{i-1}|$ distinct vertices of $K_n$).

We will apply Lemma~\ref{PropSB} 
to each $[S_{i-1},S_i]$, so need to 
check hypotheses (that is, of Lemma~\ref{PropSB}(a) if we are in (a) 
and similarly for (b)); but these are immediate for (b), and for (a) follow from
$p> n^{-\gb}$ and
\[
1/d(S_{i-1},S_i) \geq 1/d(S_0,S_1) = 1/\mmm(W,Z) > \gb.
\]
Thus, letting
\[
\gz_i = \left\{\begin{array}{ll}
1 + n^{-\eps} &\mbox{in (a),}\\
\CCC_i &\mbox{in (b)}
\end{array}\right.
\]
with $\eps$ the smallest of the $\eps$'s in our $k$ appeals to 
Lemma~\ref{PropSB}(a), we find that with probability 
at least $1 - kn^{|Z|}n^{-\go(1)} = 1 - n^{-\go(1)}$
\[\tdt_i(Y) < \gz_i\tdm_i \] 
for every $i\in [k]$ and sequence $Y$ of $|S_{i-1}|$ distinct vertices of $K_n$;
which in (a) implies
\[
\max_X\tdt(X) \leq \prod(1+n^{-\eps}) \tdm_i < (1+n^{-\eta})\tdm
\]
for any fixed $\eta< \eps$, and in (b),
\[
\max_X\tdt(X) \leq \prod \CCC_i\tdm_i = \tdm_p(W,Z) \prod \CCC_i.
\]

\end{proof}

The next result is a weakening of a special case of \cite[Theorem 2]{Spencer}, so we omit
its proof, just noting that the upper bound follows easily from
Lemma~\ref{PropSB} and the lower bound is an 
application of Lemma~\ref{TJanson} (similar to the proof of Lemma~\ref{JansApp}). 

\begin{thm}
\label{thm:Spencer}
There is a $C$ such that if $\gLp \geq C \log n$, then w.h.p.
\begin{align}
\label{eq:Spencer}
\bft(x,y) = \Theta( \gLp )\quad \forall \, \{x,y\}\in \Cc{V}{2}.
\end{align}
\end{thm}
\nin
(Recall from \eqref{Etau} that $\gLp =\E \bft(x,y)$.)

\begin{cor}
\label{cor:pip} 
There are $C, c, \vs >0$ such that if $\gLp > C \log n$, then w.h.p.
$\bfs(x,y)=\Omega(\pi)$ for all $x,y$, with
$\pi = \pi(n,p)$ equal to
\begin{align}
\gLp & \quad \text{ if } ~p < n^c p^*, \label{eq:pi1}\\
n^{\vs}p/p^*& \quad \text{ if } ~p \geq n^c p^*.\label{eq:pi2}
\end{align}
\end{cor}

\nin
(Note the assumption on $\gLp$ is superfluous for $p$ as in \eqref{eq:pi2}.)

\begin{proof} 
For \eqref{eq:pi1} we can use Theorem~\ref{thm:Spencer} 
and Lemma~\ref{311}, with $C$ as in the former and $c$ as in the latter.

For \eqref{eq:pi2}, since
Theorem~\ref{thm:Spencer} says that w.h.p.\ $|V(\gG_{xy})| =\gO(\gLp)$
$\forall x,y$
(recall $\gG_{xy}$ was defined just before Lemma~\ref{disext}), 
it is enough to show
\beq{gDsmall}
\mbox{w.h.p. $~\gD(\gG_{xy}) =O(\gLp/\pi) \,\,\,\forall x,y$}
\enq
(where $\gD$ is maximum degree and we use the
the trivial $\bfs(x,y)\geq |V(\gG_{xy})|/(\gD(\gG_{xy})+1)$).

The degree (in $\gG_{xy}$) of $K\in V(\gG_{xy})$ is at most
\beq{deltabd}
\sum \tau_{[(u,v,w),H]}(x,y,z),
\enq
where $uv$ runs over isomorphism types of ordered edges of $H$, 
$w$ over $V(H)\sm \{u,v\}$, and $z$ over $\inte{V}(K)$.

Given such $uv$ and $w$, set $W=\{u,v,w\}$ and $Z=V(H)$; let 
$W = S_0 \subset S_1 \subset \cdots \subset S_k= Z $
be as in \eqref{WZchain}; and let $v_i$ and $e_i$ be the numbers of vertices
and edges of $[S_{i-1},S_i]$.
Notice that, for some {$j\in [k+1]$ (so possibly not in our range)},
\beq{geq1}
n^{v_i}p^{e_i}\geq 1
\enq
iff $i\geq j$; for if \eqref{geq1} holds for $i$, then the fact that $d(S_{i-1},S_i)$ is nonincreasing 
in $i$ (see \eqref{L8b}) gives
\[
n^{v_{i+1}}p^{e_{i+1}} = (n^{v_{i+1}/e_{i+1}}p)^{e_{i+1}} \geq (n^{v_i/e_i}p)^{e_{i+1}} \geq 1^{e_{i+1}} = 1.
\]

Noting that $\tdm_i:= \tilde{\mu}_p(S_{i-1},S_i) = n^{v_i}p^{e_i}$ 
(so $\prod \tdm_i =\tdm_p(W,Z)$)
and defining
$C_i$ to make $C_i \tdm_i $ equal to
$\log^2n$ (a convenient $\go(\log n)$)
if $i < j$, and to $\tdm_i\log^2n $ if $i \geq j$, we find that
Lemma~\ref{Erho}(b) says that w.h.p.
\beq{tdWZ}
\tdt_{[W,Z]}(X) < \prod_i C_i \tdm_i = (\log^{2k}n )\prod_{i\geq j}\tdm_i=
(\log^{2k}n)  \tdm_p(S_{j-1},V(H))\,\,\,\,\,\forall X
\enq
(though here we only care about $X=(x,y,z)$ with $z$ as in \eqref{deltabd}).

On the other hand, letting $L=H[S_{j-1}]$, we have 
(with justification of the inequality to follow)
\beq{mupS}
\tdm_p(S_{j-1},V(H)) = 
\gLp_H/\gLp_L < \gLp_H (p^*/p)n^{-\gO(1)}.
\enq
When $j\le k$ the inequality is given by Proposition~\ref{balcons} 
(which really only needs $p\geq p^*$).  If $j = k+1$ (so $L=H$)
then the l.hs.\ of \eqref{mupS} is $1$, while, since $e_H\geq 4$ (see \eqref{Otriv}),
$\gLp_H(p^*/p) = \gL_H^*(p/p^*)^{e_H-2} \asymp \log n (p/p^*)^{e_H-2}> n^c$;
so the r.h.s.\ of \eqref{mupS} is greater than 1 for a small enough $\gO(1)$.

Thus each of the $O(1)$ summands in
\eqref{deltabd} is $O (\gLp_H (p^*/p)n^{-\gO(1)})$, yielding
\eqref{gDsmall} for a small enough $\vs$.

\end{proof}

Finally, the following application of Janson's Inequality (Theorem~\ref{TJanson}) will be used for
Lemma~\ref{lem:Qsharpthresh} (where something simpler would do),
and more seriously for Lemma~\ref{prop:extind}.

\begin{lemma}\label{JansApp}
If 
$p = \tilde{\Theta}(n^{-1/d_2(H)})$, 
then 
for fixed k, distinct $\{x_1,y_1\},\ldots, \{x_k,y_k\} \in \C{V}{2}$ and $\gc \in [0,1]$,
\beq{eq:extind1}
\pr\left(\sum_i \bft(x_i,y_i) \leq \gc k\gLp\right) < \exp[-\varphi(\gamma-1)k\gLp + o(1)].
\enq
\end{lemma} 
\nin
(Recall from \eqref{eq:varphidef} that $\varphi(x) = (1+x)\log(1+x)-x$.)

\begin{proof}
It will be convenient to bound a slightly larger probability. Let
\beq{KK1}
\K =\cup_{i\in [k]}\{K:  \mbox{$K$ is a bridge on $x_iy_i$ with
$V(K)\cap (\bigcup \{x_j,y_j\})= \{x_i,y_i\} $}\}.
\enq
(The condition on $V(K)$ gives away essentially nothing and avoids the 
annoying possibility of extensions of different pairs $x_iy_i$ whose edge sets coincide.)
Then 
$
X:=|\{K\in \K: K\sub G\}|
$
is a lower bound on $\sum_i \bft(x_i,y_i)$, so it's enough to show 
\beq{Xbd1}
\pr(X \leq k\gamma \gLp)<
\exp[-\varphi(\gamma-1)k\gLp + o(1)].
\enq

We apply Theorem~\ref{TJanson} with $X$ as above; thus $X=\sum_{s\in [m]} I_s$,
where $\K=\{K_1\dots K_m\}$, $A_s= E(K_s)$ and $I_s=\mbone_{\{A_s\sub G\}}$.
Then $\mu = (1 - O(1/n))k\gLp$ (since without the condition on $V(K)$ in \eqref{KK1} 
we'd have $\mu=k\gLp$,
and adding the condition decreases $\mu$ by less than 
$k\cdot 2k\cdot n^{\vH-3}p^{\eH-1} =O(\gLp/n)$).
On the other hand, we will show that $\ov{\gD}<\mu+n^{-\gO(1)}$ ($<(1+ n^{-\gO(1)}) \mu$); 
that is,
\beq{Deltabar1}
\sum\{\E I_sI_t:  s,t\in [m], s\neq t, A_s\cap A_t\neq \0\}=n^{-\gO(1)}.
\enq
If this is true, then Theorem~\ref{TJanson} gives (with justification below)
\beq{prXk}
\pr(X \leq k\gamma \gLp ) < \exp\left[-\varphi(-(1 -k\gc \gLp/\mu))(1- n^{\gO(1)})\mu \right] < 
\exp\left[-(1- n^{-\gO(1)})\varphi(\gamma-1)k\gLp \right],
\enq
which is less than the bound in \eqref{eq:extind1} since $\gLp=\tilde{O}(1)$.  
(For the second inequality in \eqref{prXk} we use
\begin{eqnarray}
\vp(k\gc \gLp/\mu-1) &=& (1+O(1/n))\gc\log[(1+O(1/n))\gc] - ((1+O(1/n))\gc-1)\nonumber\\
&=& \gc\log \gc -(\gc-1)+O(1/n) =\vp(\gc-1)+O(1/n).)\label{vpkgc}
\end{eqnarray}

Finally, for \eqref{Deltabar1}, we observe that
the contribution to the sum 
of pairs $(s,t)$ with $|A_s\cap A_t|=b$ ($\in [1,\eH-2]$) and
$a$ \emph{internal} vertices of $K_{s}$ in $K_{t}$ is (again with justification to follow)
\[
O(k^2n^{2(\vH-2)-a}p^{2(\eH-1)-b}) < O(\mu^2 n^{-\gO(1)})  < n^{-\gO(1)}.
\]
Here the second inequality is immediate since 
\beq{ptdT}
p=\tilde{\Theta}(n^{-1/d_2(H)})
\enq 
gives 
$\mu = \tilde{O}(1)$, and for the first we need
\beq{napb2}
n^ap^b =n^{-\gO(1)}.
\enq
So suppose $K_s$ and $K_t$ are bridges (with $a,b$ as above)
on $\{x_i,y_i\}$ and $\{x_j,y_j\}$.  Then for $i=j$ 
\eqref{napb2} is given by Proposition~\ref{balcons} 
(since, in view of \eqref{ptdT}, 
$n^ap^b = \gL_K^p$ for some $K$ as in the first line of \eqref{nzeta}).
If, on the other hand, $i\neq j$, then \eqref{ptdT} says we just need 
$b/a < d_2(H)$.
But $K_s\cap K_t$ is a proper subgraph of $K_s$ with $b$ edges and
\[
|V(K_s \cap K_t)|=\left\{\begin{array}{ll}
a &\mbox{if $\{x_i,y_i\}\cap\{x_j,y_j\}= \0$}, \\
a+1 &\mbox{if $|\{x_i,y_i\}\cap\{x_j,y_j\}|= 1$}; 
\end{array}
\right.
\]
and Proposition~\ref{LS2Bd} gives 
$b/a< b/(a-1)< d_2(H)$ in the first case and $b/a < d_2(H)$ in the second.

\end{proof}

\section{Two simple points}
\label{SEC:LemmaThresh}

Here we dispose of Lemma~\ref{lem:Qsharpthresh} and the derivation of
Theorem~\ref{lem:Qsharpthresh'} from Theorem~\ref{thm:Hprecise}.
(Recall we are using $G$ for $G_{n,p}$ and $V$ for $V(G)$.)

\begin{proof}[Proof of Lemma~\ref{lem:Qsharpthresh}.] 
(We continue to use $\bft=\bft^p$.)
We begin with the 1-statement,
for which we assume 
$p=O(p^*)$ (as we may since
for larger $p$ the statement is contained in Theorem~\ref{thm:Spencer}). 
Lemma~\ref{JansApp} gives
\begin{align}
\label{eq:LemmaThreshPaths}
\pr(\bft(x,y)=0) 
<e^{-\gL^p + o(1)}.
\end{align}
So the probability that $\cQ_H$ fails (that is, $\bft(x,y)=0$ for some
$xy\in G$) is less than
\[
\Cc{n}{2}p e^{-\gL^p + o(1)} 
=o(1)
\]
(since 
$\gL^p =(p/p^*)^{\eH-1}\gL^* \more(1+\eps)^{\eH-1}\log(\C{n}{2}p)$; see \eqref{gLup*}).

For the 0-statement, we first dispose of the easy case $p \ll n^{-1/d_2(H)}$. 
Here
\[
\Psi^p_H = n^2 p \, n^{\vH - 2}p^{\eH - 1} \ll n^2 p \,n^{\vH - 2}(n^{-1/d_2(H)})^{\eH - 1} 
=  n^2p 
\] 
(using $e_H>1$ for the inequality; see \eqref{Otriv}); 
so, with $X_H$ the number of copies of $H$ in $G$, Markov's Inequality gives
$\pr(X_H > n^2p/3) =o(1)$.
But since $p\gg n^{-2}$, 
Proposition~\ref{prop:routine} says  $|G| \sim n^2p/2$ w.h.p.  So w.h.p.\ $X_H = o(|G|)$, 
and the 0-statement follows.

For larger $p$---say 
\beq{prange'}
n^{-1/d_2(H)}/\log n < p < (1-\eps)p^*
\enq
 ---we use the
second moment method.
With $Z_{xy}$ the indicator of the event $\{xy \in G\} \wedge \{\bft(x,y)=0\}$
($x,y\in V$)
and $Z=\sum Z_{xy}$, it is enough to show
\beq{EZ}
\E Z =\go(1)
\enq
and,
for distinct $\{x,y\}$,$\{u,v\}\in \binom{V}{2}$,
\beq{EZxy}
\mathbb{E}[Z_{xy}Z_{uv}] < (1+o(1))\mathbb{E}[Z_{xy}]^2.
\enq
(Then $\mathbb{E}Z^2\sim \mathbb{E}[Z]^2$ gives 
the 0-statement (namely $\pr(Z= 0)\ra 0$)
via Chebyshev's Inequality.)

By Theorem \ref{thm:Harris},
$\pr(\bft(x,y)=0) > (1-p^{\eH-1})^{\gL^1} > \exp[-\gL^p-o(1)]$,
whence
\begin{align}
\label{eq:LemmaThresh0}
\mathbb{E}[Z_{xy}] > p\exp[-\gLp-o(1)].
\end{align}
This gives \eqref{EZ} (since, by \eqref{prange'} and \eqref{gLup*},
$\gLp =(p/p^*)^{\eH-1}\gL^* \lesssim (1-\eps)^{\eH-1}\log(\C{n}{2}p)$),
and also \eqref{EZxy} when combined with
\[
\mathbb{E}[Z_{xy}Z_{uv}] ~\leq ~ p^2\cdot \pr(\bft(x,y)=\bft(u,v)=0) ~
\leq ~ p^2\exp[-2 \gLp+o(1)].
\]
Here the first inequality is given by Theorem \ref{thm:Harris}
(since $\{xy, uv \in G\}$ and $\{\bft(x,y)=\bft(u,v)=0\}$
are increasing and decreasing respectively), and, since we assume 
\eqref{prange'}, the second is given by Lemma~\ref{JansApp}.

\end{proof}

\nin
\begin{proof}[Proof that Theorem~\ref{thm:Hprecise} implies Theorem~\ref{lem:Qsharpthresh'}.]

This is again routine and we aim to be brief. Lemma~\ref{lem:Qsharpthresh} gives the 1-statement
(which is the interesting part).  For the 0-statement,
it is enough to say that for relevant $p$, $G = G_{n,p}$ w.h.p.\ contains an edge lying
in a member of $\cW_H$ but not in any (copy of) $H$. 
This is again given by Lemma~\ref{lem:Qsharpthresh} if $p$ is large enough that \emph{all} edges
are in members of $\cW_H$, which is (trivially) 
true for all relevant $p$ when $\cW_H \in \{\cE, \cD\}$, and for $p>(1+\gO(1))\log n/n$ if 
$\W_H\sub \cee$ by Claim~\ref{evencyc}.

For smaller $p$ (and $\W_H\sub \cee$), 
Claim~\ref{evencyc} says that w.h.p.\ the largest even cycle length in $G$ is 
$\go(1)$ if $p> (1-o(1))/n$
and $\gO(n)$ if $p> (1+\gO(1))/n$.  On the other hand, the expected number of
$H$'s in $G$ is less than $\Psi^p_H= n^{\vH} p^{\eH} = O((np)^{\eH})$ (see \eqref{Otriv}), which is $O(1)$ if $p = O(1/n)$ and $O(\log^{\eH} n)$ when $p = O(\log n / n)$. Since number of edges in $H$'s is
w.h.p.\ less than $\go \cdot \Psi^p_H$ for any $\go=\go(1)$, in the range under discussion, the $H$'s w.h.p.\  don't cover
the edges of the longest even cycle in $G$.

\end{proof}

\section{Proof of Lemma~\ref{lem:couplingdown}} 
\label{PLCD}

The proof 
here adapts that of
Lemma~2.3 of \cite{BK} to allow for the different possibilities for $\cW_H(G)$. 
By Corollary~\ref{cor:pip}, there is a $K>1$ such that if $p > Kp^*$, 
then w.h.p.
\begin{align}
\label{eq:sigmaLB}
\bfs^p(x,y) = \Omega(\pi)\,\,\,\,\forall \{x,y\}\in\Cc{V}{2},
\end{align}
where $\pi=\pi(n,p)$ is as in \eqref{eq:pi1},\eqref{eq:pi2} (with $c$ as in the
corollary).  We prove 
Lemma~\ref{lem:couplingdown} with this $K$.

We work in the coupling framework of
Section~\ref{subsec:Coupling}, taking $q = Kp^*$ and $G_0=G_{n,q}$.
For Lemma \ref{lem:couplingdown} it is enough to show
\[   
\pr(\{G \notin \mathcal{T}\} \; \wedge \; \{G_0 \in \mathcal{T}\}) \ra 0,
\]   
which, since $G_0\in \T$
implies $F_0\in \cW_H^\perp(G_0)$
(since $F_0\in \cee_H^\perp(G_0)$; see \eqref{F0eperp}), will follow from
\beq{fGF0}
\pr(\{F \neq \0 \} \; \wedge \; \{F_0 \in \cW_H^\perp(G_0)\})\ra 0.
\enq
So it will be enough to show that
\beq{F0cee}
F_0 \notin \cW_H^\perp(G_0)
\enq
follows (deterministically) from
\beq{F0}
F \neq \0
\enq
combined with various statements that we already know hold w.h.p.
This is not hard, but is more circuitous than one might wish.

\mn
{\em A convention.}
To slightly streamline the presentation we agree
that in this argument, appeals to a probabilistic statement $X$---e.g.
``$X$ implies'' or ``by $X$''---actually refer to
{\em the conclusion of} $X$,
which conclusion will always be something that $X$ says holds w.h.p.
The appeals to Lemma~\ref{LemCG}
and Proposition~\ref{prop:Frightsize} in the next paragraph are first examples
of this.

\medskip
If \eqref{F0} holds, then
\eqref{eq:sigmaLB} and \eqref{xyF}
(for the lower bound) together with
Lemma~\ref{LemCG} (for the upper)
imply 
\beq{couple1}
\Omega(\pi) <|F| < n^2p/10.
\enq
Since $\pi q/p \gg 1$ (using $\gL^p = (p/p^*)^{e_H-1}\gL^*$ and \eqref{gLH}
if $\pi$ is as in \eqref{eq:pi1}), the lower bound in \eqref{couple1} and the first part
of Proposition~\ref{prop:Frightsize} give $|F_0|\sim |F|q/p$, so
\beq{couple2}
1 \,\ll\, |F_0| < (1+o(1)) n^2q/10;
\enq
so, since $|G_0|\sim n^2q/2$ (by Proposition~\ref{prop:routine}), 
$F_0\not\in \{\0,G_0\}$, which already gives \eqref{F0cee} if 
$\cW_H \in \{ \cE, \cD\}$.

If instead $\cW_H\subseteq \cC$, then Proposition~\ref{prop:routine} (with \eqref{q}),
\eqref{dFdG} and the second part of Proposition~\ref{prop:Frightsize} give
\[
d_{F_0}(v) <(1+o(1))nq/2 \,\,\,\forall \, v \in V.
\]
Thus, setting $H_0=G_0\sm F_0$ and recalling the approximate ($nq)$-regularity of $G_0$
given by Proposition~\ref{prop:routine}, we have
\beq{couple3}
d_{H_0}(v) > (1-o(1))nq/2 \,\,\,\forall \, v \in V.
\enq
We claim that this, with Proposition~\ref{cpts} (applied to $G_0$), implies
\beq{H0nice}
\mbox{$H_0$ is connected (spanning) and nonbipartite.}
\enq
This gives \eqref{F0cee} because then, for any $xy\in F_0$, $(x,y)$-path $P$ in $H_0$,
and odd cycle $C$ in $H_0$, either $P\cup\{xy\}$ or its symmetric difference (i.e.\ sum)
with $C$ (or both if $\W_H=\cee$) is a member of $\W_H(G_0)$ meeting $F_0$ exactly once.
\begin{proof}[Proof of \eqref{H0nice}]
Suppose $X$ and $Y$ are components of $H_0$.  
By \eqref{couple3} and
Proposition~\ref{cpts} (applied to $H_0$), we have $|X|,|Y| > n/3$, which implies $X=Y$:
otherwise $X$ and $Y$ are disjoint and we have the contradiction
\[
(1-o(1))n^2q/9 < |\nabla_{G_0}(X,Y)| \leq |F_0| < (1+o(1)) n^2q/10,
\]
where the first inequality is given by Proposition~\ref{density} (applied to $G_0$),
the second holds because $\nabla_{G_0}(X,Y)\sub F_0$,
and the third is given by \eqref{couple2}.

Similarly, if $X\cup Y$ is a bipartition for $H_0$, say with $|X|\geq n/2$,
then Proposition~\ref{density} and \eqref{couple2}
give the contradiction
\[
(1-o(1))n^2q/8 < |G_0[X]| < |F_0| < (1+o(1)) n^2q/10.\qedhere
\]
\end{proof}

This completes the proof of Lemma~\ref{lem:couplingdown}.

\section{Proof of Lemma~\ref{lem:p=O(p)'}}
\label{SEC:LemmaP=O(p*)}

\mn
Here we introduce the main assertions, Lemmas~\ref{lem:Claim2}
and \ref{Rlemma}, underlying
Lemma~\ref{lem:p=O(p)'}, and prove the latter assuming them.
The supporting lemmas themselves are then proved in Sections~\ref{PC2}-\ref{PLJ2}.

Note that for the proof of Lemma~\ref{lem:p=O(p)'},
Lemma~\ref{lem:Qsharpthresh} allows us to restrict attention to the range
\begin{align}
\label{prange}
(1-\eps)p^*< p< Kp^*
\end{align}
(for any fixed
$\eps>0$), throughout which we have (repeating \eqref{gLH})
\beq{gL}
\gL^p\asymp \log n.
\enq
Note also that Lemma~\ref{LemCG} says
it's enough to show that for a given $\lambda=\lambda(n) \ra 0$,
\begin{align}
\label{eq:ETS1}
\pr(\{G \in \cQ\} \; \wedge \; \{0 < |F| < \gl n^2p/2\}) \ra 0.
\end{align}

We again work with the coupling of Section~\ref{subsec:Coupling},
now taking 
\beq{qp}
q=\vartheta p
\enq
with a {\em fixed} $\vartheta \in (0,1)$
small enough to support the discussion below
(the rather mild constraints on $\vt$ are at
\eqref{eq:GminusG0capR} and \eqref{vt.constraint2}).
Define the random variables $\alpha$ and $\ga_0$ by
\begin{align}
\label{eq:defalpha}
\mbox{$|F| = \alpha n^2p/2~$ and $~|F_0| = \alpha_0 n^2q/2$.}
\end{align}

We now use $\bfs_0$ for the counterpart of $\bfs$ in $G_0$ 
(we will not need $\bft_0$).
For $S \subseteq G$, a bridge on $xy$
is $S$-\emph{central}
if it contains an odd number of edges of $S$, at
least one of which does not meet $\{x,y\}$.
We use $\bfs(xy;S)$ for the maximum size of a set of internally disjoint
$S$-central $xy$-bridges, and $\bfs_0(xy;S)$ for the corresponding quantity in $G_0$.

With $c$ to be ``specified'' below (it is the $c$ in Lemma~\ref{lem:71up}, which will 
eventually be identified following \eqref{xyFgs}), set 
\beq{R}
R(S) = \{ \{x,y\} \in \Cc{V}{2} : \bfs_0(xy;S) > c \gL^q\}
\enq
and define events
\[     
\cR= \{|F \cap R(F_0)| \geq c\alpha n^2p\}
\]    
and
\[
\pee =\{0 < |F| < \lambda n^2p/2\}
\]
(the second conjunct in \eqref{eq:ETS1}).

Of the following two main points, 
Lemma~\ref{lem:Claim2} is similar to \cite[Lemma 7.1]{BK}, and its proof is to some extent also 
similar (see Section~\ref{PC2} for more on this).  
In contrast, though they play corresponding roles, there is no relation between 
Lemma~\ref{Rlemma} and \cite[Lemma 7.2]{BK}; rather the former,
which we consider the most important part of this whole story,
is where the present work most thoroughly diverges from \cite{BK}.

\begin{lemma}
\label{lem:Claim2}
There is a fixed $\eps>0$ such that for $p$ as in \eqref{prange}, w.h.p.
\begin{align}
\label{eq:Claim2statement'}
G \in \mathcal{Q}\wedge\pee \;\; \Rightarrow \;\;
G\in \cR.
\end{align}
\end{lemma}
\nin
(In other words,
$\pr(G \in \mathcal{Q}\wedge\pee\wedge \ov{\cR}) \ra 0.$
Of course $\cR$ holds trivially if $F=\0$, so it's
only the upper bound in $\pee$ that's of interest here.)

\mn
{\em Remarks.}
For $\{x,y\} \in \binom{V}{2}$, $\bfs_0(x,y)$
should be on the order of $\gL^q$.
Lemma~\ref{lem:Claim2} says that, provided $G \in \cQ\wedge \pee$,
it's likely that for a decent fraction of the edges $xy$ of $F$, even
$\bfs_0(x,y,F_0)$ is of this order of magnitude---which
is
{\em un}natural if $F_0$ is small relative to $G_0$
(since then bridges should typically avoid $F_0$).
Viewed from Lemma~\ref{lem:Claim2} the parity requirement in the definition of
``central'' may look superfluous, since a
bridge of $G_0$ joining ends of an edge of $F$ necessarily has odd intersection with $F_0$;
but this extra condition
will later play a brief but important role in justifying \eqref{RSimplies}.

For $S\sub \C{V}{2}$, set
\nin
\beq{gaS}
\ga_S = 2|S|/(n^2q).
\enq

\begin{lemma}\label{Rlemma}
There is a fixed $\gd>0$ such that w.h.p.
\beq{RS}
|R(S)| = O(\ga_S^{1+\gd}n^2)
\enq
for every $S\sub G_0$ with $\ga_S< 2\gl$. 
\end{lemma}

\mn
{\em Preview.}
The proof of Lemma~\ref{lem:p=O(p)'} is based mainly on
``coupling up'': using information about $(G_0,F_0)$ to constrain what
happens when we choose $G\sm G_0$.
On the other hand, the proof of 
Lemma~\ref{lem:Claim2} in Section~\ref{PC2}
is based on ``coupling down'': most of the work there is devoted to
the proof of a similar statement (Lemma~\ref{lem:71up}) involving only $G$,
from which the desired hybrid statement follows easily via coupling.
In sum, we couple down to show that $\cR$ is likely
(precisely, the conjunction of its failure and $\cQ\wedge\pee$ is unlikely),
and couple up to show it is {\em un}likely.
A little more on the latter:

We would like to say that 
$\pee\wedge\cR$
is unlikely, which gives \eqref{eq:ETS1} via Lemma~\ref{lem:Claim2}.
But if $F\neq \0$ (as in $\pee$) and \eqref{RS} holds (for all $S$ as in Lemma~\ref{Rlemma}),
then the lower bound on
$|G \cap R(F_0)|$ ($= |F \cap R(F_0)|$) in
$\cR$ is
larger by a crucial factor $\ga^{-\gO(1)}$ than $|R(F_0)|p$---its
natural value when we ``couple up''---which {\em ought} to make $\cR$ unlikely.
But of course $F_0$ depends on $G$ (not just $G_0$); so, given
$G_0$, we are forced to sum
the probability of this supposedly unlikely event over possible values $S$ of
$F_0$, which turns out to mean that our argument collapses if we 
replace the above $\ga^{-\gO(1)}$ by $\ga^{-o(1)}$.
(Recall $\pee$ says $\ga$ is small.)

A word on presentation.
We prove the desired
\beq{QP}
\pr(\cQ\wedge \pee) = o(1)
\enq
(= \eqref{eq:ETS1})
by producing a list of unlikely events and showing
that at least one of these must hold if $\cQ\wedge\pee$ does.
A
more intuitive formulation might, for example, begin:
``By Lemma~\ref{lem:Claim2} (since we assume $\cQ\wedge \pee$), {\em we may assume}
$\cR$.''
But note this would really mean, not that we
{\em condition} on $\cR$---not something we can hope to understand---but
that we need only bound probabilities $\pr(\sss\wedge \cR)$
for $\sss$'s of interest, and for a formal discussion this
seems most clearly handled by something like the present approach.

\mn

For the derivation of Lemma~\ref{lem:p=O(p)'}
we need two more events (supplementing $\pee,\cQ,\cR$ above).
The first of these is simply
\[
\sss = \{ \ga_0\sim \ga\}
\]
(i.e.
for any $\eta>0$, $\ga_0=(1\pm \eta)\ga$ for large enough $n$;
recall $\ga,\ga_0$ were defined in \eqref{eq:defalpha}).
The second, which we call $\T$, is the conjunction of two properties of $G_0$
that we already know hold w.h.p., namely  (i) $|G_0|\sim n^2q/2$ (see Proposition~\ref{prop:routine}),
and (ii) \eqref{RS} holds for all $S$ as in Lemma~\ref{Rlemma}.

\begin{proof}[Proof of Lemma~\ref{lem:p=O(p)'}]
We first observe that (a secondary use of $\cR$)
\beq{RSo}
\mbox{w.h.p. $\cR\wedge \{F\neq\0\} \Ra \sss$.}
\enq
(Of course $\{F\neq\0\}$---given by $\pee$---could be omitted here since 
$\sss$ is automatic if $F=\0$.)
\begin{proof}[Proof of \eqref{RSo}]

If $F\neq \0$ (i.e.\ $\alpha> 0$) and $\cR$ holds, then
$F\cap R(F_0)\neq \0$, while by \eqref{xyF}, for any $xy\in F\cap R(F_0)$,
\[
|F|> \bfs(xy)\geq \bfs_0(xy;F_0)
> c\gL^q = \Omega(\log n).
\]
But then (since $\log n \gg p/q$)
Proposition \ref{prop:Frightsize} says that w.h.p.\
$|F_0| \sim \vartheta |F|$, which is the same as $\sss$.\qedhere

\end{proof}
Another easy (now deterministic) observation is that, provided $\vt$ is sufficiently small,
\beq{RSimplies}
\cR\wedge\{F\neq\0\} \wedge \sss ~\Longrightarrow ~|(G \setminus G_0) \cap R(F_0)| 
>  0.9c\alpha n^2p.
\enq
\begin{proof}

Note it is always true that
$     
G_0 \cap R(F_0) \subseteq F_0,
$    
since for any $xy \in (G_0 \cap R(F_0)) \setminus F_0$ there is an $xy$-bridge (actually, many such)
having odd intersection with $F_0$, and adding $xy$ to such a bridge would produce a copy of $H$
having odd intersection with $F_0$.
(This is the above-mentioned reason for ``odd'' in the definition of central.)
So if $\cR$, $\{F_0\neq \0\}$ and $\sss$ hold (and $\vt $ is slightly small), then
\begin{align}
\label{eq:GminusG0capR}
|(G \setminus G_0) \cap R(F_0)| ~\geq ~|F \cap R(F_0)| - |F_0|
~> ~c \alpha n^2p - (1+o(1))\ga n^2q/2  ~>~ 0.9c\alpha n^2p.
\end{align}\qedhere

\end{proof}

So the conjunction of $\pee,\cR,\sss$ and $\T$ implies
(again deterministically) the event---call it $\U$---that
$|G_0|< n^2q$ (say) and there is an $S\sub G_0$
(namely the one that will become $F_0$) satisfying 
\beq{U}
\mbox{$\ga_S<2\gl$, $|R(S)|=O(\ga_S^{1+\gd}n^2)$,
and $|(G \setminus G_0) \cap R(S)| >  0.8c\ga_S n^2p$}.
\enq
(The 0.8 is just to allow for the difference between $\ga_S$ and $\ga$.)

Thus, finally, for \eqref{eq:ETS1} it is enough to show the routine 
\beq{Utoshow}
\pr(\U) =o(1).
\enq
(It is enough because: since $\ov{\U}$ implies $\ov{\pee}\vee\ov{\cR}\vee\ov{ \sss}\vee\ov{\T}$,
\eqref{Utoshow} implies
\[
\pr(\cQ\wedge (\ov{\pee}\vee\ov{\cR}\vee\ov{ \sss}\vee\ov{\T}))=\pr(\cQ)-o(1);
\]
but the l.h.s.\ here is at most
\[
\pr(\cQ\wedge \ov{\pee}) + \pr(\cQ\wedge\pee\wedge\ov{\cR}) +
\pr(\pee\wedge \cR\wedge \ov{ \sss})+\pr(\ov{\T})=\pr(\cQ\wedge\ov{\pee})+o(1),
\]
the second and third terms on the l.h.s.\ being $o(1)$ by Lemma~\ref{lem:Claim2} and \eqref{RSo}
respectively;
so $\pr(\cQ\wedge \pee) =\pr(\cQ)-\pr(\cQ\wedge \ov{\pee}) =o(1)$.)

\begin{proof}[Proof of \eqref{Utoshow}]
Given $G_0$, $S$, we have
$|(G \setminus G_0) \cap R(S)|\sim {\rm Bin}(m,p')$, with $m\leq |R(S)|$ and
$p'<p$
defined by $(1-q)(1-p')=1-p$
(as in (B) of Section~\ref{subsec:Coupling}).
So for $|R(S)|$ as in \eqref{U},
Theorem~\ref{thm:Chernoff'} gives
\[
\pr(|(G \setminus G_0) \cap R(S)| >  0.8c\alpha_S n^2p) < \exp[- \gO(\ga_Sn^2p\log (1/\ga_S))],
\]
where the implied constant depends on $\gd$ but not on $\vt$.
Thus, assuming $|G_0|< n^2q$ (as in $\U$),
setting $\ga_s = 2s/(n^2q) $ 
(where $s$ will be $|S|$, so $\ga_s=\ga_S$),
and
summing over $s< 2\gl n^2q$,
we have
\begin{eqnarray}
\pr(\U|G_0) &<&\mbox{$\sum_s\binom{n^2q}{s}\exp[- \gO(\ga_sn^2p\log (1/\ga_s))]$}
\nonumber\\
&<&\mbox{$\sum_s \exp[\ga_sn^2p\{(\vt/2) \log (2e/\ga_s)- \gO(\log (1/\ga_s)\}]$},
\label{vt.constraint2}
\end{eqnarray}
which is $o(1)$ for small enough $\vt$
(which implies \eqref{Utoshow} since
$\pr(\U) =\sum\{\pr(G_0)\pr(\U|G_0):|G_0|<n^2q\}).$
\end{proof}
\end{proof}

\section{Proof of Lemma~\ref{lem:Claim2}}\label{PC2}

Fix $\eps>0$ and $p$ as in \eqref{prange} with $\eps$
small enough to support the proof of Propositions~\ref{prop:Hcomps} below; this is our
only constraint on $\eps$, and it will be clear it is satisfiable.

Most of our effort here is devoted to proving the following
variant of
Lemma~\ref{lem:Claim2} in which we replace $\bfs_0(xy,F_0)$ by
$\bfs(xy,F)$ and $q$ by $p$.

\begin{lem}\label{lem:71up}
There is a fixed $c>0$ such that 
w.h.p.
\beq{eq:71up}
G \in \cQ \wedge \pee \implies |\{xy \in F : \bfs(xy, F) > 2c\gL^p\}| \geq 2c|F|.
\enq
\end{lem}

\nin
``Coupling down'' will then easily get us to Lemma~\ref{lem:Claim2} itself.
(The extra $2$'s leave some room for this.) Throughout this section we use ``central'' for ``$F$-central'' and, except in the ``coupling down'' proof of Lemma~\ref{lem:Claim2} at the end of the section, write $\gL $ for $ \gL^p$.

\mn
{\em Preview.}
The proof of \cite[Lemma 7.3]{BK}, which corresponds to the present
Lemma~\ref{lem:71up},
is indirect, showing that there are (i) many pairs $(e,K)$ 
with $K\sub G$ a copy of $H$ and $e\in K\cap F$, and (ii) significantly fewer
such pairs in which $K$ is \emph{non}central for $e$.
Something similar is needed here, but now needs to be combined
with a direct approach
in situations where we have less control over the numbers in (ii).

The ``direct'' part of the proof of Lemma~\ref{lem:71up}
depends on $H$ having a vertex of degree at 
least three; so for cycles, economizing where we can, 
we just appeal to the earlier version:

\begin{lem}[\cite{BK}, Lemma 7.3]
\label{CycleSpread}
If $H$ is a cycle of length at least four, then 
\[\cQ \wedge \cP \implies |\{xy \in F: \bfs(xy; F) > .26\gL\}| \geq .26|F|.\]
\end{lem}
\nin
(The statement in \cite{BK} is for \emph{odd} cycles, which were the concern there,
but the proof is valid in general.
The difficulty for even cycles is in the derivation of Lemma~2.2 from 
Lemmas~7.1 and 7.2 in Section~7.)
So we now assume, until we are done with Lemma~\ref{lem:71up}, that $H$ is not
a cycle (though the assumption will not be needed until the end of the proof; see 
the paragraph containing \eqref{XvK}).

\nin\emph{Definitions. } 
\emph{For the rest of this section, $K$ is always a copy of $H$.} 
We write $K \sim K'$ when $K$ and $K'$ are distinct and share at least one edge. 
For distinct edges $e$ and $f$ of $G$, we set
\beq{eq:edgeadj}
e \sim f \iff \mbox{there is $K\subset G$ such that $e,f \in K$},
\enq
\beq{eq:edgedubadj}
e \approx f \iff \mbox{there are $K \sim K'$ with $e \in K$ and $f \in K'$,}
\enq
$S(e) = \{g \in G: e \sim g\}$, and $T(e) = \{g \in G: e \approx g\}$. For $\gc \in (0, 1)$ let
\[L(\gc) = \{\{x,y\} \in \Cc{V}{2}: \bfs(xy) < \gc \gL\}\]
and $F(\gc) = F \cap L(\gc)$. 
Finally, let $\cS$ be the event that the inequality in
\eqref{eq:tausig} holds for every $xy \in G$.

\begin{prop}\label{prop:regext}
For any fixed $\gz>0$, w.h.p.
\[G \in \cQ \implies |F(1-\gz)| = o(|F|).\]
\end{prop}
\nin

\nin
We prove this by showing---in Propositions~\ref{prop:largeextdev} and \ref{prop:smallextdev}, 
assisted by Propositions~\ref{prop:extind} and \ref{prop:Hcomps}---that 
$F(\gz)$ and $F(1-\gz)\sm F(\gz)$ are small.

\begin{prop}\label{prop:extind}
For fixed $\gc \in (0,1)$ and k, and distinct $\{x_1,y_1\},\ldots, \{x_k,y_k\} \in \C{V}{2}$, 
\beq{eq:extind}
\pr(\cS \wedge \{\{x_i,y_i\}\in L(\gc)\; \forall i\in [k]\}) < 
(n^2p)^{-(k - o(1))(1-\eps)^{\eH -1}\varphi(\gamma-1) }.
\enq
\end{prop}
\nin
(We again recall that $\varphi(x)$ was defined in \eqref{eq:varphidef}.)
Note the bound here is natural, being, for $p$ at the lower bound in
\eqref{prange} (and up to the $o(1)$), what
Theorem~\ref{thm:Chernoff} would give for the probability that $c$ independent binomials,
each of mean $\gL^p $,
are all at most $\gamma \gL^p $.

\begin{proof}
Since $\cS$ gives $\bft(xy) \leq \bfs(xy)+C <  (1+o(1))\gamma \gL $
for $\{x,y\}\in L(\gamma)$, the event in \eqref{eq:extind} implies
\beq{firstX}
\sum_{i \in [k]} \bft(x_i,y_i) < (1+o(1))k\gamma \gL =: k\gc'\gL
\enq
(where $\gc' \sim \gc$);
so the l.h.s.\ of \eqref{eq:extind} is less than the probability of \eqref{firstX}. 
But Lemma~\ref{JansApp} gives (cf.\ \eqref{vpkgc})
\[
\pr(\eqref{firstX}) ~<~ \exp\left[-\varphi(\gamma'-1)k\gL +o(1)\right]
~< ~\exp\left[-(1-o(1))\varphi(\gamma-1)k\gL \right],
\]
which, since
$\gL  \more(1-\eps)^{\eH-1}\log[\Cc{n}{2}p]$ (see \eqref{gLup*}), 
is at most
the r.h.s.\ of \eqref{eq:extind}. 
\end{proof}

\begin{prop}\label{prop:Hcomps}
W.h.p.\ if $K_1\sim K_2\sim K_3 \sim K_4$ are copies of $H$ in $G$, then 
\beq{Q's1}
|(\cup K_i) \cap L(\gz)| \leq 1. \enq
Also, there is a fixed $M$ such that w.h.p.\ 
\beq{Q's2}
|S(e) \cap L(1-\gz)| < M \quad \forall e\in G.\enq
\end{prop}
\begin{proof}
Write $\eta_\gc$ for the quantity
$
(n^2p)^{-(1 - o(1))(1-\eps)^{e_H - 1}\varphi(\gamma-1) }
$
appearing in \eqref{eq:extind} (here without the $k$).

Since $\mathcal{S}$ occurs w.h.p.\ (see Lemma~\ref{311}),
it suffices to show that the probability that it holds while either
\eqref{Q's1} or \eqref{Q's2} fails is $o(1)$.
Thus in the case of \eqref{Q's1} we want to bound the probability
that $
\mathcal{S}  \wedge  \{J \subseteq G\}  \wedge  \{|J \cap L(\zeta)| \geq 2\}$
holds for some $J\sub K_n$ of the form $\cup_{i \in [4]} K_i$, where the $K_i$'s are $H$'s
sharing edges as appropriate.  With $\T(J)=\sss\wedge \{|J\cap L(\gz)|\geq 2\}$,
this probability is at most
\begin{align*}
\mbox{$\sum_J \pr(\{J\sub G\}\wedge\T(J))$}~
&\leq~
\mbox{$\sum_J \pr(J \subseteq G) \pr(\T(J))~$}\\
&\leq ~\mbox{$O(n^2p\gL^4\eta_\gz^2) =o(1).$}
\end{align*}
Here the first inequality is an instance of Theorem~\ref{thm:Harris}
(since $\{J \subseteq G\}$ and
$\T(J)$ are increasing and decreasing respectively);
Proposition~\ref{prop:extind} gives
$\pr(\T(J)) =O(\eta_\gz^2)$ (for any $J$);
and the $o(1)$ holds (for small enough $\eps$ and $\gz$) by the definition of $\eta_\gz$
(since $\gL$ is as in \eqref{gL}).
The argument for
\beq{sumJG}
\sum\pr(J\sub G) = O(n^2p\gL^4)
\enq
is as follows.
If $K_1\sim K_2 \sim K_3 \sim K_4$ and $R_i:=\cup_{j\leq i}K_j$, 
then $Q_i:= R_{i-1} \cap K_i$ is a non-empty subgraph of $H$ for $i\geq 2$. 
So for $i\geq 2$ the 2-balance of $H$ implies
$n^{|V(Q_i)|}p^{|E(Q_i)|}\geq n^2p$ (see \eqref{gLH} and \eqref{gLKbig}) and thus
$n^{|V(H)|-|V(Q_i)|}p^{|E(H)|-|E(Q_i)|}=O(\gL)$,
and \eqref{sumJG} follows.

Treatment of \eqref{Q's2} is similar. Here $J$ runs over subsets of $K_n$ of the form $\cup_{i \in [M]} K_i$, where the $K_i$'s are $H$'s with a common edge,
and, with $\T(J) =\sss\wedge \{|J \cap L(1-\zeta)| \geq M\}$,
the probability that $\sss$ holds while \eqref{Q's2} fails is at most
\[
\mbox{$\sum \pr(\{J \subseteq G\} \wedge \T(J)) $}
~\leq ~O(n^2p \gL^M\eta_{1-\gz}^M) ~=~o(1).
\]
This is shown as above, with
$n^{|V(J)|}p^{|E(J)|} \leq n^2p\gL^M$
given by the passage following \eqref{sumJG}
(with $M$ in place of 4),
and the $o(1)$ valid for large enough $M$ by definition of $\eta$.
\end{proof}

As noted above, the next two assertions give Proposition~\ref{prop:regext}

\begin{prop}\label{prop:largeextdev}
W.h.p.\ 
\beq{eq:largeextdiv}
G \in \cQ \implies |F(\gz)| = o(|F|).\enq
\end{prop}
\begin{proof}
By the first part of Proposition~\ref{prop:Hcomps} it is enough to show that
the r.h.s. of \eqref{eq:largeextdiv} follows (deterministically) from
the conjunction of $\{G \in \cQ\}$ and \eqref{Q's1}.
But these imply that $|T(e) \cap F| \geq \zeta \gL $ for each $e\in F(\gz)$:
$\{G \in \cQ\}$ gives at least one $H$ containing $e$;
this $H$ contains a second $F$-edge $g$ (since $F\in \cC_H^\perp$), which
by \eqref{Q's1} is not in $L(\gz)$; thus $T(e)$ contains the at least $\zeta \gL $
(distinct) $F$-edges of $S(g)$.
Moreover, again by \eqref{Q's1},
$T(e)\cap T(f)=\0 $ for distinct $e,f\in F(\zeta)$.
Thus $|F(\zeta)|< |F|/(\zeta \gL )$ ($=o(|F|)$), as desired.
\end{proof}

\begin{prop}\label{prop:smallextdev}
W.h.p.\ 
\beq{mid}
|F(1-\gz)\sm F(\gz)| = o(|F|).\enq
\end{prop}
\begin{proof}
It's enough to show that \eqref{Q's2} implies
\eqref{mid} (since Proposition~\ref{prop:Hcomps} says \eqref{Q's2} holds w.h.p.).
This is again easy:
Set $B=F(1-\zeta)\sm F(\zeta)$ and consider the graph with vertex set $F$ and
adjacency as in \eqref{eq:edgeadj}.
Each $e \in B$ has degree at least $\zeta \gL  $ in this graph,
while \eqref{Q's2} says
no vertex has more than $M$ neighbors in $B$.
Thus $|B| (\zeta \gL  - M) \leq |\nabla(B, F \sm B)| \leq |F \sm B|M$, yielding \eqref{mid}.
\end{proof}

We turn to the second part of the proof of Lemma~\ref{lem:71up}
(producing \emph{central} pairs $(e,K)$), support for which will be
provided by Lemmas~\ref{spread}-\ref{spreaderr}. For these three lemmas we fix $x\in V(H)$, 
say with $d_H(x)=d$, and use $a$ for an integer from $[0, \dHx ]$.
For $v\in V$, 
let $\m_v$ be the set of copies $K$ of $(H;x)$ on $v$ (i.e.\ with $\cop{x}=v$)
in $K_n$, and
\[
M_v= |\{K \in \m_v:K\sub G\}|.
\]
\nin

\nin
By Proposition~\ref{LS2Bd},
$[x,V(H)]$ is strictly balanced with $d(x,V(H)) < m_2(H)$; 
so Lemma~\ref{PropSB}(a), with \eqref{prange}, says that w.h.p.
\beq{MUB}
M_v < (1+n^{-\ccc})\E M_v  \,\,\,\,\forall v \in V.
\enq
(Of course $\E M_v$ doesn't depend on $v$.)  

For the rest of this section we fix $\kappa\in (0,1/(6d))$ and set
\beq{zetac}
\vs= \log^{-\kappa}n.
\enq

\begin{lem}\label{spread}
W.h.p.: for all $v \in V$ and 
$S\sub \nabla_G(v)$, with 
\beq{gammavs}
\gc:=|S|/ d_G(v) \in 
\left\{\begin{array}{ll}
(\vs, 1)&\mbox{if $a=d$,}\\
(0, 1-\vs)&\mbox{if $a=0$,}\\
(\vs, 1-\vs)&\mbox{otherwise,}
\end{array}\right.
\enq
\beq{spreadIneq}
|\{K\in \m_v :K\sub G, \,d_{K\cap S}(v) = a\}| > 
(1 - \log^{-1/3}n)\C{\dHx }{a} \gc^a (1-\gc)^{\dHx  - a}\E M_v.
\enq
\end{lem}
\nin
(We won't need the expanded range of $\gc$ for $a=0$.)
\begin{proof} 
Let
\[
\cB_{v,S}=\{\mbox{\eqref{spreadIneq} fails for $(v,S)$}\}.
\]
By Proposition~\ref{prop:routine} there is a fixed $\eta$ 
such that w.h.p.
\[
\nabla_G(v) \in \cZ_v:= \{Z \subset \nabla_{K_n}(v): |Z| = (1\pm n^{-\eta})np\} \,\,\,\forall v;
\]
so it is enough to show that, for any $v$, $Z\in \cZ_v$ and $S\sub Z$ as in \eqref{gammavs},
\beq{prcBvF}
\pr(\cB_{v,S}\cond \nabla_G(v)=Z) < e^{-\go(np)},
\enq
since then 
\begin{align}
    \pr(\cup_{v,S} \cB_{v,S}) &< 
    o(1) + \sum_v\sum_{Z \in \cZ}\pr(\nabla_G(v) = Z) \sum_{S}
    \pr(\cB_{v,S}\cond \nabla_G(v) = Z)\\
    &<
    o(1)  + ne^{np - \omega(np)} \, =\,o(1),
\end{align}
where $e^{np}$ (over)counts possibilities for $S$ contained in $Z\in \cZ_v$.

\mn

For the rest of the proof, aiming for \eqref{prcBvF}, 
we fix $v$, $Z\in \cZ_v$, say of size $z$,
and $S\sub Z$ (with $\gc=|S|/|Z|$ as in \eqref{gammavs}),
and condition on $\nabla_G(v)=Z$.
We now use $\m$ and $M$ for $\m_v$ and $M_v$, and \emph{always assume} $K\in\m$.
Intending to use Theorem~\ref{TJanson}, we set (given $a$)
\[
\cK = \{K: \nabla_K \subset Z, d_{K\cap S}(v) = a\} =\{K_1\dots K_m\},
\]
and let $I_j$ be the indicator of the event $\{K_j \sub G \}\,$
(so the $A_j$'s for the theorem are the sets $E(K_j-v)$) and $X = \sum I_j$.
To ease notation, set $b = \dHx  - a$ and $\gG_{a} = \C{\dHx }{a} \gc^a (1-\gc)^{b}$.

We first observe that (with $\mu = \E X$ as in Theorem~\ref{TJanson})
\beq{muM}
\mu = (1 \pm O(n^{-\eta})) \gG_{a}\, \E M.
\enq
\begin{proof}
Let $J$ be the number of $K$'s with a given $\nabla_K(v)$ (namely, $J= \C{n-1}{d}^{-1}|\m|$),
and notice that
\[
\gG_a \E M 
=\C{\dHx }{a} \gc^a (1-\gc)^{b}|\m|p^{\eH},
\]  
while
\[
\mu = \C{\gc z}{a}\C{(1-\gc)z}{b}J p^{\eH-d}
= \frac{1}{a!b!} (\gc z)_a((1-\gc)z)_b \, \frac{d!}{(n-1)_d}
|\m| p^{\eH-d};
\] 
so for \eqref{muM} we just need 
\[
 (\gc z)_a((1-\gc)z)_b  = (1 \pm O(n^{-\eta})) \gc^a(1-\gc)^b (n-1)_d\, p^d,
 \]
which is true since $Z\in \cZ_v$ and,
for fixed $k$, $(m)_k$ ($:=m(m-1)\cdots (m-k+1)$) $= (1-O(1/m))m^k$.

\end{proof}

It follows that for \eqref{prcBvF} it is enough to show 
\[     
\pr(X< (1-0.5 \log^{-1/3}n)\mu) < e^{-\go(np)},
\]    
which is implied by 
Theorem~\ref{TJanson} provided
\beq{ovDsp}
\ov{\gD} = o\left(\frac{\mu^2}{np\log^{2/3}n }\right)
\enq
(where, as in \eqref{Delta},
$\ov{\gD} = \sum_i \sum_{j\sim i} \E I_i I_j$).
So it remains to check \eqref{ovDsp}.

\nin

We bound $\ov{\gD}$ in a few steps. 
[Language:  In
each step we add to the specification of a pair $(K_i,K_j)$ that contributes to $\ov{\gD}$,
bounding the number of possibilities (``choices'') for the step, and, 
if adding $t$ edges not containing $v$, ``paying'' a ``cost'' $p^t$, 
the probability that these edges are in $G$.  
The ``contribution'' of the step is the product of these two values.
The bound on the number of choices for a step depends on where its new vertices are to be located:
there are at most $\gc z$ choices for vertices that must lie in $N_S(v)$;
$(1-\gc)z$ for those that must lie in $N_{Z\sm S}(v)$; and $n$ for those that are unconstrained.
(The bounds then correspond to \emph{labeled} $(K_i,K_j)$'s, but as the target is 
\eqref{ovDsp} there is no profit in avoiding this overcount to gain constant factors.)]

\nin
(i)  Choose and pay for $K_i$, contributing a factor $\mu$.
    
\nin
(ii) Choose in $O(1)$ ways the \emph{overlap} $L:= K_i \cap K_j$ (there is no cost here since $L\sub K_i$).

\nin
(iii)
Choose the remaining vertices of $K_j$ in at most
\[
(\gc z)^{a - |\nabla_L(v) \cap S|}((1-\gc) z)^{b - |\nabla_L(v) \sm S|}
n^{\vH - (\dHx + 1)- |V(L)\sm (\{v\}\cup N_L(v))| }
\]
ways, paying
\[
p^{\eH - \dHx - |E(L-v)|}.
\]
Since $j \sim I$ means $A_i\cap A_j$ (not just $K_i\cap K_j$) is nonempty,
$L$ must be a copy on $v$ of $[x,V(H')]$ for some 
\beq{Lrange}
\mbox{$H' \sub H $ with $ x\in V(H')$ and $E(H'-x) \neq \0$.}
\enq

Collecting the contributions from (i)-(iii) (and using $Z\in \cZ_v$),
we find that the product of the terms not involving $L$ is
\[
O(\mu \gc^a(1-\gc)^b n^{\vH -1}p^{\eH}) =O(\mu^2)
\]
(see \eqref{muM}),
and the inverse of the product of terms that do involve $L$ is 
\[
\gO(\gc^{|\nabla_L(v) \cap S|} (1-\gc)^{|\nabla_L(v)\sm S|}(np)^{d_L(v)} n^{v_L-d_L(v)-1}p^{e_L -d_L(v)})
\,=\,
\gO(\gc^{|\nabla_L(v) \cap S|} (1-\gc)^{|\nabla_L(v)\sm S|}n^{v_L-1}p^{e_L}).
\]
So, letting $\gS_L = n^{v_L-1}p^{e_L}$, we have 
\beq{Delc}
\ogd = O(\mu^2/\min_L\{\gc^{|\nabla_L(v) \cap S|} (1-\gc)^{|\nabla_L(v)\sm S|}\gS_L\}).
\enq

Finally, we claim that, for each (possible) $L$,
\beq{PsiLow}
\gS_{L} = \gO(n p \log n).
\enq
If this is true, then inserting in \eqref{Delc}, and using 
\[
\gc^{|\nabla_L(v) \cap S|}(1-\gc)^{|\nabla_L(v) \sm S|} \,>\, \vs^{\dHx} \,\,\,\,(\gg \log^{-1/3}n)
\]
(which trivially holds for all $(a,\gc)$ covered by the lemma),
gives \eqref{ovDsp}.

For \eqref{PsiLow}, just notice that, for any $H'$ as in \eqref{Lrange}, 
Proposition~\ref{balcons} (with \eqref{prange}) implies
\[
\gS_L\asymp np\gL_{H'} =\left\{\begin{array}{ll}
\Theta(np\log n)&\mbox{if $H'=H$ (i.e.\ $i=j$)},\\
\gO(n^{1+\gO(1)}p)&\mbox{otherwise.}
\end{array}
\right.
\]
\end{proof}

\begin{lem}\label{lem:spreadeq}
W.h.p.\ for all $v \in V$,
$S\sub \nabla_G(v)$, and $\gc=|S|/ d_G(v)$ as in \eqref{gammavs},
\beq{spreadEq}
|\{K\in \m_v:K\sub G, \,d_{K\cap S}(v) = a\}|
\sim \C{\dHx}{a} \gc^a (1-\gc)^{\dHx - a}\E[M_v].
\enq
\end{lem}

\nin
(We just need to bound the order of magnitude of the l.h.s.\ of \eqref{spreadEq}
(see \eqref{Cnv-1}), but the present argument, combining \eqref{MUB} with the \emph{lower} bounds 
of Lemma~\ref{spread},
automatically gives the stronger statement.)

\begin{proof} Lemma~\ref{spread} gives the asymptotic lower bound, so it suffices to show probable nonexistence of $v$ and $a_0$ satisfying (with $\vs$ as in \eqref{zetac}, though any
$o(1)$ would suffice)
\beq{spreadFail}
|\{K\in \m_v:K\sub G, \,d_{K\cap S}(v) = a_0\}| 
> (1 + \vs) \C{\dHx}{a_0} \gc^{a_0} (1-\gc)^{\dHx - a_0}\E[M_v].
\enq
(Recall $\m_v$ and $M_v$ were defined following the proof of Proposition~\ref{prop:smallextdev}.)

As mentioned above we show \eqref{spreadFail} by combining \eqref{MUB} with the lower bounds of 
Lemma~\ref{spread}.

\nin
(i)  For general $a_0$ and $\gc\in (\vs, 1-\vs)$, \eqref{spreadFail} and Lemma~\ref{spread} 
(and 
$\sum_a \gG_{a} = 1$), give
\begin{align}
    \nonumber M_v = \sum_{a = 0}^{\dHx} |\{K\in \m_v :K\sub G, \,d_{K\cap S}(v) = a\}|
    &>  (1 + \vs) \gG_{a_0}\E[M_v] + \sum_{a \neq a_0} \left(1 - \log^{-1/3}n\right) \gG_{a}\E[M_v]\\
    &>\label{MLB} \left(1 + \vs \gG_{a_0} - \log^{-1/3}n \right)\E[M_v],
\end{align}
which, since $\vs \gG_{a_0} \gg\log^{-1/3}n$, contradicts \eqref{MUB}.

\nin
(ii) For $a_0=d$ and $\gc\geq 1-\vs$, the lower bound in \eqref{spreadFail} is already larger than
the upper bound in \eqref{MUB}
(namely, $(1+\eta)\gc^d > 1+n^{-\eps}$, which really just needs $\gc $ a little more than $1-\eta/d$),
which is again a contradiction.
(And the argument for $a_0=0$ and $\gc\leq \vs$ is identical.)

\end{proof}

For $v\in V$ and $S\subseteq \nabla_G(v)$, let 
$
T_S(v)= \{K \in \m_v:K\sub G, |K\cap S|\geq 2\}
$
and $\tau_S(v) = |T_S(v)|$. 
(So $\tau_S(v)$ is a variant of $\tau_{[x,H]}(G,v)$.)
\begin{lem}\label{spreaderr}
For any $\gc_0=o(1)$,
w.h.p.\ 
\beq{eq:spreadSmall}
\tau_S(v)  < o(\gc n^{\vH-1} p^{\eH})
\enq
for all $v \in V$ and $S\sub \nabla_G(v)$ with $|S| = \gc np< \gc_0np$.
\end{lem}

\begin{proof} We use a reduction similar to (but cruder than) the one used for 
Lemma~\ref{311}. 
Let $\gs_S(v)$ be the size of a largest collection, say $\cee$, of $K$'s from $T_S(v)$ that 
are edge-disjoint outside $\nabla_G(v)$.

\begin{claim*}
There is a (fixed) $C$ so that w.h.p.\ for all $v\in V$ and $S\subseteq \nabla(v)$,
\[\tau_S(v) < C \gs_S(v).\]
\end{claim*}
\begin{proof}
Let $\gG$ be the graph whose vertices are the copies, $K$, of $H$ in $G$, with
$K \sim K'$ if $|V(K) \cap V(K')| \geq 3$. 
With $C$ TBA, it is enough to show
\beq{compbound}
\text{w.h.p.\ no component of $\gG$ has more than $C$ vertices.}
\enq
[\emph{Because}:  
any component of $\gG$
meeting $T_S(v)$, say in $K$, also meets $\cee$, since some $K'\in \cee$
shares $v$
and an edge off $v$ with $K$, so is either $K'$ or a neighbor of $K'$ in $\gG$.]

We show that, for some fixed $M$, w.h.p.\ there do not exist copies
$K_1\dots K_M$ of $H$ such that
$1 \leq |V(K_i) \sm \cup_{j<i}V(K_j)| \leq \vH - 3$ for $i\geq 2$. 
Suppose $K_1\dots K_M$ is such a sequence, with $R_i=\cup_{j\leq i}K_j$, 
$|E(K_i)\sm E(R_{i-1})|=b_i$ and
$|V(K_i)\sm V(R_{i-1})|=a_i$.
Then Proposition~\ref{balcons} (applied with $L=K_i\cap R_{i-1}$) and \eqref{prange}
give $n^{a_i}p^{b_i}\leq n^{-\gz}$ for some fixed $\gz>0$ and $i\geq 2$, whence
\beq{M-1}
n^{|V(R_M)|}p^{|E(R_M)|} \leq n^{\vH} p^{\eH}(n^{-\gz})^{M-1}.
\enq
So the probability of seeing such a sequence is at most the r.h.s.\ of \eqref{M-1},
which is $o(1)$ for a (not very) large (fixed) $M$.
So w.h.p.\ the $K$'s in any component of $\gG$ cover fewer than $(\vH-3)M+3$ vertices
of $G$, and the claim follows.
\end{proof}

By the claim, it is enough to show \eqref{eq:spreadSmall} with $\tau$ replaced by $\gs$.
Given
$v$ and $S \subset \nabla_G(v)$ (so we condition on this) 
of size $\gc np$, let
\beq{maybemu}
\mu = \E \tau_S(v) <\gc^2 n^{\vH - 1} p^{\eH}.
\enq
Let $\max\{\gc^{1/2},1/\log n\}\ll \vr\ll 1$ and
\beq{maybemu2}
A =A_S= \vr \gc n^{\vH - 1} p^{\eH} = \Theta(\vr \gc np\log n).
\enq
Then $A=o(\gc n^{\vH-1} p^{\eH})$, so for \eqref{eq:spreadSmall} (with $\gs$ in place of $\tau$),
it is enough to show that we're unlikely to see any $v$ and $S$ with $\gs_S(v)> A_S$.

With $K := A/\mu \geq \vr/\gc$
(not to be confused with the $K$'s that are copies of $H$), Lemma~\ref{JUB} gives 
\beq{eq:spread2}
\pr(\gs_S(v) \geq A) <
\exp[-A\log(K/e)] =
\exp[-\gO(A\log(1/\gc))]
=\exp[- \gO(\vr\gc np\log n\log(1/\gc))];
\enq
so the probability that there are $v$ and $S$ with $\gs_S(v)>A_S$ is less than
\beq{spUn}
n\C{n}{\gc np}p^{\gc np} \exp[-A \log(K/e)] 
< 
\exp[\log n + \gc np \log(e/\gc) - \gO(\vr\gc np\log n\log(1/\gc))] = n^{-\go(1)}
\enq
where the first three terms on the l.h.s.\ correspond to summing $\pr(S\sub G)$ over
$v\in V$ and $S\sub \nabla_{K_n}(v)$,
and we note that $S\neq \0$ implies $\gc\geq 1/(np)$.

\end{proof}

\begin{proof}[Proof of Lemma~\ref{lem:71up}] 
We are now done with the special $x$ of Lemmas~\ref{spread}-\ref{spreaderr},
and use $K$ for a general copy of $H$ in $G$.
For $v\in V$ let $\gc_v = d_F(v)/d_G(v)$, 
and set $B=\{v: \gc_v > \theta\}$, with the positive constant $\theta $ small enough to support
the following argument (see \eqref{smalltheta}).
As mentioned at the beginning of this section, we supplement the indirect approach of \cite{BK}, 
which works here only when $B$ is small, with a direct argument for larger $B$.

\subsubsection*{Proof when $|B| < \theta \ga n$} Write $\bft^*(xy)$ 
for the number of $xy$-bridges with an
edge of $F$ incident to $x$ or $y$. Let $\vp_{v,k} = |\{K \ni v: |K\cap \nabla_F(v)| = k\}|$
and recall $\vs$ was defined in \eqref{zetac}. 
Writing $\sum'$ and $\sum''$ for sums over $v$'s with $\gc_v > \vs$ and 
$\gc_v \leq \vs$ respectively, we have w.h.p.
\beq{Cnv-1} 
\sum_{xy \in F} \bft^*(xy) \,\leq\, \sum_v \sum_{k\geq 2} k \vp_{v,k}
\,\leq \, C n^{v_H - 1}p^{e_H} \left[\sideset{}{'}\sum \gc_v^2  + \sideset{}{''}\sum o(\gc_v)\right],
\enq
where $C$ is a constant (depending on $H$) 
and the limit in $o(\gc)/\gc\ra 0$ as $\gc \ra 0$.
Here the first inequality comes from considering how many times each side counts the 
various $K$'s (the l.h.s.\ is 
$\sum_{xy\in F}|\{K\supset xy:\max\{d_{K\cap F}(x), d_{K\cap F}(y)\}\geq 2\}|$
and the r.h.s.\ is 
$\sum_{xy\in F} \sum_{K\supseteq xy}
[\mbone_{\{d_{K\cap F}(x)\geq 2\}} +\mbone_{\{d_{K\cap F}(y)\geq 2\}}]$),
and the second uses Lemmas~\ref{lem:spreadeq}
(with \eqref{MUB} when $\gc> 1-\vs$) and~\ref{spreaderr}
(with, pickily, Proposition~\ref{prop:routine} to say $d_G(v)\sim np$ in the definition of 
$\gc_v$).
Since $\sum_v \gc_v \sim \ga n$ w.h.p.\ (by Proposition~\ref{prop:routine}), the second sum is $o(\ga n)$.  For the first, we have
\[\sum_{v \in B} \gc_v^2 \leq |B| < \theta \ga n\]
and 
\[\sum_{v \not\in B} \gc_v^2 \leq \sum_v \theta \gc_v = \theta \ga n,\]
so the total is at most $2\theta \ga n$. So in all,
\beq{badF}
 \sum_{xy \in F} \bft^*(xy) < 3\theta C \ga n^{v_H}p^{e_H}.
\enq

Now let (say) $\gz =1/4$ and
$F^* = \{xy\in F: \bfs(xy) > (1-\gz) \gL\}$ ($= F\sm F(1-\gz)$),
and recall (see Proposition~\ref{prop:regext}) that
$|F^*| \sim \ga n^2p/2$ w.h.p. If $xy \in F^*$, then $\bfs(xy;F) > \gL/2$ unless $\bft^*(xy)> (1/2 - \gz)\gL$;
so, with $\tilde{F} = \{xy \in F: \bfs(xy, F)\leq \gL/2\}$, \eqref{badF} implies, 
for small enough (fixed, positive) $\theta$,
\beq{smalltheta}
|\tilde{F}| \leq \frac{3\theta C \ga n^{v_H}p^{e_H}}{(1/2 - \gz)\gL}
<\ga n^2 p /5,
\enq
whence $|F^*\sm \tilde{F}| >\ga n^2 p/4$, which is (more than) enough for Lemma~\ref{lem:71up}.

\subsubsection*{Proof when $|B| \geq \theta \ga n$}
We first recall two conditions that have already been shown to hold w.h.p.:
first, with $\bft(xy;F)$ the number of central bridges on $xy$
(and $\bfs(xy;F)$ its internally disjoint counterpart), 
\beq{whp1}
\bft(xy;F)-\bfs(xy;F) \,\,\,\,(\leq \bft(xy)-\bfs(xy))\,\, <C \,\,\,\forall xy
\enq
for a suitable fixed $C$ (see Lemma~\ref{311}); 
and second, 
\beq{whp2}
\bft(xy) < C \gL
\,\,\,\,\forall x,y,
\enq
again, for a suitable fixed $C$ (see Lemma~\ref{PropSB}(b), with 
its assumed strict balance given by \eqref{S2Bequiv}, and \eqref{gLH}).

For the lemma (in the present ``large $B$'' case) it is enough to show that w.h.p.
\beq{goal71}
|\{(e,K) : e\in F,\, \mbox{$K$ is central for $e$}\}| = \gO(|F|\gL),
\enq 
since \eqref{whp1} and \eqref{whp2}  (respectively) then give, still w.h.p.,
\beq{sumxyF}
\sum_{xy \in F} \bfs(xy,F) > \gO(|F|\gL) -O(|F|)  = \gO(|F|\gL)
\enq
and, with $2\nu$ the implied constant on the r.h.s.\ of \eqref{sumxyF},
\[
\sum_{xy \in F} \bfs(xy,F)  < C\gL|\{xy\in F: \bfs(xy,F)>\nu \gL\}| +\nu|F|\gL,
\]
implying
\beq{xyFgs}
|\{xy \in F : \bfs(xy, F) > \nu\gL\}|  >\nu|F|/C.
\enq

So we may take $c$ in Lemma~\ref{lem:71up} to be $\nu/(2C)$ (or, pedantically,
the minimum of this and 1/4, since for $|B|< \theta \ga n$ we had
$|\{xy \in F : \bfs(xy, F) > \gL/2\}|  >\ga n^2p/4$;
see following \eqref{smalltheta}).

For \eqref{goal71},
recalling that we assume $H$ is not a cycle (see following Lemma~\ref{CycleSpread}) and using
\eqref{Otriv},
we may fix $x\in V(H)$ with $d_H(x)$
at least 3 and odd if $H$ is non-Eulerian, and 
at least 4 if $H$ \emph{is} Eulerian, and define $a $ to be $d_H(x)$ in the first case 
and $d_H(x)-1$ in the second.
With $\cN_v= \{\mbox{copies of $(H,x)$ on $v$ in $G$}\}$,
Lemma~\ref{spread}, combined with \eqref{dFdG} in the Eulerian case
(making $\gc$ in the lemma at most 1/2), then says that (w.h.p.)
for every $v\in B$,
\beq{XvK}
X_v := |\{K \in \cN_v: |\nabla_{K\cap F}(v)| = a\}| = \gO(n^{v_H - 1}p^{e_H}).
\enq
So, since each $K$ in \eqref{XvK} is central for at least one $vw\in F$, the l.h.s.\ of \eqref{goal71}
is at least
\[ 
\frac{1}{2}\sum_{v\in B} X_v \geq \gO(\ga n^{v_H}p^{e_H}) = \gO(|F|\gL).
\]
\end{proof}

\begin{proof}[Proof of Lemma~\ref{lem:Claim2}]
As mentioned earlier, this follows easily from Lemma~\ref{lem:71up}
via  ``coupling down'':
it is enough to show that if $G$ satisfies the r.h.s.\ of \eqref{eq:71up} then
w.h.p.\ it also satisfies $\cR$; that is,
$|F \cap R(F_0)| \geq c\alpha n^2p$.

For
$xy\in F':= \{xy\in F:\bfs(xy;F) > 2c\gL^p\}$,
Theorem~\ref{thm:Chernoff} gives
\[
\pr(\bfs_0(xy;F_0) \leq c\gL^q) < \exp[-\Omega(\log n)]
= n^{-\gO(1)},
\]
since members of a set of $\bfs(xy;F)$ internally disjoint, $F$-central
bridges on $xy$ survive in $G_0$ (and become $F_0$-central) independently, each with probability
$\vt^{e_H-1}$.
So by Markov's Inequality,
w.h.p.
\[
|\{xy\in F': \bfs_0(xy;F_0) \leq c\gL^q\}|=o(|F'|).
\]
The lemma follows.
\end{proof}

\section{Proof of Lemma~\ref{Rlemma}}\label{PL7.2}

From this point all randomization is in $G_0$ ($G$ does not appear), and we
drop the superscript and subscript $q$'s 
from $\bft$, $\tau$, $\bfs$, $\gs$ $\mu$, $\gL$ and $\Psi$
(e.g.\ $\tau_{[W,Z]}(X)$ is $\tau_{[W,Z]}(G_0,X)$).
In addition, all $S$'s are now contained in $G_0$; we always assume
$0< \ga_S<2\gl$;
and $x,y$ are always distinct members of $V$.

Let $\ell$ range over triples $(s,t,g)$ with $st$ and $g$ 
disjoint edges of $H$.
For such an $\ell$ define an \emph{$(\ell,S)$-bridge} on $(x,y)$ to be 
a copy of $(H-st;s,t,g)$ with $\cop{s}=x$, $\cop{t}=y$ and $\cop{g}\in S$.
(So now \emph{we do keep track of the order of $(s,t)$ and $(x,y)$}.)
Let $\bfs_0(x,y;\ell,S)$ be the maximum size of a set of internally disjoint $(\ell,S)$-bridges 
on $(x,y)$ in $G_0$ (where, as earlier, 
\emph{internally disjoint} means sharing no vertices other than $x$ and $y$).

Then
\[
\bfs_0(xy;S) \leq \sum_\ell \bfs_0(x,y;\ell,S).
\]
So with $\mmm$ the number of $\ell$'s, $c'=c/\mmm$ ($c$ as in \eqref{R}), and 
\beq{RlS}
R_\ell(S) = \{(x,y) : \bfs_0(x,y;\ell,S) > c'\gL\}
\enq
(recall $\gL=\gL^q$), we have
\[
R(S) \sub \bigcup_\ell R_\ell(S);
\]
and for \eqref{RS} it suffices to show, again for some fixed $\gd>0$ and every $\ell$,
\beq{Rl}
\mbox{w.h.p. $\,|R_\ell(S)| = O(\ga_S^{1+\gd}n^2)\,$ for all $S$.}
\enq

For the rest of this and the next section (i.e.\ for the rest of the paper) 
we fix $\ell = (s,t,g)$ (as above).
The following graph $\Hs$ and subgraph $\Ks$ will be central.

Let $H^*$
be the graph gotten from copies $H_1$, $H_2$ of $H$ by 
identifying the two copies of $s$, and similarly $t$, and deleting the copies of the edge $st$.
Thus $v_{H^*} = 2\vH-2$, $e_{H^*} = 2\eH - 2$ and
\beq{gLH*}
\gL_{H^*}\asymp \gL^2/q.
\enq
We continue to use $s$ and $t$ for the appropriate vertices of $H^*$, and 
use $g_\eps$ for the copy of $g$ in $H_\eps$ ($\eps =1,2$).

A \textit{connection} is a connected subgraph of $\HH$ 
containing $g_1$ and $g_2$.  For the rest of the paper we fix a connection
$\Ks$ with $\gL_\Ks$ as small as 
possible.\footnote{We 
mention for perspective, though we won't need this, that $\Ks$ can be taken to be $L_1\cup L_2$, 
where $L\sub H$ contains $g$ and at least one vertex of $f$,
and $L_i$ is the copy of $L$ in $H_i$.}

We use $\gL_*$ for the ``clean'' version of the 
expected number of copies 
of $(K^*;g_1,g_2)$ with $\cop{g}_2$ equal to some specified
$xy\in \C{V}{2}$ and all other edges in $G_0$; 
thus
\beq{gL*}
\gL_* = 2\tdm_{q}(g_2,\Ks)/\textrm{aut}(K^*;g_1,g_2).
\enq
\nin
(Here the literal expectation is 
$2(n-2)_{v_{K^*}-2}\,q^{e_{K^*}-1}/\textrm{aut}(K^*;g_1,g_2) 
= (1-O(1/n))\gL_*$. We note for clarity that
$\gL_*$ differs from $\gL^*$ in referring to $\Ks$ rather than $H$
and in being ``clean'' rather than exact.)

We will be particularly interested in copies of $(K^*;g_1,g_2)$ in $G_0$ with 
$\cop{g}_1$ and $\cop{g}_2$
belonging to a given $S$; we
call such copies $S$-\emph{ropes}, or simply \emph{ropes} if $S$ is understood.
(The name is from \cite{BK}, where the ropes were paths with terminal edges in $S$.)
We use $\K(S,S)$ for the set of $S$-ropes in $G_0$ and 
$\rho(S,S)=|\K(S,S)|$.

\nin
\emph{Preview.} 
Each $(x,y)\in R_\ell(S)$ generates on the order of $\gL^2$
copies $\cop{H}^*$ of $H^*$ with $\cop{g}_1,\cop{g}_2\in S$
(since any ordered pair of disjoint $(\ell,S)$-bridges on $(x,y)$ gives such a copy);
and each such $\cop{H}^*$ corresponds to an $S$-rope, with---crucially 
and nontrivially---most 
of these ropes not corresponding to too many $\cop{H}^*$'s.
So a large $R_\ell(S)$ implies existence of many such $\cop{K}^*$'s.
A little more precisely, the number of $S$-ropes
implied by an $R_\ell(S)$ of size $\ga_S^{1+\gd}n^2$ 
turns out to be (w.h.p.) more than $\ga_S^{1+2\gd}\Psi_{K^*}$, which 
(for $\gd<1/2$) is much larger than the natural $\ga_S^2\Psi_{K^*}$.
The surprisingly challenging
main point is then to show that this unnatural behavior really is unlikely.  

These two main steps are formalized in the next two lemmas, from which \eqref{Rl} follows
immediately (since Lemma~\ref{MP2} says that w.h.p.\ the behavior in \eqref{rhoLB}
does not occur).

\begin{lemma}\label{rhoCor}
W.h.p.\ for all $S\sub G_0$ (with $\ga_S<2\gl$) and $|R_\ell(S)|> \ga_S^{1+\gd}n^2$, 
\beq{rhoSSS}
\rho(S,S) = \left\{\begin{array}{ll}
\gO(|R_\ell(S)| \gL^2) &\mbox{if $\Ks =\Hs$,}\\
\gO(|R_\ell(S)| \gL_* q\log^{-v_{H^*}}(1/\ga)) &\mbox{if $\Ks\neq\Hs$;}
\end{array}\right.
\enq
in particular, 
\beq{rhoLB}
\rho(S,S) = \gO(\ga_S^{1+2\gd} \Psi_{\Ks}).
\enq
\end{lemma}

\nin
(The $2\gd$ in \eqref{rhoLB} is convenient but has no special significance.
The first case of \eqref{rhoSSS} doesn't require the lower bound
on $|R_\ell(S)|$, but of course \eqref{rhoLB}---for which see the value of $\gL_{H^*}$
in \eqref{gLH*}---does.)

\begin{lemma}\label{MP2}
There is a (fixed) $\gd>0$ such that w.h.p.\ 
$\rho(S,S)\ll \ga_S^{1+2\gd}\Psi_{\Ks}$ for all $S\sub G_0$.
\end{lemma}
\nin
(Recall we assume $S\sub G_0$ and  $\ga_S<2\gl$.)

The proof of Lemma~\ref{rhoCor} and the first part of that of Lemma~\ref{MP2} are given
in this section, with details for the main point of the latter---the 
Theorem~\ref{JDK} application in Lemma~\ref{J1}---given in Section~\ref{PLJ2}.
Before turning to these main points, we record a few easy properties of $\Ks$ and $\Hs$.

\begin{prop}\label{gLprop}
{\rm (a)}  If $\Ks\neq H^*$ then
\beq{gap}
\gL_* <\log^{-\ccc}n\cdot \gL_{H^*}.
\enq
{\rm (b)}
For any connection $F \subsetneq \Ks$,
\beq{gammagap}
\gL_{F} > n^{\ccc}\cdot\gL_*.
\enq

\nin
{\rm (c)}
For any $Y\in \lp V(\Ks),V(\Hs)\rb$, $d(V(\Ks),Y)\leq d_2(H)$.
\end{prop}

\nin
\emph{Proof}.
We show the equivalent statements with $\gL$ replaced by $\Psi$ ($=\Psi^q$), all of which 
follow from 
\beq{anyF}
\Psi_F \asymp n^{v_F-  e_F (v_H-2)/(e_H-1)}\log^{e_F/(e_H-1)}n \,\,\,
\mbox{(for any $F$)},
\enq
which is immediate from the definition of $q$ (see \eqref{qp}, \eqref{prange} and 
the definition of $p^*$).

\nin
{\rm (a)}
If $\Psi_{\Ks}\leq \Psi_{H^*}$ then 
in view of \eqref{anyF} (and the definition of $K^*$), either
\beq{vKvH}
v_\Ks- (v_H-2)e_\Ks/(e_H-1) < v_{H^*}- (v_H-2) e_{H^*}/(e_H-1),
\enq
implying
$
\gL_* = n^{-\gO(1)}\gL_{H^*},
$
or the two sides of \eqref{vKvH} are equal so $e_{\Ks}<e_{H^*}$ implies 
\[
\gL_* = \Theta((\log  n )^{(e_\Ks-e_{H^*})/(e_H-1)}\gL_{H^*});
\]
and in either case we have \eqref{gap}.

\nin
{\rm (b)}
Any $F$ as in (b) must satisfy
$v_F- v e_F/(e_H-1) > v_\Ks- v e_\Ks/(e_H-1) $, since otherwise $e_F<e_\Ks$ would imply 
$\Psi_F < \log^{-\gO(1)}n\cdot \Psi_\Ks$.

\nin
(c)  Since $\Psi_Y\geq \Psi_\Ks$, \eqref{anyF} implies
\[
v_\Ks- (v_H-2)e_\Ks/(e_H-1) \leq |Y|- (v_H-2) e_{H^*[Y]}/(e_H-1);
\]
or, rearranging,
\[ 
0 \leq v(V(K^*), Y) - (v_H-2)e(V(K^*), Y)/(e_H-1),
\]
which is what we want.

\qed

We will be interested in lower bounds on $\Psi_K$ 
for subgraphs $K$ of $H^*$, for which we mainly need to 
understand such bounds for \emph{connected} subgraphs.  
Recall (from \eqref{gLKbig}) 
that strict 2-balance of $H$ implies 
\beq{gLKbig'}
\gL_{L} = n^{\ccc}
\enq
for any $L \subsetneq H$ with $e_{\LLL}\geq 2$.

For connected $L\sub H^*$ \eqref{gLKbig'} gives the following bounds.  
If $L $ is
contained in one of $H_1$, $H_2$, then 
\beq{L0}
\Psi_L= 
n^2q \gL_{L} 
\left\{\begin{array}{ll}
\mbox{$=n^2q$ if $L\cong K_2$,}\\
\mbox{$> n^{2+\ccc} q$ otherwise.}
\end{array}\right.
\enq
For $L\not\sub H_1,H_2$ we use 
$v_{\eps}$ and $e_{\eps}$ for the 
numbers of vertices and edges of $L_\eps:=L \cap H_\eps$ ($\eps=1,2$);
in these cases:
If $L$ contains just one of $s,t$, then 
\beq{L1}
\Psi_L
~=~n^{v_1+v_2-1}q^{e_1+e_2}
~=~ 
n^3 q^2\gL_{L_1}\gL_{L_2} >n^{3+\ccc}q^2.
\enq
If $L$ contains both $s$ and $t$, 
then with $M_\eps =L_\eps +st$ (a subgraph of $H$),
\beq{L2}
\Psi_L ~=~  
n^{v_{M_1}+v_{M_2}-2}q^{e_{M_1}+e_{M_2}-2} ~=~ 
n^2   \gL_{M_1}\gL_{M_2} 
\left\{\begin{array}{ll}
\asymp n^2\log^2n&\mbox{if $L=H^*$,}\\
> n^{2+{\ccc}} &\mbox{otherwise.}
\end{array}\right.
\enq

These observations will be used repeatedly in the Proof of Lemma~\ref{J1};
for now we just note the easy consequences:
\beq{m2H*}
m_2(H^*) < d_2(H);
\enq
\beq{gamlarge}
\gL_* > n^{\ccc};
\enq
and
\beq{K*2}
\mbox{if $\Ks\neq H^*$ then
$\Ks$ contains exactly one of $s, t$, and $\gL_* > n^{1+\ccc}q$.}
\enq

\nin
\emph{Proofs.}  
For \eqref{m2H*}, note that a $K\sub H^*$ with $d_2(K)\geq d_2(H)$ would have 
$\gL_K=\tilde{O}(1)$
(see \eqref{gLH}),
whereas \eqref{L0}-\eqref{L2} and \eqref{q} imply $\gL_K > n^{\ccc}$.  
Similarly, \eqref{gamlarge} follows from \eqref{L1} and \eqref{L2} (and \eqref{q}).
Finally, the first assertion of \eqref{K*2} holds since otherwise we have
(using the second part of \eqref{L2}, Proposition~\ref{gLprop}(a), \eqref{gLH*} and \eqref{gLH})
the contradiction 
$n^{\ccc}/q < \gL_* < \gL_{H^*}\asymp \gL^2/q\asymp q^{-1}\log^2n$;
and the second assertion is then given by \eqref{L1}.
\qed

\subsection*{Many ropes}

Let 
\[
V(\Ks) =:W=S_0 \subset S_1 \subset \cdots \subset S_k= Z :=V(\Hs)
\]
be as in \eqref{WZchain}; so $k< v_\Hs$ and, as in \eqref{L8b}, 
\beq{dSSH}
d(S_{i-1},S_i)\leq d(S_{0},S_1)\leq d_2(H) \,\,\,\,\,\forall i,
\enq
where the second inequality is given by Proposition~\ref{gLprop}(c).

\nin
\emph{Proof of Lemma~\ref{rhoCor}.}
Here $i$ runs over $[k]$ and we set $\mu_i=\mu(S_{i-1},S_i)$ (recall
$\mu:= \mu_q$)
and
\beq{mu*}
\mbox{$\mu^* = \min_i\mu_i = \gO(\log^{1/(e_H -1)} n)$}
\enq
(with the bound following from \eqref{dSSH}).

Call a copy of the form $(\cdots;\cdots,\cg_1,\cg_2)$ \emph{good} if $\cg_1,\cg_2\in S$
(so $\K(S,S) =\{\mbox{good copies of $(\Ks;g_1,g_2)$}\})$.

For $S$ as in Lemma~\ref{rhoCor}, 
let $N=N_S$ be the number of 
good copies of 
of $(\Hs;\Ks, g_1,g_2)$ (in $G_0$).
We first observe that, for any such $S$,
\beq{largerho}
N
= \gO(\ga_S^{1 + \gd} n^2 \gL^2).
\enq
\emph{Proof.}
For each $(x,y)\in R_\ell(S)$ there are at least  
$c \gL (c\gL-1)$ good
copies of $(\Hs;s,t,g_1,g_2)$.
Each of these corresponds in the natural way to a (good) copy of  
$(\Hs;g_1,g_2)$, and each such copy is gotten from only $O(1)$ choices of $(x,y)$;
so there are $\gO(\gL^2)$ good copies of $(\Hs;g_1,g_2)$, and each of these produces 
at least one good $(\cHs;\cKs, \cg_1,\cg_2)$.\qed

If $\Ks = \HH$ then $\rho(S,S)=N$ and we are done
(see \eqref{gLH*});
so we may assume $\Ks\neq\Hs$.  To avoid too much repetition, now call
a good $(\cKs;\cg_1,\cg_2)$ a \emph{triad} and a good
$(\cHs;\cKs, \cg_1,\cg_2)$ a \emph{quad}.
So to get from \eqref{largerho} to \eqref{rhoSSS} we should limit
the numbers of quads corresponding to the various triads.
Here we can somewhat simplify, using the trivial
\beq{onon}
\mbox{the number of quads including the triad 
$(\cop{K}^*;\cop{g}_1,\cop{g}_2)$ is at most the number of $\cHs$'s on $\cKs$}
\enq
(i.e.\ for which $(\cHs;\cKs)$ is a copy of $(\Hs; \Ks)$).

Say $\cop{H}^*$ is $\ga$-\emph{nice} if for each $i$ and 
$X \sub V(\cop{H}^*)$ of size $|S_{i-1}|$, 
\beq{tauSS}
\tau_{[S_{i-1},S_i]}(X)  < \max\{2e^2\mu_i, 3\log(1/\ga)\} =:\pi_\ga,
\enq
and $\ga$-\emph{bad} if for some such $i$ and $X$,
the number of copies of
$[S_{i-1},S_i]$ on $X$ using no vertices of $V(\cop{H}^*)\sm X$ is 
at least $ \pi_\ga':= 1.1 \max\{e^2\mu_i, 2\log(1/\ga)\}$.

Lemma~\ref{rhoCor} will follow from the next three assertions, for the last of which,
recalling that $\mu^* = \min_i\mu_i $, we set
\[
\xi_\ga = \min\{\exp[-e^2\mu^*], \ga^2\}.
\]

\begin{claim}\label{nnisb}
    W.h.p.\ every $\cop{H}^*$ that is not $\ga$-bad is $\ga$-nice.
\end{claim}

\begin{obs}\label{Onice}
For any $\cop{K}^*$, the number of $\ga$-nice 
$\cHs$'s on $\cop{K}^*$ is 
\[O(\log^k(1/\ga))\mu(\Ks,\Hs).\]
\end{obs}

\begin{claim} \label{Cgabad}
W.h.p.
\beq{gabad}
\mbox{for all $\ga$ the number of $\ga$-bad $\cop{H}^*$'s is at most $\xi_\ga \Psi_{H^*}$.}
\enq
\end{claim}

\nin
Indeed, if \eqref{gabad} holds then for any $\ga$ there are at most
$o(\ga^{1+\gd}n^2\gL^2)$
quads 
with $\cop{H}^*$ $\ga$-bad,
since $\xi_\ga\ll \ga^{1+\gd}$, $\Psi_{H^*}\asymp n^2\gL^2$ (see \eqref{gLH*}),
and no $\cHs$ contributes more than $O(1)$ quads. 
So Claims~\ref{Cgabad} and \ref{nnisb}, with \eqref{largerho},
say that for $S$ as in Lemma~\ref{rhoCor} 
there are $\gO(\ga_S^{1+\gd}n^2\gL^2)$ quads 
with $\cop{H}^*$ $\ga_S$-nice. 
The lemma is thus given by Observation~\ref{Onice} (and \eqref{onon}),
since $\mu(S_0,S_k) \asymp\gL_{H^*}/\gL_*\asymp \gL^2/(q\gL_*)$ (again see \eqref{gLH*})
and $k\leq v_{\Hs}$.
\qed

\begin{proof}[Proof of Claim~\ref{nnisb}]
Let $[k] = I \sqcup J$ where $i \in I$ iff $d(S_{i-1},S_i) < d_2(H)$ 
(so $d(S_{i-1},S_i) = d_2(H)$ for $i \in J$ by \eqref{dSSH}). 
Since the $[S_{i-1},S_i]$'s are strictly balanced,
Lemma~\ref{PropSB}(a) says that w.h.p.\ \eqref{tauSS} holds  for every 
$i \in I$ and $X$.
So it's enough to show that for any fixed $i\in J$,
w.h.p.\ \eqref{tauSS} holds for each $\cHs$ that is not $\ga$-bad
(and each relevant $X$).

If $i\in J$ then $q$ satisfies the hypothesis in Lemma~\ref{disext} for $(W,Z) = (S_{i-1},S_i)$ (and some fixed $\vs$), so 
w.h.p.\ 
\beq{dss}
\tau_{[S_{i-1},S_i]}(X)  -\gs_{[S_{i-1},S_i]}(X)  =O(1)
\,\,\,\, \forall X,
\enq
implying in particular that for any $\cop{H}^*$ and $X \subset V(\cop{H}^*)$, 
the number of copies of
$[S_{i-1},S_i]$ 
on $X$ using vertices of $V(\cop{H}^*)\sm X$ is $O(1)$
(since if there are $t$ such copies sharing $v\in V(\cop{H}^*)\sm X$,
then the l.h.s.\ of \eqref{dss} is at least $t-1$).
But if this is true then
any $\cHs$ that is not $\ga$-bad has
$
\tau_{[S_{i-1},S_i]}(X)  < \pi_\ga'+O(1) < \pi_\ga.
$
\end{proof}

\begin{proof}[Proof of Observation~\ref{Onice}]  The number of $\ga$-nice
$\cHs$'s on a given
$\cop{K}^*$
is at most the number of sequences
$(X_0\dots X_k)$ with $X_0=V(\cKs)$
and, for each $i\in [k]$, $X_i$ the vertex set of an
$[S_{i-1},S_i]$-extension (in $G_0$) on $X_{i-1}$ with
$\tau_{[S_{i-1},S_i]}(X_{i-1})  <\pi_\ga $, and thus
$\tdt_{[S_{i-1},S_i]}(X_{i-1}) =O(\pi_\ga )$.
But the number of such sequences is (crudely) less than
\[
O(1)\log^k(1/\ga)\prod \mu_i =O(\log^k(1/\ga) \mu(\Ks,H^*)).\qedhere
\]
\end{proof}

\begin{proof}[Proof of Claim~\ref{Cgabad}]
We first note that we may assume $\mu^* =O(\log n)$.
(The assumption is convenient but unimportant.)
For if $\mu^* > C\log n$ with, say, $C=v_{H^*}$, then Lemma~\ref{PropSB}(b) 
says that
$\tau_{[S_{i-1},S_i]}(X)  < e^2\mu_i$ for \emph{every} $i$ and $X$
(so no $\cop{H}^*$ is $\ga$-bad for any $\ga$).

For a given copy $J$ of $\Hs$ in $K_n$ to be an $\ga$-bad $\cHs$ 
we must have $J\sub G_0$ and, for some $i\in [k]$ and $X\sub V(J)$ of size $|S_{i-1}|$,
\beq{tausub}
\tau_{[S_{i-1},S_i]}(X) > \pi_\ga'.
\enq
Note that by the 
definition of $\ga$-bad 
$\{J\sub G_0\}$ is
independent of the events in \eqref{tausub}.

By Lemma~\ref{PropSB}(b), the probability of \eqref{tausub} for a given $i$ and $X$ 
is less than 
$
\exp[-(1-o(1))\pi'_\ga]
$
(this absorbs the $n^{-\go(1)}$ of \eqref{prtauST'} since we assume $\mu^* =O(\log n)$).
So (in view of the above-noted independence)
the expected number of $\ga$-bad $\cHs$'s is less than 
\[
\sum_J\pr(J\sub G_0) \cdot \sum_{i,X} \exp[-(1-o(1))\pi'_\ga] <
O(\Psi_{\Hs} \cdot \exp[-0.1e^2\mu^*] \cdot\xi_\ga),
\]
and by Markov's Inequality the probability that there are at least $\xi_\ga \Psi_{H^*}/4$
such $\cHs$'s is less than $O(\exp[-0.1e^2\mu^*])$, which by \eqref{mu*} is 
$o(1/\log n)$.
(The reason for the 4 will appear momentarily.)

This is not quite  enough for a union bound (over $\ga$), but
we may proceed, in a standard way, as follows.
Let $D=\{\ga<2\gl:  \ga =2^i\cdot 2/(n^2 q), i\in \pr\}$.
then $|D| < \log (\gl n^2 q)=O(\log n)$, so w.h.p.\ 
\beq{gainD}
\mbox{there are fewer than
$\xi_\ga \Psi_{H^*}/4$ $\ga$-bad $\cHs$'s for each $\ga\in D$.}
\enq
But for any other $\ga$ there is an $\ga'\in D\cap [\ga,2\ga]$, so, since 
$\ga$-bad implies $\ga'$-bad, \eqref{gainD} implies that the number of
$\ga$-bad $\cHs$'s is less than
$\xi_{\ga'} \Psi_{H^*}/4 < \xi_\ga \Psi_{H^*}$. 
\end{proof}

This completes the proof of Lemma~\ref{rhoCor}.

\subsection*{Few ropes}

Here we prove Lemma~\ref{MP2} modulo the proof of 
Lemma~\ref{J1}, which will be given in Section~\ref{PLJ2}.

We now slightly extend $\rho$, letting
$\rho(A,B)$ be the number of copies of
$(\Ks;g_1,g_2)$ with $\cop{g}_1\in A$, $\cop{g}_2\in B$ 
($A,B\sub \C{V}{2}$)
and all other edges in 
$G_0$ (and writing $\rho(xy,\cdot)$ for $ \rho(\{xy\},\cdot) $).
Note that then (cf.\ \eqref{gL*})
\beq{rhoxyG}
\rho(G_0,xy) = 
\tdt_{[g_2,\Ks]}(x,y)/\textrm{aut}(K^*;g_1,g_2) + 
\tdt_{[g_2,\Ks]}(y,x)/\textrm{aut}(K^*;g_1,g_2).
\enq

For $U\sub \C{V}{2}$, set
\beq{muU}
\muu = \mu_{_{|U|}}= |U|\gL_* = (1+O(1/n)) \E \rho(G_0,U) 
\enq
and 
\beq{XU1}
X_U =|\{xy\in G_0: \rho(xy,U)\neq 0\}|.
\enq

We need one simple deterministic observation (which holds for any fixed $\gd>0$):

\begin{prop}\label{Propu}
If $G_0$ satisfies \eqref{generic1}
(with q in place of p), then for any $S$ ($\sub G_0$)
with
$\rho(S,S) = \gO(\ga_S^{1+2\gd} \Psi_{\Ks})$
and any $u\leq |S|$, there is a $U\sub S$ with $|U|=u$,
\beq{gdudan}
\gD_U\leq \lceil 16u/(\ga n)\rceil,
\enq
and 
\beq{rhoU1}
\rho(S,U) =\gO(u\rho(S,S)/|S|) \,\,\,
(=\gO(u\ga_S^{2\gd} \gL_*)
= \gO(\ga_S^{2\gd}\muu)).
\enq
\end{prop}
\nin
(For the first expression in parentheses, recall $\ga_S=2|S|/(n^2q)$ 
and $\gL_*\asymp \gL_{\Ks} =\Psi_{\Ks}/(n^2p)$.)

For the proof of this we use the following observation from
\cite{deWerra,McDiarmid}.  Recall that an {\em equitable} coloring
is one in which the sizes of the color
classes differ by at most one.

\begin{prop}\label{PCol}
For any $m\geq \gD+1$,
the edges of any simple graph of maximum degree at most
$\gD$ can be equitably colored with $m$ colors.
\end{prop}

\begin{proof}[Proof of Proposition~\ref{Propu}]
Let $\{M_j\}$ be the equitable partition of $S$ into $d:=\gD_{S}+1$ matchings 
promised by Proposition~\ref{PCol};
thus $t:= \lfloor |S|/d \rfloor\leq |M_j|\leq t+1$ $\forall $u$j$. 
Assume $\rho(S,M_1)\geq \cdots  \geq \rho(S, M_d)$. 

If $|S|<d$ (i.e.\ $S$ is a star) then we have \eqref{gdudan}-\eqref{rhoU1} with
$U$ any $u$-subset of $S$ maximizing $\rho(S,U)$ (for \eqref{gdudan} note that, since
 $d < 2nq$ (see \eqref{generic1}), $|S|<d$ implies $\ga< 4/n$).  So we may assume $t\neq 0$.

If $u \leq t$, we achieve \eqref{gdudan} and \eqref{rhoU1} by taking $U \subseteq M_1$ of 
size $u$ with maximum $\rho(S,U)$, since then $\gD_U = 1$ and 
$
\rho(S,U) \geq (u/|M_1|)\rho(S, M_1) \geq (u/|M_1|)(t/|S|)\rho(S, S) 
\geq (t/(t+1))(u/|S|) \rho(S,S).
$

In other cases, let
$U' = M_1 \cup \cdots \cup M_{\lceil u/t\rceil}$.
Then $\rho(S,U') \geq u\rho(S,S))/|S|$ and (again using $d < 2nq$),
\[
\gD_{U'} \leq \lceil u/t\rceil \leq 2u/t < 4ud/|S| < 
16u/(\ga n).
\]  
And, since $u \leq |U'| \leq (t+1)\lceil u/t\rceil \leq 4u$,
we achieve \eqref{rhoU1} with
$U\subseteq U'$ of size $u$ maximizing $\rho(S,U)$.

\end{proof}
\nin
\emph{Proof of Lemma~\ref{MP2}.}  
For any $S$ ($\sub G_0$) and $U$,
\begin{align}
    \nonumber \rho(S,U) &\le |S| + \sum_{xy \in S} (\rho(xy,U) - 1)^+\\
    \nonumber &\le |S| + \sum_{xy \in G_0} (\rho(xy,U) - 1)^+\\
    \nonumber &= |S| + \rho(G_0,U) - X_U\\
    \label{t}&=|S| + (\rho(G_0,U) - \mu_{_U}) + (\mu_{_U}-X_U).
\end{align}
On the other hand, for a suitable $\gd$ and  each relevant $\ga$,
we will specify
(see \eqref{defineu}) a
$u=u_\ga$ such that
\beq{uga1}  
\ga n^2q =o( \ga^{2\gd}\mu_u)
\enq
and
\beq{uga2}
\mbox{w.h.p.\
the second and third terms in \eqref{t}
are $o(\ga^{2\gd}\mu_u)$ for each $\ga$ and
$U$ of size $u$.}
\enq
(The asymptotic statements in \eqref{uga1} and \eqref{uga2} use $\ga < 2\gl$ and $\gl \rightarrow 0$.) But \eqref{uga1} and \eqref{uga2} say that w.h.p.\ the bound \eqref{t}
is $o(\ga_S^{2\gd}\mu_{_U})$ for all (relevant) $S$ and $U$;
thus w.h.p.\ \eqref{rhoU1} holds for \emph{no} $S$ and $U$, so there is no
$S$ as in Proposition~\ref{Propu}, which is Lemma~\ref{MP2}. 

\qed

For \eqref{uga2}, the second term in \eqref{t} is handled by 
Lemma~\ref{Erho14}, and the 
third, the heart of the matter, by 
Lemma~\ref{J1}.
(The lemmas do not themselves involve $\gd$; rather, we will choose $\gd$ 
below---see \eqref{delta}---so that \eqref{uga1} holds and the lemmas support
\eqref{uga2}.)

\begin{lem}\label{Erho14}
There is a (fixed) $\eta > 0$ such that if $q= \gO(p^*)$, then
w.h.p.
\[  
\mbox{$\rho(G_0,U) < (1+n^{-\eta})\muu \,\,\,\,\forall \,U\sub \C{V}{2}$.}
\]  
\end{lem}

\nin
\emph{Proof.}   
Since $\rho(G_0,\cdot)$ is additive, it is enough to prove this when
$U$ consists of a single $\{x,y\}\in\C{V}{2}$.  

We use Lemma~\ref{Erho}(a) with $F$ (the graph assumed in the lemma) equal to 
$\Ks$, $W=g_2$ and $Z=V(\Ks)$.
Setting $\gb = 1/d_2(H)$, we have (as required by the lemma)
$q>n^{-\gb}$ (since $q= \gO(p^*)$) and
$\gb < 1/\mmm(W,Z)$ 
(as follows from
\eqref{m2H*} since $a(W,Z)\leq m_2(H^*)$).  So with $\eta$ as in the lemma, we have
\[
\pr(\tdt_{[g_2,\Ks]}(xy) > (1+n^{-\eta})\tdm(W,\Ks)) < n^{-\go(1)},
\]
which with \eqref{gL*} and \eqref{rhoxyG} gives
\[
\pr(\rho(G_0,xy) > (1+n^{-\eta})\gL_*) < n^{-\go(1)}
\]
and Lemma~\ref{Erho14}.

\qed

The bounding of the third term in \eqref{t} is based on Theorem~\ref{JDK},
and is broken into two regimes, depending on $\ga$.
Here we finally need to pay attention to some of our small constants.
Let
\beq{theta}
\mbox{$\theta$ be the smallest of the $\ccc$'s appearing in 
\eqref{q}, \eqref{gammagap}, 
\eqref{L0}-\eqref{L2}, \eqref{gamlarge}, \eqref{K*2},} 
\enq
and let $\sg$ (again fixed) satisfy
\beq{ngd1}
0<\sg < \theta/2.
\enq
We then say $\ga$ is \emph{large} if $\ga> n^{-\sg}$ and \emph{small} otherwise.

The arguments for these two regimes differ only in parameters, and we treat them together.
Set
\beq{gz}
\gz=\gz_\ga= \left\{\begin{array}{ll} 
\ga^{1/2} &\mbox{if $\ga$ is large}\\
n^{-\sg/2}&\mbox{if $\ga$ is small}
\end{array}\right.
\enq
and 
\beq{gb}
\gb=\gb_\ga= \left\{\begin{array}{ll} 
\ga^{1/2} &\mbox{if $\ga$ is large}\\
\ga n^{\sg/2}&\mbox{if $\ga$ is small}
\end{array}\right\} \leq\gz,
\enq
and, in either case,
\beq{defineu}
u = u_\ga= \lceil \gb n^2q/\gL_*\rceil.
\enq

Here we should check that $u\leq |S|$ when $\ga=\ga_S$.
This is trivial (and irrelevant) if $S=\0$, so we may assume $u\geq 1$;
but then $u\ll |S|$ since \eqref{gb}, \eqref{gamlarge}, \eqref{theta}
and \eqref{ngd1} give $\gb/\gL_*\ll \ga$.

In what follows
we are really just interested in the order of magnitude of $u$, so the $\lceil\,\,\rceil$'s 
are irrelevant if $u>1$, in which case
\beq{mu}
\mbox{$u\asymp \gb n^2q/\gL_*~$ and
$~\mu \asymp \gb n^2 q.$}
\enq
The easy case $u=1$ will be handled separately at the one point where
$u=1$ is possible and the main argument uses \eqref{mu} (see the paragraph containing \eqref{uneq1}); but with this exception the following discussion includes $u=1$.

As suggested above, we consider the next assertion our main point.
\begin{lem}\label{J1}
W.h.p.\ for every $\ga< 2\gl$ and $U \sub \C{V}{2}$ of size $u=u_\ga$ satisfying
\eqref{gdudan},
\[
X_U > \mu_u - O(\gz \mu_u).
\]
\end{lem}
\nin
(Note that though we continue to reference \eqref{gdudan}, Lemma~\ref{J1} does
not require $U\sub S$.)

We can now say what \eqref{uga1} and \eqref{uga2} need from $\gd$.
First, 
since $\gb n^2q\leq \mu$, \eqref{uga1} holds provided
\[
\ga \ll \ga^{2\gd}\gb = \left\{\begin{array}{ll} \ga^{1/2 +2\gd}&\mbox{if $\ga$ is large,}\\
\ga^{1 +2\gd}n^{\sg/2}&\mbox{if $\ga$ is small,}
\end{array}\right.
\]
which is true if
$   
\gd < \min\{1/4, \sg/8\} =\sg/8.
$   
(For small $\ga$ this bound, with the trivial $\ga >n^{-2}$, 
gives $\ga^{2\gd}n^{\sg/2}\gg 1$.)
For \eqref{uga2}, in view of Lemmas~\ref{Erho14} and \ref{J1}, we need
$
\ga^{2\gd} \gg \max\{n^{-\eta}, \gz\}
$ 
(with $\eta$ as in Lemma~\ref{Erho14}),
which (again using $\ga > n^{-2}$) is true if 
$   
\gd < \min\{\eta/4, \sg/8, 1/4\}.
$   
So with (say)
\beq{delta}
\gd = \min\{\eta/8 ,\sg/(16) \},
\enq
we have \eqref{uga1} and, with Lemma~\ref{J1} already established, will have
\eqref{uga2} (and Theorem~\ref{thm:Hprecise}) once we prove
Lemma~\ref{J1}.

\section{Proof of Lemma~\ref{J1}}\label{PLJ2}

In what follows it will sometimes be important to have $\gD_U=1$ (i.e.\ $U$ is a matching).
The next observation gives this when needed.

\begin{prop}\label{gDprop}
If $\Ks\neq H^*$ then $u \leq \max\{\ga n/(16), 1\}$ and
$\gD_U=1$.
\end{prop}

\nin
\emph{Proof.}
The first assertion (with \eqref{gdudan})
implies the second (note $u=1$ $\Ra$ $\gD_U=1$), 
and for the first we have
\[
u_\ga =\lceil \gb n^2q/\gL_*\rceil \leq \lceil \gb n^{1-\theta}\rceil
\leq \max\{\ga n/(16), 1\},
\]
where the first inequality uses \eqref{K*2} (and \eqref{theta}) and the 
second is given by
\eqref {ngd1} (and \eqref{gb}). 

\qed

Lemma~\ref{J1} is an application of Theorem~\ref{JDK} in which we use:

\begin{itemize}

\item
$i$ (the first parameter indexing our events) runs over $\C{V}{2}$;

\item
given $i$, $j$ runs over copies of $(\Ks;g_1,g_2)$
in $K_n$ with $g_1(j)=i$ and $g_2(j)\in U$;
and

\item
$\,\,
B_{ij}= \{\mbox{all edges of $j$ other than $g_2(j)$ are in $G_0$}\}.
$

\end{itemize}

\nin
(Note formally different $B_{ij}$'s can be the same event, but these coincidences will be rare
and will not interfere with the remark following Theorem~\ref{JDK}.)

From this point we fix $\ga$, abbreviating $u_\ga=u$ and $\mu_u=\mu$, and 
use $\gD$ for $\gD_U$. Note we are abusively identifying 
the $\mu_u$ of 
\eqref{muU} and the $\mu$ of Theorem~\ref{JDK}; but the two
differ only by a factor $1+O(1/n)$ (see \eqref{gL*}), making this change inconsequential.

Our goals for $\gc$ and $\ov{\gD}$ in Theorem~\ref{JDK} are as follows.
We will show 
\beq{gamma}
\gc = O(\gz\mu)
\enq 
(see \eqref{gz} for $\gz$) 
and take $t=\gc +\max\{\gc, \gz \mu\}$ ($=\Theta(\gz\mu)$).
The theorem then gives
\beq{prXU}
\pr(X_U < \mu - t) 
\,\,(< \exp[-(t-\gc)^2/\ov{\gD}]) ~ <
\exp[-\gO(\gz^2\mu^2)/\ov{\gD}],
\enq
which, to allow a union bound over possible $\ga$'s and $U$'s, 
should be small compared to 
$\left(n^2\C{n^2}{u}\right)^{-1}$
(the $n^2$ could be $\gl n^2q$),
as will be true if 
\beq{largegD}
\ov{\gD} \ll \gz^2 \mu^2/(u \log n)
\enq
(or $\ov{\gD} < \eps \gz^2 \mu^2/(u \log n)$ 
for a suitable fixed $\eps$, but we have room for \eqref{largegD}).

\mn
\emph{Framework.}

We are interested in the contributions of pairs of pairs $((i,j),(k,l))$ to
$\gc$ and $\ogD$.
(So $\elll$, like $j$, is a copy of $(\Ks;g_1,g_2)$.)
An important role will be played by the \emph{type} of $(j,\elll)$, defined to be
\[
a = a(j,\elll) = 
|V(j)\cap V(g_2(\elll))| \,\,\,(\in \{0,1,2\}).
\]

The language here is similar to that in Lemma~\ref{spread}
(see the paragraph following \eqref{ovDsp}).
We specify 
the above pairs and associated costs
(the ``cost'' of $((i,j),(k,l))$ being 
$q^m$, with $m$ the number of edges that $B_{ij}B_{k\ellll}$ forces to be in $G_0$)
in several steps. For most of our discussion the steps 
will be (i)-(v) below; these are always valid but
not always adequate, the situations requiring more care being:
\beq{except}
\mbox{$\gD_U=1$ and $a=1$,}
\enq
for which we replace (i) and (ii) by (i$'$) and (ii$'$);
and---but only when we are bounding $\gc$---
\beq{except2}
\mbox{$i = k$ and $j \sm g_2(j) = \elll \sm g_2(\elll)$,}
\enq
for which the entire sequence (i)-(v) is replaced by
(i$''$) and (ii$''$).
(See the remark following \eqref{mast''} for some explanation of the exceptions.)
Note that our bounds on numbers of choices and costs
may be convenient (slight) overestimates.

\nin
(i)  Choose and ``pay'' for 
$j$ (note $j$ determines $i$);  this contributes (at most) a factor
$
u\gL_*= \,\mu.  
$

\nin
(ii)  Choose $g_2(\elll)$, contributing a factor $u$ if $((i,j),(k,l))$ is of type 0; $O(\gD)$ if type 1;
and $O(1)$ if type~2.

\mn
If \eqref{except} holds we replace (i) and (ii) by:

\nin
(i$'$)  Choose $g_2(j)$ and $g_2(\elll)$---which \eqref{except} forces to be 
disjoint---contributing a factor $u^2$;  

\nin
(ii$'$)  Choose and pay for the rest of $j$, contributing a factor
$O(\gL_* /n)$ (since \eqref{except} says $a = 1$).

\mn
For \eqref{except2} we need just the following two steps, 
each of which gives two bounds, depending on whether $g_2(j)\cap g_2(\ell)$
is empty or a single vertex (of course here 
$g_2(j)=g_2(\ell)$ would imply $j=\ell$):

\nin
(i$''$)  Choose the pairs $g_2(j)$ and $g_2(\elll)$, contributing a factor 
$u^2$ if $g_2(j)\cap g_2(\elll)=\0$ and $2u\gD$ otherwise;  

\nin
(ii$''$)  Choose and pay for the rest of $j$ (a.k.a.\ the rest of $\ell$), contributing a factor
$O(\gL_* /n^2)$ if $g_2(j)\cap g_2(\elll)=\0$ and $O(\gL_* /n)$ otherwise
(note that in the latter case it's only the \emph{edges} that are deleted in \eqref{except2}, 
so we are necessarily in type 2).

\nin
In cases other than \eqref{except2}
(including \eqref{except}), the additional steps below apply.

\mn
(iii)  Choose, in $O(1)$ ways, the \emph{overlap},
\beq{Kjg}
K:= (j\sm g_2(j))\cap (\elll\sm g_2(\elll));
\enq
thus $K$ is a subgraph of $j$ containing at least one edge (since $(i,j)\sim (k,l)$).
(As in \eqref{except2}, it's only edges that are subtracted in \eqref{Kjg}.
Note that we \emph{can} have $g_2(j)\in \elll$ and/or $g_2(\elll)\in j$, so ``overlap''
can be a misnomer.) 

\nin
(iv)  Choose the rest of $\elll$, contributing 
\beq{ivcontrib}
n^{v_{\Ks}-|V(K\cup g_2(\elll))|}  = 
n^{v_{\Ks}-2 - (|V(K\cup g_2(\elll))|-2)} =
n^{v_{\Ks}-2-(v_K-a)}.
\enq

\nin
(v)  Pay for the as yet unpaid for edges of $\elll$, contributing 
\[
q^{e_{K^*}-1-e_K} .
\]

\nin

Collecting (and bounding \eqref{ivcontrib} using worst $K$'s), 
we find
that in non-exceptional cases (those not satisfying \eqref{except}
or \eqref{except2}),
pairs of type $a$ contribute 
\beq{mast}
O(\mu Y_a n^a \gL_*/\min\{\Psi_K\})
=O(\mu^2 (Y_a/u)n^a /\min\{\Psi_K\}),
\enq
where 
\[
Y_a=\left\{\begin{array}{ll}
u&\mbox{if $a=0$,}\\
\gD&\mbox{if $a=1$,}\\
1&\mbox{if $a=2$,}
\end{array}\right.
\]
while pairs with \eqref{except} contribute 
\beq{mast'}
O(u^2\gL_* n^{-1}\gL_* n/\min\{\Psi_K\})
=O(\mu^2  /\min\{\Psi_K\})
\enq
(with the min always over $K$'s allowed in the situation under 
consideration), and those with \eqref{except2} contribute
\beq{mast''}
O(\max\{\mu u/n^2, \mu \gD/n\}) = O(\mu n^{-\theta}).
\enq
Here $u/ n^2 = O(n^{-\theta})$ is trivial if $u=1$, and $\gD/n = O(n^{-\theta})$ is trivial if $\gD=1$. Otherwise they both follow from \eqref{q}, \eqref{ngd1}, \eqref{gdudan}, and $u\asymp \gb n^2q/\gL_*$ (see \eqref{mu}), since, as noted following
\eqref{defineu}, $\gb/\gL_*\ll \ga$.

\mn
\emph{Remark}. 
When \eqref{except} holds the cost in (i$'$)-(ii$'$) multiplies that in (i)-(ii) by 
$u/n $. 
This gain is critical for the bound on $\gc$ when 
$g_2$ is pendant in $\Ks$, say with pendant vertex $w$, and we are bounding
the contribution of pairs $j,\elll$ that agree on $\Ks-w$.
(Proceeding as in (i)-(ii) is wasteful when $u\ll n$ and $a=1$ since for most 
choices of $j$ there is \emph{no} legal $l$.) Similarly, under \eqref{except2}, the $O(1)$ in 
(ii) is wasteful since $i$ is usually the only edge of $U$ in $j$, and the $n^{-\theta}$ gain
in \eqref{mast''}
is again critical.

\begin{proof}[Proof of Lemma~\ref{J1}]
We consider $\gc$ and then $\ogd$, in each case dividing the analysis by type.
Note that if $u=1$ then all relevant $(j,\elll)$'s are of type 2 (since we must have 
$g_2(\elll)=g_2(j)$);  thus we are free to use \eqref{mu}
(which is only valid for $u>1$) in dealing with types 0 and 1.

To be clear, this last bit of our discussion is more or less mechanical and not all that 
enlightening:  we can show \eqref{gamma} and \eqref{largegD} hold, but don't have
much insight into why they should (or, indeed, why our definition of $\Ks$ is
the ``right one'').

\nin
\emph{Bounds for $\gc$}.  Recall we are aiming for
\beq{gcOga}
\gc = O(\gz \mu)
\enq
(see \eqref{gamma})  and that $\gc$ only involves pairs with $i=k$.

\mn 
\emph{Pairs of type $0$: } According to \eqref{mast}, the contribution of these pairs to $\gc$ is
\[
O(\mu^2 /\min\{\Psi_K\}) = O(\mu^2/(n^2q)) = O(\gz \mu),
\]
where: 
$K$ ranges over subgraphs of $\Ks$
containing $g_1$ (we just need $e_K\neq 0$);
the first equality holds because the r.h.s.'s of \eqref{L0}-\eqref{L2} are all at least $n^2q$;
and we use $\mu \asymp \gb n^2q$ and $\gb\leq \gz$ (see \eqref{mu}, \eqref{gb}).

\mn 
\emph{Pairs of type $1$}:  
Suppose first that $\gD>1$, so the relevant bound, as in 
\eqref{mast}, is
\beq{mu2gD}
O(\mu^2 (\gD/u)n /\min\{\Psi_K\}).
\enq
Since the common vertex of $j$ and $g_2(\elll)$ is in $K$, either $K$ has at least two 
components, one of which contains $i$ ($=g_1(j) =g_1(\elll)$),
or $K\cup g_2(\elll)$ is a connection of $\elll$.

In the first case, \eqref{L0}-\eqref{L2} give 
\[
\Psi_K \geq \min\{(n^2q)^2,n^3q\} = n^3q
\]
(where $n^3q$ corresponds to $K$ having exactly two components, one of them $i$
and the other a single vertex); so (again using $\mu \asymp \gb n^2q$ and $\gb\leq \gz$,
with the trivial $\gD\leq u$)
\eqref{mu2gD} is 
\[
O(\mu^2 /(n^2q)) 
= O(\gz\mu).
\]

Now suppose $K \cup g_2(\elll)$ is a connection of $\elll$.
Then \eqref{L1} and \eqref{L2} give 
\[
\Psi_K > \min\{n^2, n^3q^2\},
\]
and, using \eqref{gdudan}
(which, since $\gD>1$, says $\gD=O(u/(\ga n))$) and, again, 
$\mu\asymp \gb n^2q \leq \gz n^2q$, 
we find that \eqref{mu2gD} is 
\[
O(\mu^2/(\ga \min\{n^2, n^3q^2\})) 
=O( (\gb/\ga) \mu \max\{q,(nq)^{-1}\} )  =O(\gz \mu).
\]
Here the r.h.s.\ uses $\gd_1<\theta/2$ (see \eqref{ngd1})
and $\max \{(nq)^{-1}, q\}<n^{-\theta}$ (see \eqref{q} and \eqref{theta}) to say
\[
\ga \,\,(>n^{-\gd_1}) \,\, > \max \{q,(nq)^{-1}\}
\]
if $\ga$ is large, and
\[
(\gb/\ga) \max \{q,(nq)^{-1}\} =n^{\gd_1/2} \max \{q,(nq)^{-1}\} < n^{-\sg/2} =\gz
\]
(see \eqref{gz}) if $\ga$ is small.

If instead $\gD=1$ then the contribution is bounded by \eqref{mast'}, which is $O(\gb\mu)$
since $\mu\asymp \gb n^2q$ and (again using \eqref{L0}-\eqref{L2}) 
$\Psi_K \geq n^2q$.

\mn 
\emph{Pairs of type 2: } Here \eqref{mast} is
\beq{mu2gD'}
O(\mu^2 n^2 /(u\min\{\Psi_K\})).
\enq

If some component, say $M$, of $K\cup g_2(j)$ is a connection of $j$, 
then either $M\subsetneq j$ or \eqref{except2} holds
(since if $M = j$, then $j \sm g_2(j) = K \subseteq \elll \sm g_2(\elll)$ implies
$j \sm g_2(j) = \elll \sm g_2(\elll)$).
If $M \subsetneq j$ then Proposition~\ref{gLprop}(b) 
(with \eqref{theta}) gives 
(with the first inequality a weak consequence of \eqref{L0}-\eqref{L2})
\[
\Psi_K\ge \Psi_M/q = n^2 \gL_M > n^2 n^{\theta} \gL_*,
\]
and the expression in
\eqref{mu2gD'} is (with the second ``$=$'' justified by \eqref{gz}-\eqref{theta})
\[
O(\mu^2 n^2/(u \gL_* n^{2+\theta})  = O(n^{-\theta}\mu) = O(\gz\mu);
\]
and if \eqref{except2} holds, the contribution is bounded by the r.h.s.\ of \eqref{mast''}, 
which is again $O(\gz \mu)$.

Otherwise (i.e.\ if there is no $M$ as above),
$K$ contains $k=i$ and the \emph{ends}
of $g_2(j)$ (but not the edge),
so is disconnected with: a component containing $i$ and neither end of $g_2(j)$,
and either (i) a component containing both ends of $g_2(j)$, or (ii) two
(distinct) components each 
containing one of these ends.  
Again using \eqref{L0}-\eqref{L2} we find that 
the component with $i$ contributes at least a factor $n^2q$ to $\Psi_K$,
and either the component in (i) contributes $n^2$ (since adding $g_2(\elll)$ to that component
produces $M$ with $\Psi_M\geq n^2q$), or each component in (ii) contributes at least 
$\min\{n^2q,n\}=n$.
Thus $\Psi_K\geq n^4q$ and \eqref{mu2gD'} is $O(\mu^2 n^2 /(un^4 q))$.

Here $u=1$ requires special treatment.  If $u>1$ then 
\eqref{mu} gives $\mu\asymp \gb n^2q$ and
\beq{uneq1}
O(\mu^2 n^2 /(un^4 q)) = O(\gb \mu/u)= O(\gz \mu).
\enq
When $u=1$ we lose \eqref{mu} but have $\mu= \gL_*$ and
\[
\mu^2 n^2 /(un^4 q) = \mu \gL_* (n^2q)^{-1}
\leq \mu \gL_{H^*} (n^2q)^{-1} \asymp \mu (n^2q^2)^{-2}\log^2n 
< n^{-2\theta}\mu\log^2n \ll \gz\mu
\]
(with the ``$\leq$'' given by definition of 
$\Ks$, the ``$\asymp$'' by 
\eqref{L2}, and the ``$<$'' by \eqref{q}).

This completes the proof of \eqref{gcOga}.

\mn
\emph{Bounds for $\ov{\gD}$}.
Recall (from \eqref{largegD}) we want
\beq{largegD'}
\ov{\gD} \ll \gz^2  \mu^2/(u \log n).
\enq

\nin 
\emph{Pairs of type $0$: } Since $K$ must contain an edge,
\eqref{mast} (with \eqref{L0}-\eqref{L2}) bounds the contribution here by
\[
O(\mu^2/(n^2 q)) \ll \gz^2 \mu^2/(u \log n),
\]
where for $\gz^2\gg u\log n/(n^2q)$ we use 
\beq{gz2}
\gz^2 \gg n^{-\theta}\log n
\enq
(see the definition of $\gz$ in \eqref{gz}, recalling from \eqref{ngd1} that 
$\gd_1<\theta/2$)
and $u/(n^2q) \asymp \gb/\gL_* < n^{-\theta}$ 
(see \eqref{mu}, \eqref{gamlarge} and \eqref{theta}).

\mn 
\emph{Pairs of type $1$}:  
Here \eqref{mast} and \eqref{mast'} bound the contribution by 
\[
O(\mu^2 (\gD/u)n /\min\{\Psi_K\}) =O(\mu^2/(\ga \min\{\Psi_K\}))
\]
if $\gD>1$ (see \eqref{gdudan}), and by $O(\mu^2/\min\{\Psi_K\})$ if $\gD=1$; 
so, since $\Psi_K\geq n^2q$, it is enough to say (for \eqref{largegD'})
\beq{type1}
\mu^2/( n^2q) \ll 
\left\{\begin{array}{ll}
\ga\gz^2 \mu^2 /(u\log n)&\mbox{if $\gD>1$,}\\
\gz^2 \mu^2 /(u\log n)&\mbox{if $\gD=1$.}
\end{array}\right.
\enq
In fact, the stronger first inequality always holds:  
we have
$u\leq O(\gb n^2q/\gL_*)$ (see \eqref{defineu}, recalling 
from the first paragraph of ``Proof of Lemma~\ref{J1}'' that
``type 1'' implies $u>1$), 
so the first bound in \eqref{type1} is valid if
\[
\frac{\gL_*}{\log n} \gg \frac{\gb}{\ga\gz^2} 
= \left\{\begin{array}{ll}
\ga^{-3/2}&\mbox{if $\ga$ is large,}\\
n^{3\sg/2}&\mbox{if $\ga$ is small,}
\end{array}\right.
\]
which is again given by \eqref{ngd1}
(using $\gL_* > n^\theta$; see \eqref{gamlarge}).

\nin \emph{Pairs of type 2: } 
Here, again using \eqref{mast}, we want
\[
\mu^2 n^2/(u\min\{\Psi_K\}) \ll \gz^2  \mu^2/(u \log n);
\]
i.e., for each eligible $K$,
\beq{PsiKgg}
\gz^2\gg 
n^2\log n/\Psi_K ,
\enq
which in view of \eqref{gz2} will follow from
\beq{n2Psi}
n^2/\Psi_K < n^{-\theta}.
\enq
If $K$ is disconnected, then, since $e_K>0$, \eqref{L0}-\eqref{L2} give 
$\Psi_K>\min\{n^4q^2, n^3q\}=n^3q$, and \eqref{n2Psi} follows from \eqref{q}
(and \eqref{theta}).
If instead $K$ is connected, then $M:=K \cup g_2(\elll)$ is a subgraph of $\elll$ 
with $e_M\geq 2$, and $\Psi_K = \Psi_M/q$.  So either 
$M=\elll$ and $\Psi_K\asymp n^2q^{-1}\log^2n> n^{2+\theta}\log^2n$ 
(see \eqref{L2} and, again \eqref{q}), or
\eqref{L0}-\eqref{L2} give
\[
\Psi_K\geq q^{-1}\min\{n^{2+\theta}q, n^{3+\theta}q^2, n^{2+\theta}\} = n^{2+\theta},
\]
and in each case we have \eqref{n2Psi}.

\end{proof}

\end{document}